\documentclass[11pt,letterpaper,reqno]{amsart}
\usepackage{amsmath,amssymb,amsthm,mathtools}
\usepackage{bm}
\usepackage{graphicx}
\usepackage{xcolor}
\usepackage{eso-pic}
\usepackage[
  colorlinks=true,
  linkcolor=blue,
  citecolor=red,
  urlcolor=blue,
  pdfborder={0 0 0}
]{hyperref}

\allowdisplaybreaks[2]

\newtheorem{theorem}{Theorem}[section]
\newtheorem{lemma}[theorem]{Lemma}
\newtheorem{proposition}[theorem]{Proposition}
\newtheorem{corollary}[theorem]{Corollary}
\theoremstyle{definition}

\theoremstyle{remark}
\newtheorem{remark}[theorem]{Remark}
\numberwithin{equation}{section}

\title[State-law dynamics of McKean-Vlasov stochastic reaction-diffusion equations]
{State-law dynamics of McKean-Vlasov stochastic reaction-diffusion equations on $\mathbb R^n$: pullback random attractors and zero-noise stability}

\author{Guifen Liu}
\address{Guifen Liu: School of Mathematics and Statistics, Southwest University, Chongqing, 400715,  China}
\email{ofenfen@sina.cn}

\author{Yangrong Li}
\address{Yangrong Li: School of Mathematics and Statistics, Southwest University, Chongqing, 400715,  China}
\email{liyr@swu.edu.cn}

\date{September 2, 2026}



\begin{document}

\begin{abstract}
Distribution dependence generally prevents McKean-Vlasov state solution maps from satisfying a cocycle identity, while the relevant Sobolev embedding is noncompact on unbounded domains. For a class of stochastic reaction-diffusion equations with a dissipative polynomial reaction, a one-sided monotone state-law coupling and finite-dimensional additive noise, we combine the deterministic law semiflow with the pathwise state evolution to construct a continuous random dynamical system on the product of the state space and its quadratic Wasserstein law space. An abstract product-space criterion, mixed-energy estimates, local regularity and uniform far-field estimates yield a unique pullback random attractor without global law contraction. Its law projection is the global attractor of the law semiflow, and its state fibers need not be singletons. Under an additional strict contraction condition, the law attractor reduces to the unique invariant law and the corresponding state fiber to a self-consistent random equilibrium. We also establish zero-noise upper semicontinuity and, in the contractive regime, quantitative convergence of the invariant laws and random equilibria with bounds linear in the noise amplitude.
\end{abstract}

\maketitle

\vspace{-0.12in}

\section{Introduction}\label{sec1}

We study the pathwise long-time dynamics of McKean-Vlasov stochastic reaction-diffusion equations on $\mathbb R^n$. Two structural obstacles arise: distribution dependence generally prevents the state solution map from satisfying a cocycle identity on the state space, while the embedding $H^1(\mathbb R^n)\hookrightarrow L^2(\mathbb R^n)$ is not compact. We address both through a state-law product formulation and study the resulting zero-noise stability. Fix integers $n,m\geq1$ and $\epsilon_0>0$. On a probability space $(\Omega,\mathcal F,\mathbb P)$, we consider
\begin{align}\label{Eq1}
	du(t,x)&-\Delta u(t,x)dt+\lambda u(t,x)dt
	+f(x,u(t,x))dt\nonumber\\
	&=G(x,u(t,x),\mathcal L u(t))dt+g(x)dt
	+\epsilon\sum_{j=1}^{m}h_j(x)dW_j(t),
	\qquad t>0,\ x\in\mathbb R^n,
\end{align}
with initial condition
\begin{align}\label{Eq2}
	u(0,x)=u_0(x),\qquad x\in\mathbb R^n.
\end{align}
Here $H=L^2(\mathbb R^n)$, $\lambda>0$, $\epsilon\in[0,\epsilon_0]$, and $u_0$ is a square-integrable $H$-valued initial random variable. The processes $W_1,\ldots,W_m$ are independent two-sided standard Brownian motions, and $h_1,\ldots,h_m$ are deterministic spatial profiles. The notation $\mathcal L u(t)$ denotes the unconditional Borel probability law of $u(t)$ on $H$. The function $f$ is a dissipative reaction of polynomial growth, $G$ couples the state to its law, and $g$ is a deterministic external force.

McKean-Vlasov equations arise in nonlinear diffusion and mean-field limits of weakly interacting particles; see \cite{McKean66pnas,Sznitman91lnm}. Infinite-dimensional formulations include semilinear stochastic evolution equations in Hilbert spaces \cite{Ahmed95spa} and macroscopic limits for McKean-Vlasov-type SPDEs \cite{Kotelenez10ptrf}. Decoupling methods and nonlinear PDEs on Wasserstein space were developed in \cite{Buckdahn17aop,Chaudru22jmpa}. Strong and weak well-posedness for distribution-dependent stochastic evolution equations under locally monotone hypotheses was established in \cite{Hong24aap}, extending the locally monotone SPDE framework of \cite{Liu10jfa}. Well-posedness and large deviations for fractional McKean-Vlasov stochastic reaction-diffusion equations on unbounded domains were obtained in \cite{Chen26amo}. For additive-noise McKean-Vlasov SDEs, \cite{Coghi20aap} develops a pathwise solution theory. Together, these works provide the finite-time foundation for the present analysis.

For pathwise long-time behavior, random dynamical systems and pullback random attractors provide the natural framework; see \cite{Brzez93ptrf,Crauel94,Crauel97jdde,Arnold98}. Variational random-attractor theories for broad classes of SPDEs with additive noise were developed in \cite{Gess11jde,Gess20jde}. On $\mathbb R^n$, however, the embedding $H^1(\mathbb R^n)\hookrightarrow L^2(\mathbb R^n)$ is not compact, so energy bounds alone do not imply asymptotic compactness. Related asymptotic-compactness mechanisms for stochastic flows on unbounded domains were established in \cite{Brzez06tams}; for reaction-diffusion equations, local compactness was combined with uniform far-field control in \cite{Wang99pd,Bates09jde}, and an abstract criterion for noncompact random dynamical systems was given in \cite{Wang12jde}. Variants of this method for Wong-Zakai approximations, nonlinear noise and fractional delay equations appear in \cite{Wang18jde,Wang19jde,Chen25spde}.

Much of the long-time theory for McKean-Vlasov equations is formulated at the level of probability laws. Quantitative examples include Harris-type contraction and phase-transition analysis \cite{Eberle19tams,Carrillo20arma}. For McKean-Vlasov reaction-diffusion equations on unbounded domains, pullback measure attractors, invariant measures and periodic measures on $\mathbb R^n$ were obtained in \cite{Shi24jde}, while uniform measure attractors on unbounded thin domains were studied in \cite{Zeng24}. On the torus, \cite{Chen26avg} proves an averaging principle and upper semicontinuity of pullback law attractors toward the global law attractor of the averaged equation. These results characterize asymptotic laws rather than compact invariant random sets for state trajectories along individual noise paths.

Passing from law asymptotics to pathwise invariant sets requires a cocycle. For \eqref{Eq1}, after a time shift the continuation is not determined by the current state and shifted noise alone; the evolved unconditional law must also be retained. A product-space approach to random attractors for McKean-Vlasov equations on a separable Hilbert state space and its law space was developed in \cite{Cheng25}, with applications to stochastic ordinary differential equations, reaction-diffusion equations and the two-dimensional Navier-Stokes equations. A complementary finite-dimensional construction in \cite{Gess25} couples a nonlinear Fokker-Planck equation with a rough-path state equation and permits merely measurable law solutions and degenerate diffusion coefficients. In the present variational setting on $\mathbb R^n$, the law coordinate is the deterministic semiflow generated by the self-consistent equation, while the state coordinate is the decoupled SPDE driven by the prescribed law trajectory. For the whole-space problem, one must additionally establish compactness in the product topology and control the zero-noise perturbation.

Accordingly, let $P=\mathcal P_2(H)$ be the space of Borel probability measures on $H$ with finite second moment, endowed with the $2$-Wasserstein metric $\mathcal W_2$. For $\mu\in P$, the self-consistent equation defines a deterministic nonlinear law semiflow $\varphi_P^\epsilon$. Along the prescribed law curve $t\mapsto\varphi_P^\epsilon(t)\mu$, the decoupled state equation defines a pathwise map $\varphi_H^\epsilon(t,\omega,x,\mu)$. On a fixed shift-invariant set $\Omega_0\subset\Omega$ of full $\mathbb P$-measure, set
\begin{align*}
	\Phi^\epsilon(t,\omega)(x,\mu)
	=
	\big(
	\varphi_H^\epsilon(t,\omega,x,\mu),
	\varphi_P^\epsilon(t)\mu
	\big)
\end{align*}
on $\mathbb X=H\times P$, equipped with the metric induced by the $H$-norm and $\mathcal W_2$. For a general pair $(x,\mu)$, the first coordinate is driven by the law trajectory generated by $\mu$, but its own law need not coincide with that trajectory. This triangular dependence yields the cocycle identity on the full product space.

The compactness argument treats the two coordinates separately. In the state coordinate, local $H^1$ bounds and compactness on bounded balls are combined with uniform far-field estimates. In the law coordinate, a radius-doubling recursion closes the self-referential tail estimate; the resulting local and tail bounds give tightness, while a uniform moment of order $q>2$ gives uniform integrability of second moments and hence relative compactness in $\mathcal W_2$. Together, these ingredients yield pullback asymptotic compactness in the product metric.

Zero-noise upper semicontinuity requires both finite-time convergence of the product cocycles, uniform on compact subsets of $\mathbb X$, and precompactness of the small-noise attractors. Synchronous coupling controls the law coordinate, while an Ornstein-Uhlenbeck transformation gives the pathwise state comparison. Under strict contraction, the contraction estimate further yields quantitative convergence of the invariant laws and the associated random equilibria.

The main results are as follows. First, the abstract product-space theory gives an existence and uniqueness criterion for pullback random attractors, identifies their law projections and state fibers, and provides an upper semicontinuity theorem based on finite-time convergence and precompactness.

Second, under the structural assumptions \textnormal{(A1)-(A4)} of Section~\ref{sec3}, we prove mixed-energy variational well-posedness, construct a continuous cocycle on $H\times\mathcal P_2(H)$, and obtain a unique pullback random attractor $\mathcal A^\epsilon$. Its law projection is the compact global attractor $\mathcal A_P^\epsilon$ of the law semiflow, and its state fibers satisfy the backward-history invariance and attraction relations. These conclusions require no global law contraction. Under the additional strict contraction assumption \textnormal{(A5)}, $\mathcal A_P^\epsilon=\{\mu_\epsilon^*\}$ for a unique invariant law, and the corresponding state fiber is a measurable random equilibrium $\xi_\epsilon$; hence
\begin{align*}
	\mathcal A^\epsilon(\omega)
	=\big\{(\xi_\epsilon(\omega),\mu_\epsilon^*)\big\},
	\qquad
	\mathbb P\circ\xi_\epsilon^{-1}=\mu_\epsilon^*.
\end{align*}

Third, let $\mathcal A^0$ denote the zero-noise attractor and let $\operatorname{dist}_{\mathbb X}$ be the one-sided Hausdorff semidistance induced by the product metric. Under \textnormal{(A1)-(A4)},
\begin{align*}
	\lim_{\epsilon\to0}
	\operatorname{dist}_{\mathbb X}
	\big(\mathcal A^\epsilon(\omega),\mathcal A^0(\omega)\big)=0,
	\qquad \omega\in\Omega_0.
\end{align*}
If \textnormal{(A5)} also holds, the zero-noise equation has a unique deterministic equilibrium $\bar u\in H$. Writing $\delta_{\bar u}$ for the Dirac measure at $\bar u$, we have $\mathcal A^0(\omega)=\{(\bar u,\delta_{\bar u})\}$ for every $\omega\in\Omega_0$. Moreover, Corollary~\ref{Cor4} gives $\mathcal W_2(\mu_\epsilon^*,\delta_{\bar u})=O(\epsilon)$ and $\big(\mathbb E\|\xi_\epsilon-\bar u\|^2\big)^{1/2}=O(\epsilon)$.

Section~\ref{sec2} develops the abstract product-space theory. Section~\ref{sec3} proves well-posedness and constructs the product cocycle. Section~\ref{sec4} establishes the pullback attractor and its projection-fiber structure. Section~\ref{sec5} treats the zero-noise limit.

\section{State-law product cocycles: attractors, fibers and perturbations}\label{sec2}

In this section, we study pullback random attractors for state-law product cocycles in an abstract setting. The product evolution is triangular: the law coordinate evolves autonomously and drives the pathwise state evolution, whereas the law evolution is independent of the state coordinate. We first give conditions under which this evolution defines a continuous product cocycle and establish a sufficient criterion for the existence of a unique pullback random attractor. We then identify the attractor's law projection, characterize its pathwise fibers, establish sufficient conditions for relative compactness in the Wasserstein and product topologies, and prove upper semicontinuity of the corresponding attractors under perturbations.

Let $(\Omega,\mathcal F,\mathbb P,\{\theta_t\}_{t\in\mathbb R})$ be a metric dynamical system, and let $X$ be a separable Banach space with norm $\|\cdot\|_X$. All pathwise statements are understood on a fixed measurable $\theta_t$-invariant set of full probability, denoted again by $\Omega$. Write $\mathbb R^+=[0,\infty)$, let $\mathcal P(X)$ be the set of Borel probability measures on $X$, and, for $r\geq1$, set
\begin{align*}
	\mathcal P_r(X)
	=
	\Big\{
	\mu\in\mathcal P(X):
	\int_X\|x\|_X^r\,\mu(dx)<\infty
	\Big\}.
\end{align*}
For $\mu,\nu\in\mathcal P_2(X)$, let
\begin{align*}
	\mathcal W_{2,X}^2(\mu,\nu)
	=
	\inf_{\pi\in\Pi(\mu,\nu)}
	\int_{X\times X}\|x-y\|_X^2\,\pi(dx,dy),
\end{align*}
where $\Pi(\mu,\nu)$ is the set of couplings of $\mu$ and $\nu$. Let $P$ be a complete separable metric subspace of $(\mathcal P_2(X),\mathcal W_{2,X})$ and set
\begin{align*}
	\mathbb X=X\times P,
	\qquad
	d_{\mathbb X}\big((x,\mu),(y,\nu)\big)
	=
	\|x-y\|_X+\mathcal W_{2,X}(\mu,\nu).
\end{align*}
Then $\mathbb X$ is a Polish space.

\medskip
\noindent\textbf{The product cocycle. }
Let
\begin{align*}
	\varphi_X:\mathbb R^+\times\Omega\times X\times P\to X,
	\qquad
	\varphi_P:\mathbb R^+\times P\to P.
\end{align*}
The first map is the pathwise state evolution along a prescribed law trajectory, and the second is the deterministic law semiflow. Assume the following conditions.
\begin{itemize}
	\item[(P1)] For every $\omega\in\Omega$, $x\in X$ and $\mu\in P$,
	\begin{align*}
		\varphi_X(0,\omega,x,\mu)=x,
		\qquad
		\varphi_P(0)\mu=\mu.
	\end{align*}
	
	\item[(P2)] For every $s,t\geq0$, $\omega\in\Omega$, $x\in X$ and $\mu\in P$,
	\begin{align*}
		\varphi_X(t+s,\omega,x,\mu)
		&=
		\varphi_X\big(
		t,\theta_s\omega,
		\varphi_X(s,\omega,x,\mu),
		\varphi_P(s)\mu
		\big),\\
		\varphi_P(t+s)\mu
		&=
		\varphi_P(t)\varphi_P(s)\mu.
	\end{align*}
	
	\item[(P3)] For every $t\geq0$ and $\omega\in\Omega$, the map
	\begin{align*}
		(x,\mu)
		\longmapsto
		\big(
		\varphi_X(t,\omega,x,\mu),
		\varphi_P(t)\mu
		\big)
	\end{align*}
	is continuous from $\mathbb X$ to $\mathbb X$.
	
	\item[(P4)] The map $\varphi_X$ is jointly measurable with respect to $\mathcal B(\mathbb R^+)\otimes\mathcal F\otimes\mathcal B(X)\otimes\mathcal B(P)$, and $\varphi_P$ is jointly Borel measurable.
\end{itemize}

\begin{lemma}\label{lemma1}
	Under \textnormal{(P1)-(P4)}, the map
	\begin{align}\label{lemma1-1}
		\Phi(t,\omega)(x,\mu)
		=
		\big(
		\varphi_X(t,\omega,x,\mu),
		\varphi_P(t)\mu
		\big),
		\qquad t\geq0,
	\end{align}
	defines a continuous random dynamical system on $\mathbb X$ over $\{\theta_t\}_{t\in\mathbb R}$.
\end{lemma}

\begin{proof}
	Conditions \textnormal{(P1)} and \textnormal{(P2)} give $\Phi(0,\omega)=\operatorname{id}_{\mathbb X}$ and
	\begin{align*}
		\Phi(t+s,\omega)
		=
		\Phi(t,\theta_s\omega)\circ\Phi(s,\omega),
		\qquad s,t\geq0.
	\end{align*}
	Condition \textnormal{(P3)} gives continuity in the phase variable, and \textnormal{(P4)} gives the required joint measurability.
\end{proof}

\begin{remark}
	The law coordinate in \eqref{lemma1-1} is structural: after a restart, the current state and shifted noise do not determine the continuation unless the evolved law is retained.
\end{remark}

\medskip
\noindent\textbf{Pullback random attractors in the product space. }
Let $\mathfrak D_X$ be an inclusion-closed collection of families $B_X=\{B_X(\omega)\}_{\omega\in\Omega}$ of nonempty subsets of $X$, and let $\mathfrak D_P$ be an inclusion-closed collection of nonempty subsets of $P$. The induced product universe $\mathfrak D$ consists of all families $D=\{D(\omega)\}_{\omega\in\Omega}$ of nonempty subsets of $\mathbb X$ for which there exist $B_X\in\mathfrak D_X$ and $B_P\in\mathfrak D_P$ such that
\begin{align}\label{ProductUniverse}
	D(\omega)\subset B_X(\omega)\times B_P,
	\qquad \omega\in\Omega.
\end{align}
Then $\mathfrak D$ is inclusion closed.

For $z\in\mathbb X$ and nonempty $C\subset\mathbb X$, set
$d_{\mathbb X}(z,C)=\inf_{a\in C}d_{\mathbb X}(z,a)$.
For nonempty $C_1,C_2\subset\mathbb X$, write
$\operatorname{dist}_{\mathbb X}(C_1,C_2)
=\sup_{z\in C_1}d_{\mathbb X}(z,C_2)$.
The analogous point-to-set distances and semidistances in $X$ and $P$ are denoted by $d_X$, $d_P$, $\operatorname{dist}_X$ and $\operatorname{dist}_P$.

A family $K=\{K(\omega)\}_{\omega\in\Omega}$ of nonempty closed subsets of $\mathbb X$ is a random closed set if $\omega\mapsto d_{\mathbb X}(z,K(\omega))$ is measurable for every $z\in\mathbb X$. It is $\mathfrak D$-pullback absorbing for $\Phi$ if, for every $D\in\mathfrak D$ and $\omega\in\Omega$, there exists $T_D(\omega)\geq0$ such that
\begin{align*}
	\Phi(t,\theta_{-t}\omega)D(\theta_{-t}\omega)
	\subset K(\omega),
	\qquad t\geq T_D(\omega).
\end{align*}
The cocycle $\Phi$ is $\mathfrak D$-pullback asymptotically compact if, whenever $D\in\mathfrak D$, $\omega\in\Omega$, $t_n\to\infty$ and $z_n\in D(\theta_{-t_n}\omega)$, the sequence
$\Phi(t_n,\theta_{-t_n}\omega)z_n$
has a convergent subsequence in $\mathbb X$.

A random compact set $\mathcal A=\{\mathcal A(\omega)\}_{\omega\in\Omega}$ is a $\mathfrak D$-pullback random attractor for $\Phi$ if $\mathcal A\in\mathfrak D$ and
\begin{gather*}
	\Phi(t,\omega)\mathcal A(\omega)
	=
	\mathcal A(\theta_t\omega),
	\qquad t\geq0,\\
	\lim_{t\to\infty}
	\operatorname{dist}_{\mathbb X}\big(
	\Phi(t,\theta_{-t}\omega)D(\theta_{-t}\omega),
	\mathcal A(\omega)
	\big)
	=0,
	\qquad D\in\mathfrak D.
\end{gather*}
The condition $\mathcal A\in\mathfrak D$ fixes the universe in which uniqueness is asserted.

\begin{theorem}\label{theorem1}
	Let $\Phi$ be a continuous random dynamical system on $\mathbb X$. Suppose that $\Phi$ has a $\mathfrak D$-pullback absorbing random closed set $K\in\mathfrak D$ and is $\mathfrak D$-pullback asymptotically compact. Then $\Phi$ has a unique $\mathfrak D$-pullback random attractor given by
	\begin{align}\label{theorem1-1}
		\mathcal A(\omega)
		=
		\bigcap_{\tau\geq0}
		\overline{
			\bigcup_{t\geq\tau}
			\Phi(t,\theta_{-t}\omega)K(\theta_{-t}\omega)
		}^{\mathbb X},
		\qquad \omega\in\Omega.
	\end{align}
	Moreover, $\mathcal A(\omega)\subset K(\omega)$ for every $\omega\in\Omega$.
\end{theorem}

\begin{proof}
	The set in \eqref{theorem1-1} is the pullback omega-limit of $K$. Since $K\in\mathfrak D$, pullback asymptotic compactness applied to $K$ gives the usual sequential compactness of its late pullback images. The standard omega-limit argument therefore shows that $\mathcal A(\omega)$ is nonempty and compact. Joint measurability of $\Phi$, continuity in the phase variable and measurability of $K$ imply that $\mathcal A$ is a random compact set; see \cite{Crauel94,Arnold98,Wang12jde}. The cocycle identity and continuity give strict invariance.
	
	It remains to prove pullback attraction. Suppose that attraction fails for some $D\in\mathfrak D$ and $\omega\in\Omega$. Then there exist $\eta>0$, $t_n\to\infty$ and $z_n\in D(\theta_{-t_n}\omega)$ such that, with
	$w_n=\Phi(t_n,\theta_{-t_n}\omega)z_n$,
	one has $d_{\mathbb X}(w_n,\mathcal A(\omega))\geq\eta$. Pullback asymptotic compactness gives, after passing to a subsequence, $w_n\to w$ in $\mathbb X$. Fix $r>0$. For all sufficiently large $n$,
	$t_n-r\geq T_D(\theta_{-r}\omega)$,
	and hence
	\begin{align*}
		\Phi(t_n-r,\theta_{-t_n}\omega)z_n
		\in K(\theta_{-r}\omega).
	\end{align*}
	The cocycle identity gives
	$w_n=\Phi(r,\theta_{-r}\omega)\Phi(t_n-r,\theta_{-t_n}\omega)z_n$,
	so
	\begin{align*}
		w\in
		\overline{\Phi(r,\theta_{-r}\omega)K(\theta_{-r}\omega)}^{\mathbb X}.
	\end{align*}
	For each $\tau\geq0$, choosing $r\geq\tau$ shows that $w$ belongs to the corresponding closed tail in \eqref{theorem1-1}; thus $w\in\mathcal A(\omega)$, a contradiction.
	
	Since $K\in\mathfrak D$, it absorbs itself. Therefore, for every $\omega$, all sufficiently late pullback images of $K$ lie in the closed set $K(\omega)$, and \eqref{theorem1-1} gives $\mathcal A(\omega)\subset K(\omega)$. Inclusion closure yields $\mathcal A\in\mathfrak D$. Finally, if $\mathcal A_1$ and $\mathcal A_2$ are two $\mathfrak D$-pullback random attractors, strict invariance and attraction of $\mathcal A_1$ by $\mathcal A_2$ give $\operatorname{dist}_{\mathbb X}(\mathcal A_1(\omega),\mathcal A_2(\omega))=0$; reversing the roles gives the opposite inclusion. Hence $\mathcal A_1=\mathcal A_2$.
\end{proof}

\medskip
\noindent\textbf{Law projection and pathwise fibers. }
A nonempty compact set $\mathcal A_P\subset P$ is a $\mathfrak D_P$-global attractor for $\varphi_P$ if
\begin{gather*}
	\varphi_P(t)\mathcal A_P=\mathcal A_P,
	\qquad t\geq0,\\
	\lim_{t\to\infty}
	\operatorname{dist}_P\big(
	\varphi_P(t)B_P,\mathcal A_P
	\big)=0,
	\qquad B_P\in\mathfrak D_P.
\end{gather*}
For $\mu\in\mathcal A_P$ and $t\geq0$, define the set of backward law histories by
\begin{align}\label{BackwardLawHistory}
	\Theta_{-t}\mu
	=
	\big\{
	\nu\in\mathcal A_P:
	\varphi_P(t)\nu=\mu
	\big\}.
\end{align}
It is nonempty and compact because $\mathcal A_P$ is compact and invariant and $\varphi_P(t)$ is continuous.

\begin{proposition}\label{PrLawProjection}
	Suppose that $\Phi$ has a $\mathfrak D$-pullback random attractor $\mathcal A$, that $\varphi_P$ has a $\mathfrak D_P$-global attractor $\mathcal A_P\in\mathfrak D_P$, and that $\mathfrak D_X$ is nonempty. Then
	\begin{align}\label{LawProjectionIdentity}
		\Pi_P\mathcal A(\omega)=\mathcal A_P,
		\qquad \omega\in\Omega,
	\end{align}
	where $\Pi_X(x,\mu)=x$ and $\Pi_P(x,\mu)=\mu$.
\end{proposition}

\begin{proof}
	Since $\mathcal A\in\mathfrak D$, there exist $B_X\in\mathfrak D_X$ and $B_P\in\mathfrak D_P$ such that
	\begin{align*}
		\mathcal A(\omega)
		\subset B_X(\omega)\times B_P,
		\qquad \omega\in\Omega.
	\end{align*}
	Strict invariance gives, for every $t\geq0$,
	\begin{align*}
		\Pi_P\mathcal A(\omega)
		=
		\varphi_P(t)\Pi_P\mathcal A(\theta_{-t}\omega)
		\subset
		\varphi_P(t)B_P.
	\end{align*}
	Therefore
	\begin{align*}
		\operatorname{dist}_P\big(
		\Pi_P\mathcal A(\omega),\mathcal A_P
		\big)
		\leq
		\operatorname{dist}_P\big(
		\varphi_P(t)B_P,\mathcal A_P
		\big)
		\longrightarrow0,
	\end{align*}
	and the closedness of $\mathcal A_P$ yields
	$\Pi_P\mathcal A(\omega)\subset\mathcal A_P$.
	
	For the reverse inclusion, fix $\mu\in\mathcal A_P$, choose $B_X^0\in\mathfrak D_X$, and let $t_n\to\infty$. By invariance of $\mathcal A_P$, choose $\nu_n\in\Theta_{-t_n}\mu$ and $x_n\in B_X^0(\theta_{-t_n}\omega)$. Set
	\begin{align*}
		z_n=\Phi(t_n,\theta_{-t_n}\omega)(x_n,\nu_n)
		=
		\big(
		\varphi_X(t_n,\theta_{-t_n}\omega,x_n,\nu_n),
		\mu
		\big).
	\end{align*}
	The family $D^0(\omega')=B_X^0(\omega')\times\mathcal A_P$ belongs to $\mathfrak D$. Pullback attraction gives
	\begin{align*}
		d_{\mathbb X}\big(z_n,\mathcal A(\omega)\big)
		\longrightarrow0.
	\end{align*}
	Choose $a_n\in\mathcal A(\omega)$ such that
	\begin{align*}
		d_{\mathbb X}(z_n,a_n)
		\leq
		d_{\mathbb X}\big(z_n,\mathcal A(\omega)\big)+\frac1n.
	\end{align*}
	A subsequence of $a_n$ converges in the compact set $\mathcal A(\omega)$, and the corresponding $z_n$ converges to the same limit. Its law coordinate is $\mu$, so $\mu\in\Pi_P\mathcal A(\omega)$.
\end{proof}

For $\mu\in\mathcal A_P$, define the nonempty compact fiber
\begin{align}\label{FiberDefinition}
	A(\omega,\mu)
	=
	\big\{
	x\in X:(x,\mu)\in\mathcal A(\omega)
	\big\}.
\end{align}
For $C\subset X$, use the notation
\begin{align*}
	\varphi_X(t,\omega,C,\mu)
	=
	\big\{
	\varphi_X(t,\omega,x,\mu):x\in C
	\big\}.
\end{align*}

\begin{theorem}\label{theorem2}
	Under the assumptions of Proposition~\ref{PrLawProjection}, for every $t\geq0$, $\omega\in\Omega$ and $\mu\in\mathcal A_P$,
	\begin{align}\label{theorem2-1}
		A(\omega,\mu)
		=
		\bigcup_{\nu\in\Theta_{-t}\mu}
		\varphi_X\big(
		t,\theta_{-t}\omega,
		A(\theta_{-t}\omega,\nu),
		\nu
		\big).
	\end{align}
	Moreover, for every $B\in\mathfrak D_X$,
	\begin{align}\label{theorem2-2}
		\lim_{t\to\infty}
		\sup_{\nu\in\Theta_{-t}\mu}
		\operatorname{dist}_X\left(
		\varphi_X\big(
		t,\theta_{-t}\omega,
		B(\theta_{-t}\omega),\nu
		\big),
		A(\omega,\mu)
		\right)
		=0.
	\end{align}
\end{theorem}

\begin{proof}
	Taking the $\mu$-section of
	\begin{align*}
		\mathcal A(\omega)
		=
		\Phi(t,\theta_{-t}\omega)
		\mathcal A(\theta_{-t}\omega)
	\end{align*}
	gives \eqref{theorem2-1}, because an image point has law coordinate $\mu$ exactly when its initial law belongs to $\Theta_{-t}\mu$.
	
	To prove \eqref{theorem2-2}, fix $B\in\mathfrak D_X$. Then
	\begin{align*}
		D_B(\omega')=B(\omega')\times\mathcal A_P
	\end{align*}
	belongs to $\mathfrak D$. If \eqref{theorem2-2} failed, there would exist $\eta>0$, $t_n\to\infty$, $\nu_n\in\Theta_{-t_n}\mu$ and $x_n\in B(\theta_{-t_n}\omega)$ such that, with
	\begin{align*}
		y_n
		=
		\varphi_X(t_n,\theta_{-t_n}\omega,x_n,\nu_n),
	\end{align*}
	one has
	\begin{align*}
		d_X\big(y_n,A(\omega,\mu)\big)\geq\eta.
	\end{align*}
	Since $\varphi_P(t_n)\nu_n=\mu$,
	\begin{align*}
		(y_n,\mu)
		\in
		\Phi(t_n,\theta_{-t_n}\omega)
		D_B(\theta_{-t_n}\omega).
	\end{align*}
	Product pullback attraction gives
	\begin{align*}
		d_{\mathbb X}\big(
		(y_n,\mu),\mathcal A(\omega)
		\big)
		\longrightarrow0.
	\end{align*}
	Choose $a_n=(\widetilde y_n,\lambda_n)\in\mathcal A(\omega)$ such that
	\begin{align*}
		d_{\mathbb X}\big((y_n,\mu),a_n\big)
		\leq
		d_{\mathbb X}\big((y_n,\mu),\mathcal A(\omega)\big)+\frac1n.
	\end{align*}
	After passing to a subsequence, compactness gives
	$a_n\to(\widetilde y,\lambda)\in\mathcal A(\omega)$. The preceding estimate yields $\lambda=\mu$ and $y_n\to\widetilde y$ in $X$. Thus $\widetilde y\in A(\omega,\mu)$, contradicting the lower bound above.
\end{proof}

No measurable selection of backward law histories is required in \eqref{theorem2-1}-\eqref{theorem2-2}.

\begin{corollary}\label{corollary1}
	Under the assumptions of Proposition~\ref{PrLawProjection}, suppose that $\mathcal A_P=\{\mu^*\}$. Then $\varphi_P(t)\mu^*=\mu^*$ for every $t\geq0$, and
	\begin{align}\label{corollary1-1}
		\psi(t,\omega)x
		=
		\varphi_X(t,\omega,x,\mu^*)
	\end{align}
	defines a continuous random dynamical system on $X$. Moreover,
	\begin{align}\label{corollary1-2}
		\mathcal A(\omega)
		=
		A_X(\omega)\times\{\mu^*\},
		\qquad
		A_X(\omega)=\Pi_X\mathcal A(\omega),
	\end{align}
	and $A_X$ is the unique $\mathfrak D_X$-pullback random attractor for $\psi$.
\end{corollary}

\begin{proof}
	The cocycle identity for $\psi$ follows from \textnormal{(P2)} and the invariance of $\mu^*$. Proposition~\ref{PrLawProjection} gives \eqref{corollary1-2}. Since $\Pi_X$ is continuous, $A_X$ is a random compact set; since $\mathcal A\in\mathfrak D$, inclusion closure gives $A_X\in\mathfrak D_X$. Equations \eqref{theorem2-1} and \eqref{theorem2-2}, with $\Theta_{-t}\mu^*=\{\mu^*\}$, give strict invariance and pullback attraction for $\psi$. Uniqueness follows by mutual attraction.
\end{proof}

\begin{corollary}\label{corollary2}
	Under the assumptions of Corollary~\ref{corollary1}, suppose that there exists $\alpha>0$ such that
	\begin{align}\label{corollary2-1}
		\|\psi(t,\omega)x-\psi(t,\omega)y\|_X^2
		\leq e^{-\alpha t}\|x-y\|_X^2,
		\qquad t\geq0,
	\end{align}
	for all $\omega\in\Omega$ and $x,y\in X$. Assume also that every $B_X\in\mathfrak D_X$ has bounded values and satisfies
	\begin{align}\label{corollary2-2}
		\lim_{t\to\infty}
		e^{-\alpha t}
		\operatorname{diam}_X^2 B_X(\theta_{-t}\omega)=0,
		\qquad \omega\in\Omega,
	\end{align}
	where $\operatorname{diam}_X C=\sup_{x,y\in C}\|x-y\|_X$. Then there exists a measurable map $\xi:\Omega\to X$ such that
	\begin{align}\label{corollary2-3}
		A_X(\omega)=\{\xi(\omega)\},
		\qquad
		\mathcal A(\omega)=\{(\xi(\omega),\mu^*)\},
	\end{align}
	and
	\begin{align*}
		\psi(t,\omega)\xi(\omega)
		=
		\xi(\theta_t\omega),
		\qquad t\geq0.
	\end{align*}
\end{corollary}

\begin{proof}
	Since $A_X\in\mathfrak D_X$, strict invariance and \eqref{corollary2-1} give
	\begin{align*}
		\operatorname{diam}_X^2 A_X(\omega)
		=
		\operatorname{diam}_X^2\big(
		\psi(t,\theta_{-t}\omega)A_X(\theta_{-t}\omega)
		\big)
		\leq
		e^{-\alpha t}
		\operatorname{diam}_X^2 A_X(\theta_{-t}\omega)
		\longrightarrow0.
	\end{align*}
	Thus $A_X(\omega)$ is a singleton. Its unique point defines a measurable map $\xi$ because $A_X$ is a random compact set. Formula \eqref{corollary2-3} follows from \eqref{corollary1-2}, and strict invariance gives the random-equilibrium identity.
\end{proof}

\medskip
\noindent\textbf{Compactness of the law coordinate. }
Product asymptotic compactness requires precompactness in both coordinates. For laws, tightness yields only relative compactness in the weak topology; a uniform moment of order $q>2$ gives uniform integrability of the second moments and upgrades it to relative compactness in $\mathcal W_{2,X}$.

\begin{lemma}\label{lemma2}
	Let $Y$ be a separable Banach space compactly embedded into $X$, let $q>2$, and identify $Y$ with its Borel image in $X$. For $R>0$, set
	\begin{align*}
		\mathcal C_R
		=
		\Big\{
		\mu\in\mathcal P_2(X):
		\mu(Y)=1,
		\int_Y\|y\|_Y^q\,\mu(dy)\leq R
		\Big\}.
	\end{align*}
	Then $\mathcal C_R$ is relatively compact in $(\mathcal P_2(X),\mathcal W_{2,X})$. Consequently, every set $\mathcal B\subset Y\times\mathcal P_2(X)$ satisfying
	\begin{align*}
		\sup_{(x,\mu)\in\mathcal B}
		\Big(
		\|x\|_Y^q+
		\int_Y\|y\|_Y^q\,\mu(dy)
		\Big)<\infty,
		\qquad
		\mu(Y)=1
		\quad\text{for every }(x,\mu)\in\mathcal B,
	\end{align*}
	is relatively compact in $X\times\mathcal P_2(X)$.
\end{lemma}

\begin{proof}
	For $M>0$, set
	\begin{align*}
		K_M=\overline{\{y\in Y:\|y\|_Y\leq M\}}^{X}.
	\end{align*}
	The compact embedding makes $K_M$ compact in $X$, and Markov's inequality gives
	\begin{align*}
		\sup_{\mu\in\mathcal C_R}\mu(K_M^c)
		\leq RM^{-q}.
	\end{align*}
	Hence $\mathcal C_R$ is tight. If $\|y\|_X\leq c\|y\|_Y$, then
	\begin{align*}
		\sup_{\mu\in\mathcal C_R}
		\int_{\{\|y\|_X>M\}}\|y\|_X^2\,\mu(dy)
		\leq
		c^qM^{2-q}R
		\longrightarrow0.
	\end{align*}
	Thus the second moments are uniformly integrable. Prokhorov's theorem and the characterization of $\mathcal W_2$ convergence by weak convergence together with convergence of second moments give the first assertion; see \cite{Villani09}. The second follows from the compact embedding for the state coordinate and the first assertion for the law coordinate.
\end{proof}

The preceding criterion does not apply on $\mathbb R^n$, where $H^1(\mathbb R^n)$ is not compactly embedded into $L^2(\mathbb R^n)$. Local regularity must instead be combined with spatial tail control.

\begin{lemma}\label{lemma3}
	Let $X=L^2(\mathbb R^n)$ and $q>2$. For $R>0$, set
	$B_R=\{x\in\mathbb R^n:|x|<R\}$, and interpret
	$\|w\|_{H^1(B_R)}$ as the extended Borel functional on $X$
	obtained by assigning the value $+\infty$ whenever
	$w|_{B_R}\notin H^1(B_R)$. Suppose that
	$\mathcal B\subset X\times\mathcal P_2(X)$ satisfies the following conditions:
	\begin{itemize}
		\item[(i)] for every $R>0$,
		\begin{align*}
			\sup_{(u,\mu)\in\mathcal B}
			\Big[
			\|u\|_{H^1(B_R)}^2
			+
			\int_X\|v\|_{H^1(B_R)}^2\,\mu(dv)
			\Big]<\infty;
		\end{align*}
		
		\item[(ii)]
		\begin{align*}
			\lim_{R\to\infty}
			\sup_{(u,\mu)\in\mathcal B}
			\int_{|x|>R}|u(x)|^2\,dx=0;
		\end{align*}
		
		\item[(iii)]
		\begin{align*}
			\lim_{R\to\infty}
			\sup_{(u,\mu)\in\mathcal B}
			\int_X\int_{|x|>R}|v(x)|^2\,dx\,\mu(dv)=0;
		\end{align*}
		
		\item[(iv)]
		\begin{align*}
			\sup_{(u,\mu)\in\mathcal B}
			\int_X\|v\|_X^q\,\mu(dv)<\infty.
		\end{align*}
	\end{itemize}
	Then $\mathcal B$ is relatively compact in $X\times\mathcal P_2(X)$.
\end{lemma}

\begin{proof}
	This extended functional is lower semicontinuous on $X$, hence Borel. Indeed, if $w_k\to w$ in $X$ and the lower limit is finite, a subsequence realizing it is bounded in $H^1(B_R)$. Weak compactness in $H^1(B_R)$ and uniqueness of the $L^2(B_R)$-limit give $w|_{B_R}\in H^1(B_R)$ and
	\begin{align*}
		\|w\|_{H^1(B_R)}
		\leq
		\liminf_{k\to\infty}
		\|w_k\|_{H^1(B_R)}.
	\end{align*}
	For any sequence of state coordinates, (i) and the Rellich theorem give, after a diagonal extraction, convergence in $L^2(B_j)$ for every integer $j\geq1$; condition (ii) then makes this subsequence Cauchy in $X$.
	
	It remains to prove relative compactness of the law coordinates. Fix $\eta\in(0,1)$. By (iii), choose an increasing sequence $R_j\to\infty$ such that
	\begin{align}\label{lemma3-1}
		\sup_{(u,\mu)\in\mathcal B}
		\int_X\int_{|x|>R_j}|v(x)|^2\,dx\,\mu(dv)
		\leq
		\eta\,2^{-2j-2},
		\qquad j\geq1.
	\end{align}
	By (i),
	\begin{align*}
		C_j
		=
		\sup_{(u,\mu)\in\mathcal B}
		\int_X\|v\|_{H^1(B_{R_j})}^2\,\mu(dv)
		<\infty.
	\end{align*}
	Set
	\begin{align*}
		K_\eta
		=
		\bigcap_{j\geq1}
		\Big\{
		v\in X:
		\begin{aligned}
			\|v\|_{H^1(B_{R_j})}^2
			\leq\frac{2^{j+2}(1+C_j)}{\eta}\quad\text{and}\quad
			\int_{|x|>R_j}|v(x)|^2\,dx
			\leq2^{-j}
		\end{aligned}
		\Big\}.
	\end{align*}
	The set $K_\eta$ is compact in $X$. Indeed, every sequence in $K_\eta$ has, by the local $H^1$ bounds and a diagonal use of the Rellich theorem, a subsequence converging in $L^2(B_{R_j})$ for every $j$. For two elements of this subsequence,
	\begin{align*}
		\|v_k-v_\ell\|_X^2
		&\leq
		\|v_k-v_\ell\|_{L^2(B_{R_j})}^2
		+2\int_{|x|>R_j}|v_k(x)|^2\,dx
		+2\int_{|x|>R_j}|v_\ell(x)|^2\,dx.
	\end{align*}
	Choosing $j$ large and then $k,\ell$ large shows that the subsequence is Cauchy in $X$. The limit belongs to $K_\eta$ by lower semicontinuity of the local $H^1$-functional and continuity of the tail functional. Hence $K_\eta$ is compact.
	
	Markov's inequality and \eqref{lemma3-1} give, uniformly over $(u,\mu)\in\mathcal B$,
	\begin{align*}
		\mu(K_\eta^c)
		&\leq
		\sum_{j\geq1}
		\frac{\eta C_j}{2^{j+2}(1+C_j)}
		+
		\sum_{j\geq1}
		\eta\,2^{-j-2}
		\leq\frac{\eta}{2}.
	\end{align*}
	Thus the law coordinates are tight in $X$. Finally, (iv) gives
	\begin{align*}
		\sup_{(u,\mu)\in\mathcal B}
		\int_{\{\|v\|_X>M\}}\|v\|_X^2\,\mu(dv)
		&\leq
		M^{2-q}
		\sup_{(u,\mu)\in\mathcal B}
		\int_X\|v\|_X^q\,\mu(dv)
		\longrightarrow0.
	\end{align*}
	Prokhorov's theorem and the $\mathcal W_2$ convergence criterion yield relative compactness of the laws in $\mathcal P_2(X)$, and hence of $\mathcal B$ in the product space.
\end{proof}

\medskip
\noindent\textbf{Upper semicontinuity. }
Let $\{\Phi^\epsilon\}_{\epsilon\in[0,\epsilon_0]}$ be product cocycles on $\mathbb X$.

\begin{theorem}\label{theorem3}
	Assume that there exists $\epsilon_1\in(0,\epsilon_0]$ such that the following conditions hold.
	\begin{itemize}
		\item[(U1)] For every $\epsilon\in[0,\epsilon_1]$, $\Phi^\epsilon$ has a $\mathfrak D$-pullback random attractor $\mathcal A^\epsilon$.
		
		\item[(U2)] The family
		\begin{align*}
			\mathcal C(\omega)
			=
			\overline{
				\bigcup_{\epsilon\in[0,\epsilon_1]}
				\mathcal A^\epsilon(\omega)
			}^{\mathbb X}
		\end{align*}
		belongs to $\mathfrak D$, and $\mathcal C(\omega)$ is compact for every $\omega\in\Omega$.
		
		\item[(U3)] For every $t\geq0$, $\omega\in\Omega$ and compact set $C\subset\mathbb X$,
		\begin{align*}
			\sup_{z\in C}
			d_{\mathbb X}\big(
			\Phi^\epsilon(t,\omega)z,
			\Phi^0(t,\omega)z
			\big)
			\longrightarrow0
			\qquad\text{as }\epsilon\to0.
		\end{align*}
	\end{itemize}
	Then, for every $\omega\in\Omega$,
	\begin{align}\label{theorem3-1}
		\operatorname{dist}_{\mathbb X}\big(
		\mathcal A^\epsilon(\omega),
		\mathcal A^0(\omega)
		\big)
		\longrightarrow0
		\qquad\text{as }\epsilon\to0.
	\end{align}
\end{theorem}

\begin{proof}
	Fix $\omega\in\Omega$ and $\eta>0$. Since $\mathcal C\in\mathfrak D$ and $\mathcal A^0$ pullback attracts $\mathcal C$, there exists $T>0$ such that
	\begin{align*}
		\operatorname{dist}_{\mathbb X}\big(
		\Phi^0(T,\theta_{-T}\omega)
		\mathcal C(\theta_{-T}\omega),
		\mathcal A^0(\omega)
		\big)<\frac{\eta}{2}.
	\end{align*}
	The set $\mathcal C(\theta_{-T}\omega)$ is compact, so \textnormal{(U3)} gives, for all sufficiently small $\epsilon>0$,
	\begin{align*}
		\sup_{z\in\mathcal C(\theta_{-T}\omega)}
		d_{\mathbb X}\big(
		\Phi^\epsilon(T,\theta_{-T}\omega)z,
		\Phi^0(T,\theta_{-T}\omega)z
		\big)<\frac{\eta}{2}.
	\end{align*}
	Strict invariance and
	$\mathcal A^\epsilon(\theta_{-T}\omega)\subset\mathcal C(\theta_{-T}\omega)$ now yield
	\begin{align*}
		\operatorname{dist}_{\mathbb X}\big(
		\mathcal A^\epsilon(\omega),
		\mathcal A^0(\omega)
		\big)
		&\leq
		\sup_{z\in\mathcal C(\theta_{-T}\omega)}
		d_{\mathbb X}\big(
		\Phi^\epsilon(T,\theta_{-T}\omega)z,
		\Phi^0(T,\theta_{-T}\omega)z
		\big)\\
		&\qquad+
		\operatorname{dist}_{\mathbb X}\big(
		\Phi^0(T,\theta_{-T}\omega)
		\mathcal C(\theta_{-T}\omega),
		\mathcal A^0(\omega)
		\big)<\eta.
	\end{align*}
	This proves \eqref{theorem3-1}.
\end{proof}

\begin{corollary}\label{corollary3}
	Under \textnormal{(U1)-(U3)},
	\begin{align*}
		\operatorname{dist}_X\big(
		\Pi_X\mathcal A^\epsilon(\omega),
		\Pi_X\mathcal A^0(\omega)
		\big)\longrightarrow0\qquad\text{and}\quad
		\operatorname{dist}_P\big(
		\Pi_P\mathcal A^\epsilon(\omega),
		\Pi_P\mathcal A^0(\omega)
		\big)\longrightarrow0.
	\end{align*}
	If, in addition,
	\begin{align*}
		\mathcal A^\epsilon(\omega)
		=A_X^\epsilon(\omega)\times\{\mu_\epsilon^*\},
		\qquad \epsilon\in[0,\epsilon_1],
	\end{align*}
	then
	\begin{align*}
		\operatorname{dist}_X\big(
		A_X^\epsilon(\omega),A_X^0(\omega)
		\big)\longrightarrow0
		\qquad\text{and}\quad
		\mathcal W_{2,X}(\mu_\epsilon^*,\mu_0^*)\longrightarrow0.
	\end{align*}
\end{corollary}

\begin{proof}
	Both coordinate projections are $1$-Lipschitz with respect to $d_{\mathbb X}$, so the first two assertions follow from \eqref{theorem3-1}. In the factorized case, the projections are $A_X^\epsilon(\omega)$ and $\{\mu_\epsilon^*\}$, respectively, which gives the remaining conclusions.
\end{proof}

\section{Well-posedness and construction of the product cocycle}\label{sec3}

This section establishes finite-time variational well-posedness in the mixed energy space, constructs the deterministic law semiflow and the decoupled pathwise state evolution, and combines them into a product cocycle.

\subsection{Preliminaries and assumptions}

Throughout this section, let $n,m\geq1$ be integers, let $p\geq2$, fix $q>2$, and set $p'=p/(p-1)$. Define
\begin{align*}
	H=L^2(\mathbb R^n),\qquad
	V=H^1(\mathbb R^n),\qquad
	\mathbb V=V\cap L^p(\mathbb R^n),
\end{align*}
with $\|u\|_{\mathbb V}=\|u\|_V+\|u\|_p$. The inner product and norm of $H$ are denoted by $(\cdot,\cdot)$ and $\|\cdot\|$, respectively, and $\|\cdot\|_r$ denotes the norm in $L^r(\mathbb R^n)$. The space $\mathbb V$ is separable and reflexive, and
\begin{align*}
	\mathbb V\subset H\subset\mathbb V^*
\end{align*}
is a dense and continuous Gelfand triple under the usual identification of $H$ with its dual. The dual $\mathbb V^*$ is, up to equivalent norms, the Banach sum
$\mathbb V^*=V^*+L^{p'}(\mathbb R^n)$,
where each summand is identified with its restriction to $\mathbb V$.

For $r\geq2$, set
\begin{align*}
	\mathcal P_r(H)
	=\Big\{\mu\in\mathcal P(H):
	M_r(\mu):=\int_H\|v\|^r\,\mu(dv)<\infty\Big\}.
\end{align*}
We use $P:=\mathcal P_2(H)$ as the law phase space and write $\mathcal W_2:=\mathcal W_{2,H}$; thus $(P,\mathcal W_2)$ is Polish. A set $B\subset\mathcal P_q(H)$ is called $M_q$-bounded if $\sup_{\mu\in B}M_q(\mu)<\infty$.

Let $(\Omega,\mathcal F,\mathbb P,\{\theta_t\}_{t\in\mathbb R})$ be the canonical two-sided Wiener metric dynamical system generated by the $m$-dimensional Brownian motion $W=(W_1,\ldots,W_m)$:
\begin{align*}
	\theta_t\omega(\cdot)=\omega(\cdot+t)-\omega(t),\qquad t\in\mathbb R.
\end{align*}
Let $\{\mathcal F_t\}_{t\in\mathbb R}$ be the completed natural filtration of $W$. For $\epsilon\in[0,\epsilon_0]$, consider
\begin{align}\label{Eq3}
	du^\epsilon(t)-\Delta u^\epsilon(t)dt+\lambda u^\epsilon(t)dt
	+f(\cdot,u^\epsilon(t))dt
	=\mathcal G(u^\epsilon(t),\mathcal L u^\epsilon(t))dt+gdt
	+\epsilon\sum_{j=1}^{m}h_jdW_j(t),
\end{align}
with $u^\epsilon(0)=u_0$, where $\lambda>0$, $u_0\in L^2(\Omega,\mathcal F_0;H)$ and $\mathcal L u^\epsilon(t)$ is the unconditional law of $u^\epsilon(t)$. The Nemytskii operator associated with $G:\mathbb R^n\times\mathbb R\times P\to\mathbb R$ is
\begin{align*}
	\mathcal G(u,\mu)(x)=G(x,u(x),\mu).
\end{align*}
\medskip
\noindent\textbf{(A1)} The function $f:\mathbb R^n\times\mathbb R\to\mathbb R$ is measurable in $x$, and $s\mapsto f(x,s)$ is continuously differentiable for a.e.\ $x\in\mathbb R^n$. There exist constants $\alpha_1>0$, $\alpha_2,\alpha_3,\alpha_4\geq0$ and $c_f>0$, and functions $\psi_1,\psi_2\in L^1(\mathbb R^n)$ and $\psi_3\in L^\infty_{\rm loc}(\mathbb R^n)$, such that
\begin{gather}
	f(x,s)s
	\geq
	\alpha_1|s|^p-\alpha_2|s|^2-\psi_1(x),
	\label{f-1}\\
	|f(x,s)|^{p'}
	\leq
	\alpha_3|s|^p+\psi_2(x),
	\label{f-2}\\
	\big(f(x,s)-f(x,r)\big)(s-r)
	\geq
	-\alpha_4|s-r|^2
	\label{f-3}
\end{gather}
for a.e.\ $x\in\mathbb R^n$ and all $s,r\in\mathbb R$.

For every $s\in\mathbb R$, the map $x\mapsto f(x,s)$ is weakly differentiable. For each coordinate direction $e_k$, choose its absolutely continuous representatives on a common full-measure family of coordinate lines, independently of $s$, so that
\begin{align}\label{f-acl}
	f(x+he_k,s)-f(x,s)
	=\int_0^h\partial_kf(x+re_k,s)\,dr
\end{align}
for every $s\in\mathbb R$ and every segment contained in such a line. The $C^1$-regularity in $s$ makes \eqref{f-3} equivalent to
\begin{align}
	-\alpha_4\leq\partial_sf(x,s)
	\label{f-5}
\end{align}
for every $s\in\mathbb R$ and a.e.\ $x\in\mathbb R^n$. Moreover,
\begin{gather}
	\partial_sf(x,s)
	\leq c_f\big(1+|s|^{p-2}\big),
	\label{f-6}\\
	|\nabla_xf(x,s)|
	\leq\psi_3(x)\big(1+|s|^{p/2}\big)
	\label{f-7}
\end{gather}
for every $s\in\mathbb R$ and a.e.\ $x\in\mathbb R^n$.

\medskip
\noindent\textbf{(A2)} The map $\mathcal G:\mathbb V\times P\to L^{p'}(\mathbb R^n)$ is Borel measurable. For every $\mu\in P$, the map $u\mapsto\mathcal G(u,\mu)$ is hemicontinuous with respect to the $L^{p'}$-$L^p$ duality. For every $u\in\mathbb V$, the map $\mu\mapsto\mathcal G(u,\mu)$ is $\mathcal W_2$-to-$L^{p'}$ continuous on each $M_q$-bounded subset of $\mathcal P_q(H)$.
There exist constants $\beta_G\in[0,2\alpha_1)$ and $\beta_1,\beta_2,\beta_3\geq0$ such that
\begin{align}
	2\langle\mathcal G(u,\mu),u\rangle_{\mathbb V^*,\mathbb V}
	&\leq
	\beta_G\|u\|_p^p+\beta_1\|u\|^2+\beta_2M_2(\mu)+\beta_3,
	\label{G-1}\\
	\|\mathcal G(u,\mu)\|_{p'}^{p'}
	&\leq
	\beta_3\big(1+\|u\|_p^{p}+M_2(\mu)^{p/2}\big)
	\label{G-2}
\end{align}
for all $u\in\mathbb V$ and $\mu\in P$. Moreover, there exist $\beta_4,L_G\geq0$ such that
\begin{align}
	2\langle\mathcal G(u,\mu)-\mathcal G(v,\nu),u-v\rangle_{\mathbb V^*,\mathbb V}
	\leq
	\beta_4\|u-v\|^2+L_G\mathcal W_2^2(\mu,\nu)
	\label{G-3}
\end{align}
for all $u,v\in\mathbb V$ and $\mu,\nu\in P$.

To state the far-field condition, choose $\widehat\rho\in C^1([0,\infty))$ such that $0\leq\widehat\rho\leq1$, $\widehat\rho=0$ on $[0,1]$, $\widehat\rho=1$ on $[2,\infty)$ and $|\widehat\rho'|\leq c$. Set $\rho=\widehat\rho^2$ and $\rho_R(x)=\rho(|x|^2/R^2)$. There exist $c_T\geq0$ and $r_G:[1,\infty)\to[0,\infty)$ with $r_G(R)\to0$ as $R\to\infty$ such that, for every $R\geq1$, $u\in\mathbb V$ and $\mu\in P$,
\begin{align}\label{G-4}
	2\langle\mathcal G(u,\mu),\rho_Ru\rangle_{\mathbb V^*,\mathbb V}
	\leq
	\beta_G\int_{\mathbb R^n}\rho_R|u|^p\,dx
	+\beta_1\int_{\mathbb R^n}\rho_R|u|^2\,dx
	+c_TT_R(\mu)+r_G(R)(1+M_2(\mu)),
\end{align}
where
\begin{align*}
	T_R(\mu)
	=\int_H\int_{|x|>R}|v(x)|^2\,dx\,\mu(dv).
\end{align*}

For every $\mu\in P$ and a.e.\ $x\in\mathbb R^n$, the map $s\mapsto G(x,s,\mu)$ is continuously differentiable, and, for every $(s,\mu)\in\mathbb R\times P$, the map $x\mapsto G(x,s,\mu)$ is weakly differentiable. For each coordinate direction $e_k$, choose its absolutely continuous representatives on a common full-measure family of coordinate lines, independently of $(s,\mu)$, so that
\begin{align}\label{G-acl}
	G(x+he_k,s,\mu)-G(x,s,\mu)
	=\int_0^h\partial_kG(x+re_k,s,\mu)\,dr
\end{align}
for every $(s,\mu)\in\mathbb R\times P$ and every segment contained in such a line. There exist $c_G>0$ and $\psi_4\in L^\infty_{\rm loc}(\mathbb R^n)$ such that, with the same $\beta_4$ as in \eqref{G-3},
\begin{gather}
	2\partial_sG(x,s,\mu)
	\leq\beta_4,
	\qquad
	|\partial_sG(x,s,\mu)|
	\leq c_G\big(1+|s|^{p-2}\big),
	\label{G-6}\\
	|\nabla_xG(x,s,\mu)|
	\leq\psi_4(x)
	\big(1+|s|^{p/2}+M_2(\mu)^{1/2}\big)
	\label{G-7}
\end{gather}
for every $s\in\mathbb R$, $\mu\in P$ and a.e.\ $x\in\mathbb R^n$.

\medskip
\noindent\textbf{(A3)} The deterministic noise profiles and external force satisfy
\begin{align}
	h_j\in H^2(\mathbb R^n)\cap W^{1,\infty}(\mathbb R^n),\qquad
	\int_{\mathbb R^n}(1+|x|^2)
	\big(|h_j|^2+|\nabla h_j|^2+|\Delta h_j|^2\big)\,dx<\infty
	\label{h-1}
\end{align}
for $1\leq j\leq m$, and
\begin{align}
	g\in V\cap L^{p'}(\mathbb R^n),\qquad
	\int_{\mathbb R^n}(1+|x|^2)
	\big(|g|^2+|g|^{p'}+|\nabla g|^2\big)\,dx<\infty.
	\label{g-1}
\end{align}

\medskip
\noindent\textbf{(A4)} The dissipativity gaps used in the global energy estimate, fixed-law stability and far-field localization are positive:
\begin{gather}
	\lambda_0:=2\lambda-2\alpha_2-\beta_1-\beta_2>0,\qquad
	\bar\gamma:=2\lambda-2\alpha_4-\beta_4>0,
	\label{gap}\\
	\lambda_{\rm tail}:=2\lambda-2\alpha_2-\beta_1-c_T>0.
	\label{tail-gap}
\end{gather}

\medskip
\noindent\textbf{(A5)} In addition to \textnormal{(A1)-(A4)}, the net state-law contraction gap is strictly positive:
\begin{align}\label{kappa-def}
	\kappa_0:=\bar\gamma-L_G
	=2\lambda-2\alpha_4-\beta_4-L_G>0.
\end{align}
Condition \textnormal{(A5)} is required only for strict contraction of the law semiflow and its consequences; the product attractor and its zero-noise upper semicontinuity are proved under \textnormal{(A1)-(A4)}.

\begin{remark}[A verifiable class of coefficients]\label{RemCoeff}
	The following class satisfies \textnormal{(A1)-(A4)} but need not satisfy \textnormal{(A5)}. Let $a>0$,
	\begin{align*}
		c\in L^{p'}(\mathbb R^n)\cap W^{1,\infty}_{\rm loc}(\mathbb R^n),
		\qquad
		a_0,\phi\in C_c^\infty(\mathbb R^n)\setminus\{0\},
	\end{align*}
	let $\sigma\in\mathbb R\setminus\{0\}$ and $\Theta\in C_b^1(\mathbb R)$, and define
	\begin{align*}
		f(x,s)=a|s|^{p-2}s+c(x),\qquad
		G(x,s,\mu)=\sigma a_0(x)\Theta\big(m_\phi(\mu)\big),
		\qquad
		m_\phi(\mu)=\int_H(v,\phi)\,\mu(dv).
	\end{align*}
	Since
	\begin{align*}
		m_\phi(\mu)-m_\phi(\nu)
		=\int_{H\times H}(u-v,\phi)\,\pi(du,dv)
	\end{align*}
	for every $\pi\in\Pi(\mu,\nu)$, the Cauchy-Schwarz inequality and minimization over couplings give
	\begin{align}\label{ExampleLip}
		|m_\phi(\mu)-m_\phi(\nu)|
		&\leq
		\|\phi\|
		\inf_{\pi\in\Pi(\mu,\nu)}
		\Big(\int_{H\times H}\|u-v\|^2\,\pi(du,dv)\Big)^{1/2}
		=\|\phi\|\mathcal W_2(\mu,\nu).
	\end{align}
	Since $a_0\in L^{p'}(\mathbb R^n)$ and $\Theta$ is globally Lipschitz, \eqref{ExampleLip} gives global $\mathcal W_2$-to-$L^{p'}$ Lipschitz continuity in the law variable; measurability and hemicontinuity are immediate. The reaction assumptions hold with $\alpha_2=\alpha_4=0$, any $\alpha_1<a$, suitable $\alpha_3,c_f>0$, and
	\begin{align*}
		\psi_1=C|c|^{p'},
		\qquad
		\psi_2=C|c|^{p'},
		\qquad
		\psi_3=|\nabla c|.
	\end{align*}
	For arbitrary $\eta_1,\eta_2>0$, Young's inequality and \eqref{ExampleLip} verify \textnormal{(A2)} with
	\begin{align*}
		\beta_G=\beta_2=c_T=0,
		\qquad
		\beta_1=\eta_1,
		\qquad
		\beta_4=\eta_2,
		\qquad
		L_G=\frac{\sigma^2\|a_0\|^2\|\Theta'\|_\infty^2\|\phi\|^2}{\eta_2},
	\end{align*}
	with $\beta_3$ chosen sufficiently large and
	\begin{align*}
		r_G(R)
		=\frac{\sigma^2\|\Theta\|_\infty^2}{\eta_1}
		\int_{|x|>R}|a_0(x)|^2dx.
	\end{align*}
	Here $r_G(R)\to0$ because $a_0$ is compactly supported. Estimate \eqref{G-4} follows by applying Young's inequality to $2\sigma\Theta(m_\phi(\mu))(a_0,\rho_Ru)$, while \eqref{G-acl} and \eqref{G-6}-\eqref{G-7} hold with $\partial_sG=0$, $c_G=1$ and $\psi_4(x)=|\sigma|\|\Theta\|_\infty|\nabla a_0(x)|$. Choosing $0<\eta_1,\eta_2<2\lambda$ and any $g$ and $h_j$ satisfying \textnormal{(A3)} therefore verifies \textnormal{(A1)-(A4)}.
	
	To show that \textnormal{(A5)} can fail for every admissible pair $(\beta_4,L_G)$, let $\Theta_K(r)=\tanh(Kr)$. The preceding construction verifies \textnormal{(A1)-(A4)} for every $K>0$. Let $\beta_4,L_G\geq0$ be any constants for which \eqref{G-3} holds. For $r,t>0$, set
	\begin{align*}
		\nu=\delta_{0_H},\qquad
		\mu_r=\delta_{r\phi},\qquad
		u=t\,\operatorname{sgn}(\sigma)a_0,\qquad v=0.
	\end{align*}
	Then $m_\phi(\mu_r)=r\|\phi\|^2$ and $\mathcal W_2^2(\mu_r,\nu)=r^2\|\phi\|^2$, so \eqref{G-3} gives
	\begin{align*}
		2|\sigma|t\|a_0\|^2\tanh\big(Kr\|\phi\|^2\big)
		\leq
		\beta_4t^2\|a_0\|^2+L_Gr^2\|\phi\|^2.
	\end{align*}
	Putting $t=\alpha r$, dividing by $r^2$ and letting $r\downarrow0$ yields, for every $\alpha>0$,
	\begin{align*}
		2|\sigma|\alpha K\|a_0\|^2\|\phi\|^2
		\leq
		\beta_4\alpha^2\|a_0\|^2+L_G\|\phi\|^2.
	\end{align*}
	This inequality forces $\beta_4>0$. Minimizing over $\alpha$ and applying the arithmetic-geometric mean inequality yields
	\begin{align*}
		\beta_4L_G
		\geq\sigma^2K^2\|a_0\|^2\|\phi\|^2\quad\text{and}\quad
		\beta_4+L_G
		\geq2|\sigma|K\|a_0\|\|\phi\|.
	\end{align*}
	Since $\alpha_4=0$,
	\begin{align*}
		\kappa_0=2\lambda-\beta_4-L_G
		\leq2\lambda-2|\sigma|K\|a_0\|\|\phi\|.
	\end{align*}
	Consequently, \textnormal{(A5)} is impossible whenever
	\begin{align*}
		K\geq\frac{\lambda}{|\sigma|\|a_0\|\|\phi\|}.
	\end{align*}
	Thus \textnormal{(A1)-(A4)} do not imply \textnormal{(A5)} for any $p\geq2$.
\end{remark}

We first record the probabilistic solution concept. A progressively measurable $H$-valued process $u$ is called a variational solution of \eqref{Eq1}-\eqref{Eq2} on $[0,T]$ if
\begin{align*}
	u\in
	L^2(\Omega;C([0,T];H))
	\cap L^2(\Omega;L^2(0,T;V))
	\cap L^p(\Omega;L^p(0,T;L^p(\mathbb R^n))),
\end{align*}
and, for every $\xi\in\mathbb V$ and $t\in[0,T]$, $\mathbb P$-a.s.,
\begin{align*}
	(u(t),\xi)
	&+\int_0^t
	\big[
	(\nabla u(s),\nabla\xi)
	+\lambda(u(s),\xi)
	+(f(\cdot,u(s)),\xi)
	\big]ds \\
	&=(u_0,\xi)
	+\int_0^t
	\big[
	\langle\mathcal G(u(s),\mathcal L u(s)),\xi\rangle_{\mathbb V^*,\mathbb V}
	+(g,\xi)
	\big]ds
	+\epsilon\sum_{j=1}^{m}(h_j,\xi)W_j(t).
\end{align*}

For the pathwise formulation, let $z_j(\theta_t\omega)$ be the stationary Ornstein-Uhlenbeck process satisfying
\begin{align*}
	dz_j+\lambda z_jdt=dW_j(t),
	\qquad 1\leq j\leq m,
\end{align*}
and set
\begin{align*}
	z(\theta_t\omega)=\sum_{j=1}^{m}h_jz_j(\theta_t\omega).
\end{align*}
Choose a $\{\theta_t\}_{t\in\mathbb R}$-invariant set
$\Omega_0\in\mathcal F$ of full probability such that, for every
$\omega\in\Omega_0$, the paths $t\mapsto z_j(\theta_t\omega)$ are
continuous and tempered. Since
$h_j\in L^2(\mathbb R^n)\cap L^\infty(\mathbb R^n)$, we also have
$h_j\in L^p(\mathbb R^n)$. Hence \eqref{h-1} gives
\begin{align*}
	z(\theta_t\omega)
	\in H^2(\mathbb R^n)\cap W^{1,\infty}(\mathbb R^n)
	\cap L^p(\mathbb R^n)\subset\mathbb V
\end{align*}
for all $t\in\mathbb R$ and $\omega\in\Omega_0$. We henceforth work on $\Omega_0$ and use the same symbols
$\mathcal F$, $\mathbb P$, $\{\mathcal F_t\}_{t\in\mathbb R}$ and
$\theta_t$ for their restrictions.

To separate the state and law evolutions, fix a deterministic curve
$\eta\in C([0,T];P)$. For $x\in H$, the corresponding decoupled equation is
\begin{align}\label{Eq4}
	dU(t)-\Delta U(t)dt+\lambda U(t)dt
	+f(\cdot,U(t))dt
	=\mathcal G(U(t),\eta(t))dt+gdt
	+\epsilon\sum_{j=1}^{m}h_jdW_j(t),
\end{align}
with $U(0)=x$. The Ornstein-Uhlenbeck transformation
\begin{align*}
	U(t)=v(t)+\epsilon z(\theta_t\omega)
\end{align*}
reduces \eqref{Eq4} to the pathwise random equation
\begin{align}\label{Eq4-trans}
	\frac{dv}{dt}
	-\Delta v+\lambda v
	+f\big(\cdot,v+\epsilon z(\theta_t\omega)\big)
	=
	\mathcal G\big(v+\epsilon z(\theta_t\omega),\eta(t)\big)
	+g+\epsilon\Delta z(\theta_t\omega),
\end{align}
with
\begin{align*}
	v(0)=x-\epsilon z(\omega).
\end{align*}

\subsection{The law semiflow and the pathwise state evolution}

Assumptions \textnormal{(A1)-(A4)} remain in force. We use Galerkin approximation and Minty's method \cite{Lions69}, the intersection-space It\^o formula \cite{Gyongy17spde}, and the infinite-dimensional Yamada-Watanabe theorem \cite{Roeckner08cmp}.

Let $\delta_{0_H}$ denote the Dirac measure at $0\in H$, and set
\begin{align*}
	C_{\rm en}
	=2\|\psi_1\|_1+\beta_3
	+\frac{2}{\lambda_0}\|g\|^2
	+\epsilon_0^2\sum_{j=1}^{m}\|h_j\|^2.
\end{align*}
By \eqref{gap}, $\bar\gamma>0$. We begin with equation~\eqref{Eq4} for a prescribed curve $\ell\in C([0,T];P)$.

\begin{lemma}\label{Le1}
	Let $T>0$, $\epsilon\in[0,\epsilon_0]$, $\ell\in C([0,T];P)$ and $X_0\in L^2(\Omega,\mathcal F_0;H)$. Then \eqref{Eq4} with $\eta=\ell$ and $U(0)=X_0$ has a unique variational solution
	\begin{align*}
		U_\ell
		\in
		L^2(\Omega;C([0,T];H))
		\cap L^2(\Omega;L^2(0,T;V))
		\cap L^p(\Omega;L^p(0,T;L^p(\mathbb R^n))).
	\end{align*}
	Moreover, if $U_{\ell_1}$ and $U_{\ell_2}$ correspond to $\ell_1,\ell_2\in C([0,T];P)$ and initial data $X_{0,1},X_{0,2}\in L^2(\Omega,\mathcal F_0;H)$, and are driven by the same Brownian motion, then
	\begin{align}\label{Le1-1}
		\mathbb E\|U_{\ell_1}(t)-U_{\ell_2}(t)\|^2
		\leq
		e^{-\bar\gamma t}\mathbb E\|X_{0,1}-X_{0,2}\|^2
		+L_G\int_0^t e^{-\bar\gamma(t-r)}
		\mathcal W_2^2(\ell_1(r),\ell_2(r))dr
	\end{align}
	for all $t\in[0,T]$.
\end{lemma}

\begin{proof}
	For fixed $\ell\in C([0,T];P)$, write
	\begin{align*}
		A_\ell(t,u)
		=A_{\ell,2}(u)+A_{\ell,p}(t,u),
	\end{align*}
	where
	\begin{align*}
		A_{\ell,2}(u)=\Delta u-\lambda u+g\in V^*\quad\text{and}\quad
		A_{\ell,p}(t,u)=-f(\cdot,u)+\mathcal G(u,\ell(t))
		\in L^{p'}(\mathbb R^n).
	\end{align*}
	Then $A_\ell(t,u)\in\mathbb V^*=V^*+L^{p'}(\mathbb R^n)$ and
	\begin{align*}
		\langle A_\ell(t,u),\xi\rangle_{\mathbb V^*,\mathbb V}
		&=-(\nabla u,\nabla\xi)-\lambda(u,\xi)
		-(f(\cdot,u),\xi)
		+\langle\mathcal G(u,\ell(t)),\xi\rangle_{\mathbb V^*,\mathbb V}
		+(g,\xi).
	\end{align*}
	Since
	\begin{align}\label{Le1-2}
		M_2(\nu)=\mathcal W_2^2(\nu,\delta_{0_H}),
		\qquad \nu\in P,
	\end{align}
	the continuity of $\ell$ gives
	\begin{align}\label{Le1-3}
		M_{\ell,T}
		:=\max_{0\leq t\leq T}M_2(\ell(t))
		=\max_{0\leq t\leq T}
		\mathcal W_2^2(\ell(t),\delta_{0_H})<\infty.
	\end{align}
	The drift is measurable in $t$ and hemicontinuous in $u$.
	Moreover, \eqref{f-2}, \eqref{G-2} and the linear bounds yield
	\begin{align}\label{Le1-4}
		\|A_{\ell,2}(u)\|_{V^*}^{2}
		+\|A_{\ell,p}(t,u)\|_{p'}^{p'}
		\leq
		C_{\ell,T}^{\rm gr}
		\big(1+\|u\|_V^2+\|u\|_p^p\big),
		\qquad 0\leq t\leq T,
	\end{align}
	where $C_{\ell,T}^{\rm gr}>0$ depends only on the structural data,
	$\|g\|$, $\|\psi_2\|_1$ and $M_{\ell,T}$.
	
	For $u\in\mathbb V$, \eqref{f-1} and \eqref{G-1} give
	\begin{align}\label{Le1-5}
		2\langle A_\ell(t,u),u\rangle_{\mathbb V^*,\mathbb V}
		&+\epsilon^2\sum_{j=1}^{m}\|h_j\|^2
		\leq
		-2\|\nabla u\|^2
		-(2\alpha_1-\beta_G)\|u\|_p^p
		-(2\lambda-2\alpha_2-\beta_1)\|u\|^2
		\notag\\
		&\quad+\beta_2M_2(\ell(t))
		+2\|\psi_1\|_1+\beta_3
		+2(g,u)
		+\epsilon^2\sum_{j=1}^{m}\|h_j\|^2.
	\end{align}
	By \eqref{gap},
	$2\lambda-2\alpha_2-\beta_1=\lambda_0+\beta_2$, and
	\begin{align*}
		2(g,u)
		\leq
		\frac{\lambda_0}{2}\|u\|^2
		+\frac{2}{\lambda_0}\|g\|^2.
	\end{align*}
	Hence, using \eqref{Le1-3}, $\epsilon\leq\epsilon_0$ and the definition
	of $C_{\rm en}$,
	\begin{align}\label{Le1-6}
		2\langle A_\ell(t,u),u\rangle_{\mathbb V^*,\mathbb V}
		+\epsilon^2\sum_{j=1}^{m}\|h_j\|^2
		\leq
		-2\|\nabla u\|^2
		-(2\alpha_1-\beta_G)\|u\|_p^p
		-\big(\frac{\lambda_0}{2}+\beta_2\big)\|u\|^2
		+C_{\ell,T},
	\end{align}
	where $C_{\ell,T}=\beta_2M_{\ell,T}+C_{\rm en}$. For
	$u,v\in\mathbb V$, \eqref{f-3} and \eqref{G-3}
	with the same law $\ell(t)$ yield
	\begin{align}\label{Le1-7}
		2\langle A_\ell(t,u)-A_\ell(t,v),u-v\rangle_{\mathbb V^*,\mathbb V}
		&\leq
		-2\|\nabla(u-v)\|^2
		-(2\lambda-2\alpha_4-\beta_4)\|u-v\|^2
		\notag\\
		&=-2\|\nabla(u-v)\|^2-\bar\gamma\|u-v\|^2.
	\end{align}
	
	Choose an $H$-orthonormal basis
	$\{e_k\}_{k\geq1}\subset C_c^\infty(\mathbb R^n)$ whose span is dense
	in $\mathbb V$, set $H_N=\operatorname{span}\{e_1,\ldots,e_N\}$, and let
	$P_N$ be the $H$-orthogonal projection onto $H_N$. Define the
	$H_N$-valued Galerkin approximation by
	\begin{align*}
		d(U_N(t),e_i)
		=\langle A_\ell(t,U_N(t)),e_i\rangle_{\mathbb V^*,\mathbb V}dt
		+\epsilon\sum_{j=1}^{m}(h_j,e_i)dW_j(t),\qquad 1\leq i\leq N,
	\end{align*}
	with $U_N(0)=P_NX_0$. By hemicontinuity, the local bound
	\eqref{Le1-4}, the one-sided monotonicity \eqref{Le1-7} and finite
	dimensionality, the projected drift is a Carath\'eodory field. Estimate
	\eqref{Le1-7} also gives pathwise uniqueness. Standard finite-dimensional
	existence yields a maximal adapted solution; applying It\^o's formula up
	to its exit times and using \eqref{Le1-6} excludes finite-time explosion.
	Thus the Galerkin equation has a unique global adapted solution. It\^o's formula,
	\eqref{Le1-6} and the BDG inequality give
	\begin{align}\label{Le1-8}
		\sup_{N\geq1}\Big[
		\mathbb E\sup_{0\leq t\leq T}\|U_N(t)\|^2
		+\mathbb E\int_0^T
		\big(\|U_N(t)\|_V^2+\|U_N(t)\|_p^p\big)dt
		\Big]
		\leq C_{T,\ell}\big(1+\mathbb E\|X_0\|^2\big),
	\end{align}
	where $C_{T,\ell}>0$ is independent of $N$, $\epsilon\in[0,\epsilon_0]$
	and $X_0$. By \eqref{Le1-4}, after passing to a subsequence,
	\begin{gather*}
		U_N\rightharpoonup U
		\quad\text{in }L^2(\Omega\times(0,T);V)
		\text{ and }L^p(\Omega\times(0,T);L^p),\\
		A_{\ell,p}(\cdot,U_N)\rightharpoonup\chi
		\quad\text{in }L^{p'}(\Omega\times(0,T);L^{p'}).
	\end{gather*}
	The progressively measurable subspaces of these Bochner spaces are weakly
	closed, so the limits may be chosen progressively measurable. Since
	$A_{\ell,2}:V\to V^*$ is affine and continuous,
	$A_{\ell,2}(U_N)\rightharpoonup A_{\ell,2}(U)$ in
	$L^2(\Omega\times(0,T);V^*)$. Testing the coordinate equations against
	$C_c^\infty(0,T)$, passing to the limit and using density gives
	\begin{align*}
		U(t)
		={}&X_0+\int_0^t\big(A_{\ell,2}(U(r))+\chi(r)\big)dr
		+\epsilon\sum_{j=1}^{m}h_jW_j(t)
	\end{align*}
	in $\mathbb V^*$ for $dt\times\mathbb P$-a.e.\ $(t,\omega)$. The
	intersection-space It\^o formula \cite{Gyongy17spde} yields an
	$H$-c\`adl\`ag version and the energy identity. The right-hand side is
	continuous in $\mathbb V^*$ and $H\hookrightarrow\mathbb V^*$ is
	injective, so this version is $H$-continuous. Thus, on a common event of
	probability one,
	\begin{align}\label{Le1-9}
		U(t)
		=X_0+\int_0^t\big(A_{\ell,2}(U(r))+\chi(r)\big)dr
		+\epsilon\sum_{j=1}^{m}h_jW_j(t)
	\end{align}
	holds in $\mathbb V^*$ for every $t\in[0,T]$.
	
	Choose a subsequence that realizes the lower limit of the integrated
	drift pairings. Since $\{U_N(T)\}$ is bounded in $L^2(\Omega;H)$,
	a further subsequence converges weakly there, say
	$U_N(T)\rightharpoonup\widetilde U_T$. The coordinate equations at
	$T$, together with \eqref{Le1-9}, identify $\widetilde U_T=U(T)$.
	Comparing the Galerkin
	energy identities with that of
	\eqref{Le1-9}, and using $P_NX_0\to X_0$ in $L^2(\Omega;H)$,
	$P_Nh_j\to h_j$ in $H$ and weak lower semicontinuity, gives
	\begin{align}\label{Le1-10}
		\liminf_{N\to\infty}
		\mathbb E\int_0^T
		\langle A_\ell(t,U_N(t)),U_N(t)\rangle_{\mathbb V^*,\mathbb V}dt
		\geq
		\mathbb E\int_0^T
		\langle A_{\ell,2}(U(t))+\chi(t),U(t)\rangle_{\mathbb V^*,\mathbb V}dt.
	\end{align}
	For every progressively measurable
	\begin{align*}
		w\in L^2(\Omega\times(0,T);V)
		\cap L^p(\Omega\times(0,T);L^p(\mathbb R^n)),
	\end{align*}
	monotonicity \eqref{Le1-7} gives
	\begin{align*}
		\mathbb E\int_0^T
		\langle A_\ell(t,U_N)-A_\ell(t,w),U_N-w\rangle_{\mathbb V^*,\mathbb V}dt
		\leq0.
	\end{align*}
	Using \eqref{Le1-4}, the weak convergences and \eqref{Le1-10}, we obtain
	\begin{align*}
		\mathbb E\int_0^T
		\langle B-A_\ell(t,w),U-w\rangle_{\mathbb V^*,\mathbb V}dt
		\leq0,
		\qquad
		B=A_{\ell,2}(U)+\chi.
	\end{align*}
	Let $\phi$ be a bounded progressively measurable $\mathbb V$-valued
	simple process. Since $U\pm\delta\phi$ remain in the mixed energy
	space, we may take successively $w=U-\delta\phi$ and
	$w=U+\delta\phi$. Dividing by $\delta>0$ and letting
	$\delta\downarrow0$, hemicontinuity and \eqref{Le1-4} give, respectively,
	\begin{align*}
		\mathbb E\int_0^T
		\langle B-A_\ell(t,U),\phi\rangle_{\mathbb V^*,\mathbb V}dt
		\leq0
		\quad\text{and}\quad
		\mathbb E\int_0^T
		\langle B-A_\ell(t,U),\phi\rangle_{\mathbb V^*,\mathbb V}dt
		\geq0.
	\end{align*}
	Density of such simple processes and separability of $\mathbb V$ imply
	$B=A_\ell(\cdot,U)$ a.e.\ on $\Omega\times(0,T)$. Hence
	$\chi=A_{\ell,p}(\cdot,U)$, and \eqref{Le1-9} is the
	variational form of \eqref{Eq4}. Its energy identity, \eqref{Le1-6} and
	the BDG and Young inequalities give the asserted
	$L^2(\Omega;C([0,T];H))$ integrability. Estimate \eqref{Le1-7} and the
	energy identity give pathwise uniqueness, so the limit is independent of
	the extracted subsequence.
	
	For two prescribed law curves, the stochastic terms cancel. The
	difference $U_{\ell_1}-U_{\ell_2}$ belongs to
	$L^2(0,T;V)\cap L^p(0,T;L^p)$ and has derivative in
	$L^2(0,T;V^*)+L^{p'}(0,T;L^{p'})$ a.s. Its energy identity,
	\eqref{f-3} and \eqref{G-3} yield, for a.e.\ $t\in(0,T)$,
	\begin{align*}
		\frac{d}{dt}\|U_{\ell_1}(t)-U_{\ell_2}(t)\|^2
		+2\|\nabla (U_{\ell_1}(t)-U_{\ell_2}(t))\|^2
		+\bar\gamma\|U_{\ell_1}(t)-U_{\ell_2}(t)\|^2
		\leq
		L_G\mathcal W_2^2(\ell_1(t),\ell_2(t)).
	\end{align*}
	Multiplying by $e^{\bar\gamma t}$, integrating over $[0,t]$ and taking
	expectation proves \eqref{Le1-1}.
\end{proof}

We determine the prescribed law curve by a fixed point.

\begin{lemma}\label{Le2}
	For every $\epsilon\in[0,\epsilon_0]$, $T>0$ and $u_0\in L^2(\Omega,\mathcal F_0;H)$, equation~\eqref{Eq1} has a unique variational solution $u$ on $[0,T]$. The solutions obtained on different finite intervals are consistent and therefore define a global solution on $[0,\infty)$. Moreover, $t\mapsto\mathcal L u(t)$ belongs to $C([0,\infty);P)$ and
	\begin{align}\label{Le2-1}
		&\mathbb E\sup_{0\leq r\leq T}\|u(r)\|^2
		+\mathbb E\int_0^T
		\big(\|u(r)\|_V^2+\|u(r)\|_p^p\big)dr
		\leq
		C_T\big(1+\mathbb E\|u_0\|^2\big),
	\end{align}
	where $C_T>0$ depends only on $T$ and the structural data and is independent of $\epsilon\in[0,\epsilon_0]$ and $u_0$. If $u_1,u_2$ are two solutions driven by the same Brownian motion with initial data $u_{0,1},u_{0,2}\in L^2(\Omega,\mathcal F_0;H)$, then
	\begin{align}\label{Le2-2}
		\mathbb E\|u_1(t)-u_2(t)\|^2
		\leq
		e^{(L_G-\bar\gamma)t}\mathbb E\|u_{0,1}-u_{0,2}\|^2,
		\qquad t\geq0.
	\end{align}
	Consequently, the law of the solution is determined by the initial law $\mathcal L u_0$.
\end{lemma}

\begin{proof}
	Set $\mu=\mathcal L u_0$ and
	\begin{align*}
		\mathcal E_T(\mu)
		=\{\ell\in C([0,T];P):\ell(0)=\mu\}.
	\end{align*}
	Choose $a>\max\{0,L_G-\bar\gamma\}$ and endow this space with the Bielecki metric
	\begin{align*}
		d_{T,a}(\ell_1,\ell_2)
		=\sup_{0\leq t\leq T}
		e^{-at/2}\mathcal W_2(\ell_1(t),\ell_2(t)).
	\end{align*}
	This metric is equivalent to the uniform metric on $[0,T]$; hence $(\mathcal E_T(\mu),d_{T,a})$ is complete.
	
	For $\ell\in\mathcal E_T(\mu)$, let $U_\ell$ be the solution from Lemma~\ref{Le1} with initial datum $u_0$, and define
	\begin{align}\label{Le2-3}
		(\Gamma_\mu\ell)(t)
		=\mathcal L U_\ell(t),
		\qquad 0\leq t\leq T.
	\end{align}
	Given $\mu\in P$, realize an $\mathcal F_0$-measurable initial variable with law $\mu$ on a product extension, independently of the Brownian motion. Lemma~\ref{Le1} gives a variational solution and pathwise uniqueness on this extension. The variational Yamada-Watanabe theorem \cite{Roeckner08cmp} then yields uniqueness in law, so the law of the decoupled solution depends only on $\mu$ and $\ell$.
	The path regularity in Lemma~\ref{Le1} also gives
	\begin{align*}
		\mathcal W_2^2\big((\Gamma_\mu\ell)(t),
		(\Gamma_\mu\ell)(s)\big)
		\leq
		\mathbb E\|U_\ell(t)-U_\ell(s)\|^2
		\longrightarrow0
	\end{align*}
	as $t\to s$ by dominated convergence. Thus $\Gamma_\mu$
	maps $\mathcal E_T(\mu)$ into itself.
	
	For $\ell_1,\ell_2\in\mathcal E_T(\mu)$, couple the corresponding solutions by using the same initial variable and Brownian motion. Estimate \eqref{Le1-1} gives
	\begin{align*}
		e^{-at}\mathcal W_2^2\big((\Gamma_\mu\ell_1)(t),
		(\Gamma_\mu\ell_2)(t)\big)
		\leq
		L_Gd_{T,a}^2(\ell_1,\ell_2)
		\int_0^t e^{-(a+\bar\gamma)(t-r)}dr
		\leq
		\frac{L_G}{a+\bar\gamma}
		d_{T,a}^2(\ell_1,\ell_2).
	\end{align*}
	Consequently,
	\begin{align}\label{Le2-4}
		d_{T,a}^2(\Gamma_\mu\ell_1,
		\Gamma_\mu\ell_2)
		\leq
		\frac{L_G}{a+\bar\gamma}
		d_{T,a}^2(\ell_1,\ell_2).
	\end{align}
	By the choice of $a$, the coefficient in \eqref{Le2-4} is strictly smaller than one. Banach's fixed-point theorem yields a unique
	$\ell^{\epsilon,T}_\mu\in\mathcal E_T(\mu)$ such that
	\begin{align*}
		\Gamma_\mu\ell^{\epsilon,T}_\mu
		=\ell^{\epsilon,T}_\mu.
	\end{align*}
	The process $u=U_{\ell^{\epsilon,T}_\mu}$ then solves \eqref{Eq1} on $[0,T]$, and
	\begin{align*}
		\mathcal L u(t)=\ell^{\epsilon,T}_\mu(t),
		\qquad 0\leq t\leq T.
	\end{align*}
	If $0<T_1<T_2$, the restriction of $\ell^{\epsilon,T_2}_\mu$ to $[0,T_1]$ is the unique fixed point of $\Gamma_\mu$ on $\mathcal E_{T_1}(\mu)$; hence
	\begin{align*}
		\ell^{\epsilon,T_2}_\mu|_{[0,T_1]}
		=\ell^{\epsilon,T_1}_\mu.
	\end{align*}
	Pathwise uniqueness gives the same consistency for the decoupled solutions. Thus the finite-interval solutions define a global solution $u$ and law curve $\ell^\epsilon_\mu$.
	
	Conversely, the law curve of any variational solution of \eqref{Eq1} with initial law $\mu$ is a fixed point of \eqref{Le2-3}: after that curve is prescribed, the process solves the decoupled equation from Lemma~\ref{Le1}. The fixed point is unique, and the decoupled equation is unique in law. Thus the McKean-Vlasov equation has uniqueness in law.
	
	For the energy estimate, It\^o's formula, \eqref{f-1} and \eqref{G-1} give, for $0\leq t\leq T$,
	\begin{align}\label{Le2-5}
		&\|u(t)\|^2
		+2\int_0^t\|\nabla u(r)\|^2dr
		+(2\alpha_1-\beta_G)
		\int_0^t\|u(r)\|_p^pdr
		+(2\lambda-2\alpha_2-\beta_1)
		\int_0^t\|u(r)\|^2dr
		\notag\\
		&\leq
		\|u_0\|^2
		+\big(2\|\psi_1\|_1+\beta_3
		+\epsilon_0^2\sum_{j=1}^{m}\|h_j\|^2\big)t
		+\beta_2\int_0^tM_2(\ell^\epsilon_\mu(r))dr
		\notag\\
		&\qquad
		+2\int_0^t(g,u(r))dr
		+2\epsilon\sum_{j=1}^{m}
		\int_0^t(u(r),h_j)dW_j(r).
	\end{align}
	Using
	\begin{align*}
		2(g,u)
		\leq
		\frac{\lambda_0}{2}\|u\|^2
		+\frac{2}{\lambda_0}\|g\|^2
	\end{align*}
	and $2\lambda-2\alpha_2-\beta_1=\lambda_0+\beta_2$, inequality \eqref{Le2-5} becomes
	\begin{align}\label{Le2-6}
		&\|u(t)\|^2
		+2\int_0^t\|\nabla u(r)\|^2dr
		+(2\alpha_1-\beta_G)
		\int_0^t\|u(r)\|_p^pdr
		+\big(\frac{\lambda_0}{2}+\beta_2\big)
		\int_0^t\|u(r)\|^2dr
		\notag\\
		&\leq
		\|u_0\|^2+C_{\rm en}t
		+\beta_2\int_0^tM_2(\ell^\epsilon_\mu(r))dr
		+2\epsilon\sum_{j=1}^{m}
		\int_0^t(u(r),h_j)dW_j(r).
	\end{align}
	Since $M_2(\ell^\epsilon_\mu(r))=\mathbb E\|u(r)\|^2$, taking expectation in \eqref{Le2-6} cancels the two terms with coefficient $\beta_2$ and gives
	\begin{align}
		\mathbb E\|u(t)\|^2
		+2\int_0^t\!\mathbb E\|\nabla u(r)\|^2dr
		+(2\alpha_1-\beta_G)
		\int_0^t\!\mathbb E\|u(r)\|_p^pdr
		+\frac{\lambda_0}{2}
		\int_0^t\!\mathbb E\|u(r)\|^2dr
		\leq
		\mathbb E\|u_0\|^2+C_{\rm en}t.
		\label{Le2-7}
	\end{align}
	The same calculation before time integration gives, for a.e.\ $t>0$,
	\begin{align*}
		\frac{d}{dt}\mathbb E\|u(t)\|^2
		+\frac{\lambda_0}{2}\mathbb E\|u(t)\|^2
		\leq C_{\rm en}.
	\end{align*}
	Gronwall's inequality therefore gives
	\begin{align}\label{Le2-8}
		\mathbb E\|u(t)\|^2
		\leq
		e^{-\lambda_0t/2}\mathbb E\|u_0\|^2
		+\frac{2C_{\rm en}}{\lambda_0},
		\qquad t\geq0.
	\end{align}
	
	To control the time supremum, take the supremum in \eqref{Le2-6}. Estimate \eqref{Le2-8} first gives
	\begin{align*}
		\beta_2\int_0^T M_2(\ell^\epsilon_\mu(r))dr
		\leq C_T\big(1+\mathbb E\|u_0\|^2\big).
	\end{align*}
	The BDG inequality gives
	\begin{align*}
		&\mathbb E\sup_{0\leq s\leq T}
		\Big|
		2\epsilon\sum_{j=1}^{m}
		\int_0^s(u(r),h_j)dW_j(r)
		\Big|
		\leq
		C\epsilon_0
		\mathbb E\Big(
		\int_0^T\sum_{j=1}^{m}
		|(u(r),h_j)|^2dr
		\Big)^{1/2}
		\\
		&\leq
		C\epsilon_0
		\Big(\sum_{j=1}^{m}\|h_j\|^2\Big)^{1/2}
		\Big(\int_0^T\mathbb E\|u(r)\|^2dr\Big)^{1/2}
		\leq
		C_T\big(1+\mathbb E\|u_0\|^2\big).
	\end{align*}
	Taking the supremum in \eqref{Le2-6}, using this BDG estimate together
	with \eqref{Le2-7}, proves \eqref{Le2-1}.
	
	Finally, let $u_1,u_2$ be synchronously driven solutions. The noise cancels. Using \eqref{f-3}, \eqref{G-3} and the coupling inequality
	\begin{align*}
		\mathcal W_2^2(\mathcal L u_1(t),
		\mathcal L u_2(t))
		\leq\mathbb E\|u_1(t)-u_2(t)\|^2,
	\end{align*}
	we obtain
	\begin{align*}
		\frac{d}{dt}\mathbb E\|u_1(t)-u_2(t)\|^2
		+2\mathbb E\|\nabla \big(u_1(t)-u_2(t)\big)\|^2
		+\bar\gamma\mathbb E\|u_1(t)-u_2(t)\|^2
		\leq L_G\mathbb E\|u_1(t)-u_2(t)\|^2.
	\end{align*}
	Gronwall's inequality proves \eqref{Le2-2}. Together with the preceding uniqueness-in-law argument, this shows that the solution law depends only on $\mathcal L u_0$.
\end{proof}

To realize an arbitrary $\mu\in P$, we allow a product extension carrying an $\mathcal F_0$-measurable $H$-valued initial variable independent of the Brownian motion. For any such variable $u_0$ with law $\mu$, define
\begin{align}
	\varphi^\epsilon_P(t)\mu
	=\mathcal L u(t),
	\qquad t\geq0.
\end{align}
Lemma~\ref{Le2} makes this definition independent of the chosen realization and extension.

\begin{lemma}\label{Le3}
	For every $\epsilon\in[0,\epsilon_0]$, $\varphi^\epsilon_P$ is a continuous semiflow on $P$. More precisely,
	\begin{align}\label{Le3-1}
		\varphi^\epsilon_P(0)\mu=\mu,
		\qquad
		\varphi^\epsilon_P(t+s)\mu
		=\varphi^\epsilon_P(t)\varphi^\epsilon_P(s)\mu,
		\qquad s,t\geq0,
	\end{align}
	and
	\begin{align}\label{Le3-2}
		\mathcal W_2^2\big(\varphi^\epsilon_P(t)\mu,
		\varphi^\epsilon_P(t)\nu\big)
		\leq
		e^{(L_G-\bar\gamma)t}\mathcal W_2^2(\mu,\nu),
		\qquad t\geq0,
	\end{align}
	for all $\mu,\nu\in P$. In addition, $(t,\mu)\mapsto\varphi^\epsilon_P(t)\mu$ is continuous from $[0,\infty)\times P$ to $P$.
\end{lemma}

\begin{proof}
	Let $\mu,\nu\in P$ and realize an arbitrary coupling $(u_{0,1},u_{0,2})$ of $(\mu,\nu)$ independently of a common Brownian motion. The corresponding synchronously driven solutions satisfy \eqref{Le2-2}; therefore
	\begin{align*}
		\mathcal W_2^2\big(\varphi^\epsilon_P(t)\mu,
		\varphi^\epsilon_P(t)\nu\big)
		\leq
		e^{(L_G-\bar\gamma)t}\mathbb E\|u_{0,1}-u_{0,2}\|^2.
	\end{align*}
	Taking the infimum over all couplings proves \eqref{Le3-2}.
	
	To verify the semiflow identity, let $u$ start with law $\mu$, fix $s\geq0$, and set
	\begin{align*}
		\nu=\varphi^\epsilon_P(s)\mu,
		\qquad
		\widehat u(r)=u(s+r),
		\qquad
		\widehat W(r)=W(s+r)-W(s).
	\end{align*}
	The process $\widehat W$ is a Brownian motion independent of $\mathcal F_s$, and $\widehat u$ solves \eqref{Eq1} with initial law $\nu$. Its law curve is $r\mapsto\varphi^\epsilon_P(s+r)\mu$. Uniqueness in law from Lemma~\ref{Le2} therefore gives
	\begin{align*}
		\varphi^\epsilon_P(s+r)\mu
		=\varphi^\epsilon_P(r)\nu
		=\varphi^\epsilon_P(r)\varphi^\epsilon_P(s)\mu,
		\qquad r\geq0,
	\end{align*}
	which proves \eqref{Le3-1}.
	
	For fixed $\mu$, the continuity of $t\mapsto\varphi^\epsilon_P(t)\mu$ follows from the $L^2(\Omega;C([0,T];H))$ regularity in Lemma~\ref{Le2}. If $(t_n,\mu_n)\to(t,\mu)$, then
	\begin{align*}
		\mathcal W_2\big(\varphi^\epsilon_P(t_n)\mu_n,
		\varphi^\epsilon_P(t)\mu\big)
		&\leq
		\mathcal W_2\big(\varphi^\epsilon_P(t_n)\mu_n,
		\varphi^\epsilon_P(t_n)\mu\big)
		+\mathcal W_2\big(\varphi^\epsilon_P(t_n)\mu,
		\varphi^\epsilon_P(t)\mu\big)
		\\
		&\leq
		e^{(L_G-\bar\gamma)t_n/2}\mathcal W_2(\mu_n,\mu)
		+\mathcal W_2\big(\varphi^\epsilon_P(t_n)\mu,
		\varphi^\epsilon_P(t)\mu\big).
	\end{align*}
	The exponential factor remains bounded, and both terms tend to zero. Hence the law semiflow is jointly continuous.
\end{proof}

\begin{corollary}\label{Cor1}
	Assume in addition \textnormal{(A5)}. For every $\epsilon\in[0,\epsilon_0]$, the law semiflow $\varphi_P^\epsilon$ has a unique invariant law $\mu_\epsilon^*\in P$. Moreover,
	\begin{align}\label{Cor1-1}
		\mathcal W_2\big(
		\varphi_P^\epsilon(t)\mu,
		\mu_\epsilon^*
		\big)
		\leq
		e^{-\kappa_0t/2}
		\mathcal W_2(\mu,\mu_\epsilon^*),
		\qquad t\geq0,
	\end{align}
	for every $\mu\in P$. In particular, $\{\mu_\epsilon^*\}$ attracts every $\mathcal W_2$-bounded subset of $P$.
\end{corollary}

\begin{proof}
	By \eqref{Le3-2} and \eqref{kappa-def}, the time-one map $\varphi_P^\epsilon(1)$ is a strict contraction on the complete metric space $(P,\mathcal W_2)$. The Banach fixed-point theorem gives a unique $\mu_\epsilon^*\in P$ satisfying
	$\varphi_P^\epsilon(1)\mu_\epsilon^*=\mu_\epsilon^*$. For $s\geq0$, the semiflow identity gives
	\begin{align*}
		\varphi_P^\epsilon(1)
		\varphi_P^\epsilon(s)\mu_\epsilon^*
		=
		\varphi_P^\epsilon(s)
		\varphi_P^\epsilon(1)\mu_\epsilon^*
		=
		\varphi_P^\epsilon(s)\mu_\epsilon^*.
	\end{align*}
	The uniqueness of the fixed point of $\varphi_P^\epsilon(1)$ therefore implies
	$\varphi_P^\epsilon(s)\mu_\epsilon^*=\mu_\epsilon^*$ for every $s\geq0$. Estimate \eqref{Cor1-1} follows from \eqref{Le3-2} with $\nu=\mu_\epsilon^*$, and its uniform version on $\mathcal W_2$-bounded sets gives the last assertion.
\end{proof}

We next construct the pathwise state dynamics over the law semiflow. On the invariant full-measure set $\Omega_0$ fixed above, define
\begin{align*}
	\ell^\epsilon_\mu(t)
	=\varphi^\epsilon_P(t)\mu,
	\qquad t\geq0.
\end{align*}
For a deterministic law curve $\ell$ and $\omega\in\Omega_0$, let
$v_\ell(t,\omega;v_0)$ denote the solution of \eqref{Eq4-trans} with $\eta=\ell$ and initial value $v_0$.

\begin{lemma}\label{Le4}
	Let $\epsilon\in[0,\epsilon_0]$, $\omega\in\Omega_0$, $x\in H$ and $\mu\in P$. Then equation~\eqref{Eq4-trans} with $\eta=\ell^\epsilon_\mu$ and initial datum $x-\epsilon z(\omega)$ has a unique pathwise variational solution
	\begin{align*}
		v_{\ell^\epsilon_\mu}(\cdot,\omega;
		x-\epsilon z(\omega))
		\in
		C([0,T];H)\cap L^2(0,T;V)
		\cap L^p(0,T;L^p(\mathbb R^n))
	\end{align*}
	for every $T>0$. Consequently,
	\begin{align}\label{Le4-1}
		\varphi^\epsilon_H(t,\omega,x,\mu)
		=
		v_{\ell^\epsilon_\mu}
		(t,\omega;x-\epsilon z(\omega))
		+\epsilon z(\theta_t\omega)
	\end{align}
	is well defined. The map
	$(t,\omega,x,\mu)\mapsto\varphi^\epsilon_H(t,\omega,x,\mu)$ is jointly measurable, and, for all $x,y\in H$ and $\mu,\nu\in P$,
	\begin{align}\label{Le4-2}
		\|\varphi^\epsilon_H(t,\omega,x,\mu)
		-\varphi^\epsilon_H(t,\omega,y,\nu)\|^2
		&\leq
		e^{-\bar\gamma t}\|x-y\|^2
		+L_G\int_0^t e^{-\bar\gamma(t-r)}
		\mathcal W_2^2\big(
		\varphi^\epsilon_P(r)\mu,
		\varphi^\epsilon_P(r)\nu
		\big)dr
		\notag\\
		&\leq
		e^{-\bar\gamma t}\|x-y\|^2
		+\big(e^{(L_G-\bar\gamma)t}-e^{-\bar\gamma t}\big)
		\mathcal W_2^2(\mu,\nu).
	\end{align}
	In particular, $(x,\mu)\mapsto\varphi^\epsilon_H(t,\omega,x,\mu)$ is continuous from $H\times P$ to $H$ for every $t\geq0$ and $\omega\in\Omega_0$.
\end{lemma}

\begin{proof}
	Fix $T>0$, $\omega\in\Omega_0$ and $\mu\in P$, and abbreviate
	\begin{align*}
		\ell(t)=\ell^\epsilon_\mu(t),
		\qquad
		z_t=z(\theta_t\omega),
		\qquad 0\leq t\leq T.
	\end{align*}
	The continuity of the law semiflow and the Ornstein-Uhlenbeck paths
	give
	\begin{align*}
		M_{\ell,T}
		=\max_{0\leq t\leq T}M_2(\ell(t))<\infty,
		\qquad
		Z_{\omega,T}
		:=\max_{0\leq t\leq T}
		\big(1+\|z_t\|_{H^2}^p+\|z_t\|_p^p\big)<\infty.
	\end{align*}
	Write the shifted drift as
	\begin{align*}
		\widetilde A_{\ell,\omega}(t,v)
		=\widetilde A_{\ell,\omega,2}(t,v)
		+\widetilde A_{\ell,\omega,p}(t,v),
	\end{align*}
	where
	\begin{align*}
		\widetilde A_{\ell,\omega,2}(t,v)
		&=\Delta v-\lambda v+g+\epsilon\Delta z_t\in V^*,\\
		\widetilde A_{\ell,\omega,p}(t,v)
		&=-f(\cdot,v+\epsilon z_t)
		+\mathcal G(v+\epsilon z_t,\ell(t))
		\in L^{p'}(\mathbb R^n).
	\end{align*}
	By \eqref{f-2},
	\eqref{G-2} and the definitions of $M_{\ell,T}$ and $Z_{\omega,T}$,
	\begin{align}\label{Le4-3}
		\|\widetilde A_{\ell,\omega,2}(t,v)\|_{V^*}^{2}
		+\|\widetilde A_{\ell,\omega,p}(t,v)\|_{p'}^{p'}
		\leq
		C_{\ell,\omega,T}^{\rm gr}
		\big(1+\|v\|_V^2+\|v\|_p^p\big),
		\qquad 0\leq t\leq T,
	\end{align}
	uniformly for $\epsilon\in[0,\epsilon_0]$. The shifted drift is
	measurable in $t$ and hemicontinuous in $v$. Since the shift cancels in
	differences, \eqref{f-3} and \eqref{G-3} give, for
	$v_1,v_2\in\mathbb V$,
	\begin{align}\label{Le4-4}
		2\langle
		\widetilde A_{\ell,\omega}(t,v_1)
		-\widetilde A_{\ell,\omega}(t,v_2),
		v_1-v_2\rangle_{\mathbb V^*,\mathbb V}
		\leq
		-2\|\nabla(v_1-v_2)\|^2
		-\bar\gamma\|v_1-v_2\|^2.
	\end{align}
	
	We next verify coercivity. Put $u=v+\epsilon z_t$ and
	$\delta_f=\frac{\alpha_1}{2(\alpha_3+1)}$. H\"older's and Young's
	inequalities, followed by \eqref{f-2}, give
	\begin{align*}
		\epsilon|(f(\cdot,u),z_t)|
		\leq
		\delta_f\|f(\cdot,u)\|_{p'}^{p'}
		+C\epsilon_0^p\|z_t\|_p^p
		\leq
		\frac{\alpha_1}{2}\|u\|_p^p
		+\delta_f\|\psi_2\|_1
		+C\epsilon_0^p\|z_t\|_p^p.
	\end{align*}
	Combining this with \eqref{f-1} and using
	\begin{align*}
		\|v+\epsilon z_t\|_p^p
		\geq2^{1-p}\|v\|_p^p-\epsilon_0^p\|z_t\|_p^p,
		\qquad
		\|v+\epsilon z_t\|^2
		\leq2\|v\|^2+2\epsilon_0^2\|z_t\|^2,
	\end{align*}
	we obtain
	\begin{align*}
		(f(\cdot,v+\epsilon z_t),v)
		\geq
		\frac{\alpha_1}{2^p}\|v\|_p^p
		-2\alpha_2\|v\|^2
		-C\big(1+\|z_t\|_{\mathbb V}^p\big).
	\end{align*}
	A second lower bound follows from \eqref{f-3}:
	\begin{align*}
		(f(\cdot,v+\epsilon z_t),v)
		\geq
		-\alpha_4\|v\|^2+(f(\cdot,\epsilon z_t),v),
	\end{align*}
	and \eqref{f-2} gives
	\begin{align*}
		\|f(\cdot,\epsilon z_t)\|_{p'}^{p'}
		\leq
		\alpha_3\epsilon_0^p\|z_t\|_p^p+\|\psi_2\|_1.
	\end{align*}
	Choose $\vartheta\in(0,1)$ so small that
	\begin{align*}
		\gamma_\vartheta
		:=2\lambda-\beta_4-4\vartheta\alpha_2
		-2(1-\vartheta)\alpha_4>0.
	\end{align*}
	This is possible because the expression at $\vartheta=0$ equals
	$\bar\gamma>0$. Young's inequality gives
	\begin{align*}
		(1-\vartheta)|(f(\cdot,\epsilon z_t),v)|
		\leq
		\frac{\vartheta\alpha_1}{2^{p+1}}\|v\|_p^p
		+C_\vartheta\big(1+\|z_t\|_p^p\big).
	\end{align*}
	Taking the convex combination of the two preceding lower bounds yields
	\begin{align}\label{Le4-5}
		2(f(\cdot,v+\epsilon z_t),v)
		\geq
		\frac{\vartheta\alpha_1}{2^p}\|v\|_p^p
		-\big(4\vartheta\alpha_2
		+2(1-\vartheta)\alpha_4\big)\|v\|^2
		-C_\vartheta\big(1+\|z_t\|_{\mathbb V}^p\big).
	\end{align}
	
	For the McKean-Vlasov term, apply \eqref{G-3} to
	$(v+\epsilon z_t,\ell(t))$ and $(\epsilon z_t,\delta_{0_H})$. By
	\eqref{Le1-2},
	\begin{align}\label{Le4-6}
		2\langle\mathcal G(v+\epsilon z_t,\ell(t)),v\rangle_{\mathbb V^*,\mathbb V}
		\leq
		\beta_4\|v\|^2+L_GM_2(\ell(t))
		+2\langle\mathcal G(\epsilon z_t,\delta_{0_H}),v\rangle_{\mathbb V^*,\mathbb V}.
	\end{align}
	Assumption \eqref{G-2} gives
	\begin{align*}
		\|\mathcal G(\epsilon z_t,\delta_{0_H})\|_{p'}^{p'}
		\leq C\big(1+\|z_t\|_p^p\big).
	\end{align*}
	Consequently, for every $\delta>0$,
	\begin{align*}
		2\left|
		\langle\mathcal G(\epsilon z_t,\delta_{0_H}),v\rangle_{\mathbb V^*,\mathbb V}
		\right|
		\leq
		\delta\|v\|_p^p
		+C_\delta\big(1+\|z_t\|_p^p\big),
	\end{align*}
	and hence
	\begin{align}\label{Le4-7}
		2\langle\mathcal G(v+\epsilon z_t,\ell(t)),v\rangle_{\mathbb V^*,\mathbb V}
		\leq
		\beta_4\|v\|^2+L_GM_{\ell,T}
		+\delta\|v\|_p^p
		+C_\delta\big(1+\|z_t\|_p^p\big).
	\end{align}
	The remaining forcing satisfies
	\begin{align}\label{Le4-8}
		2|(g+\epsilon\Delta z_t,v)|
		\leq
		\delta\|v\|^2
		+C_\delta\big(\|g\|^2
		+\epsilon_0^2\|\Delta z_t\|^2\big).
	\end{align}
	Combining \eqref{Le4-5}, \eqref{Le4-7} and \eqref{Le4-8} with the
	definition of $\widetilde A_{\ell,\omega}$, and choosing
	$\delta>0$ sufficiently small, gives
	\begin{align}\label{Le4-9}
		2\langle\widetilde A_{\ell,\omega}(t,v),v\rangle_{\mathbb V^*,\mathbb V}
		\leq
		-\|\nabla v\|^2
		-\frac{\gamma_\vartheta}{2}\|v\|^2
		-\frac{\vartheta\alpha_1}{2^{p+1}}\|v\|_p^p
		+C_{\ell,\omega,T}^{\rm co},
	\end{align}
	where
	$C_{\ell,\omega,T}^{\rm co}
	=C_\vartheta
	\big(1+\|g\|^2+L_GM_{\ell,T}+Z_{\omega,T}\big)$.
	This constant is finite and uniform for
	$\epsilon\in[0,\epsilon_0]$. For fixed $(\ell,\omega)$, the
	Galerkin-Minty construction of Lemma~\ref{Le1} now applies:
	\eqref{Le4-3} and \eqref{Le4-9} provide the primal and dual bounds, the
	deterministic intersection-space integration-by-parts formula yields the
	$H$-continuous limit, and \eqref{Le4-4} identifies the drift. Hence
	\eqref{Eq4-trans} has a unique pathwise variational solution with the
	asserted regularity, consistent on nested time intervals.
	
	For stability, set
	\begin{align*}
		U_x(t)
		=\varphi^\epsilon_H(t,\omega,x,\mu),
		\qquad
		U_y(t)
		=\varphi^\epsilon_H(t,\omega,y,\nu).
	\end{align*}
	Their difference equals the difference of the transformed solutions,
	because the Ornstein-Uhlenbeck shifts are identical. The mixed-space
	energy identity, \eqref{f-3} and \eqref{G-3} give
	\begin{align*}
		\frac{d}{dt}\|U_x(t)-U_y(t)\|^2
		+2\|\nabla(U_x(t)-U_y(t))\|^2
		+\bar\gamma\|U_x(t)-U_y(t)\|^2
		\leq
		L_G\mathcal W_2^2(\ell^\epsilon_\mu(t),
		\ell^\epsilon_\nu(t)).
	\end{align*}
	Integration proves the first inequality in \eqref{Le4-2}. By
	\eqref{Le3-2},
	\begin{align*}
		L_G\int_0^t e^{-\bar\gamma(t-r)}
		\mathcal W_2^2(\ell^\epsilon_\mu(r),
		\ell^\epsilon_\nu(r))dr
		\leq
		L_G\mathcal W_2^2(\mu,\nu)
		\int_0^t e^{-\bar\gamma(t-r)}e^{(L_G-\bar\gamma)r}dr
		=
		\big(e^{(L_G-\bar\gamma)t}-e^{-\bar\gamma t}\big)
		\mathcal W_2^2(\mu,\nu),
	\end{align*}
	where the last expression also holds for $L_G=0$. This proves
	\eqref{Le4-2} and continuity in $(x,\mu)$.
	
	For joint measurability, retain the fixed Galerkin basis and choose a
	countable dense set $D\subset H\times P$. For each point of $D$ and each
	integer $T\geq1$, let $v_N$ be the corresponding Galerkin solution. Its
	finite-dimensional equation has Carath\'{e}odory coefficients, so $v_N$ is
	jointly measurable in $(t,\omega)$. On
	$\Omega_{T,M}:=\{Z_{\omega,T}\leq M\}$, estimates \eqref{Le4-3} and
	\eqref{Le4-9}, together with the equation, give uniform bounds in the mixed
	energy spaces and the corresponding dual sum spaces, and make each scalar
	path $t\mapsto(v_N(t),e_k)$ equicontinuous. A countable family of tests
	therefore metrizes the weak path topology on the localized bounded set.
	Every subsequence has a cluster point in this topology; passage to the
	limit in the variational equation and pathwise uniqueness show that all
	cluster points coincide. Hence the full localized sequence converges, and
	its limit is measurable as a pointwise limit into a metrizable space.
	
	Since $H$ is separable, its weak and strong Borel $\sigma$-fields coincide.
	The scalar evaluations and the $H$-continuous representative give joint
	measurability in $(t,\omega)$ for parameters in $D$, simultaneously on all
	integer time intervals. Estimate \eqref{Le4-2} extends the solution map
	continuously from $D$ to $H\times P$, uniformly on compact time intervals.
	Approximation by points of $D$, followed by the Carath\'{e}odory
	measurability theorem, yields joint measurability in
	$(t,\omega,x,\mu)$.
	
\end{proof}

The law semiflow and the decoupled state map now define the product evolution.

\begin{theorem}\label{The1}
	For every $\epsilon\in[0,\epsilon_0]$, the map
	\begin{align}
		\Phi^\epsilon(t,\omega)(x,\mu)
		=
		\big(
		\varphi^\epsilon_H(t,\omega,x,\mu),
		\varphi^\epsilon_P(t)\mu
		\big),
		\qquad (x,\mu)\in H\times P,
	\end{align}
	defines a continuous random dynamical system on $H\times P$ over
	$(\Omega_0,\mathcal F,\mathbb P,\{\theta_t\}_{t\in\mathbb R})$. In particular,
	\begin{align*}
		\Phi^\epsilon(0,\omega)
		=\operatorname{id}_{H\times P},
	\end{align*}
	and, for all $s,t\geq0$ and $\omega\in\Omega_0$,
	\begin{align}\label{The1-1}
		\Phi^\epsilon(t+s,\omega)
		=\Phi^\epsilon(t,\theta_s\omega)
		\circ\Phi^\epsilon(s,\omega).
	\end{align}
	Moreover, $(x,\mu)\mapsto\Phi^\epsilon(t,\omega)(x,\mu)$ is continuous on $H\times P$ for every $t\geq0$ and $\omega\in\Omega_0$.
\end{theorem}

\begin{proof}
	The identity at $t=0$, joint measurability and continuity follow from Lemmas~\ref{Le3}-\ref{Le4}. It remains to verify the cocycle identity. Fix $s,t\geq0$, $\omega\in\Omega_0$, $x\in H$ and $\mu\in P$, and set
	\begin{align*}
		\nu=\varphi_P^\epsilon(s)\mu,
		\qquad
		y=\varphi_H^\epsilon(s,\omega,x,\mu).
	\end{align*}
	By \eqref{Le3-1},
	\begin{align}\label{The1-2}
		\ell^\epsilon_\mu(s+r)
		=\varphi_P^\epsilon(s+r)\mu
		=\varphi_P^\epsilon(r)\nu
		=\ell^\epsilon_\nu(r),
		\qquad 0\leq r\leq t.
	\end{align}
	Let
	\begin{align*}
		v(r)
		=v_{\ell^\epsilon_\mu}
		(r,\omega;x-\epsilon z(\omega)),
		\qquad 0\leq r\leq s+t.
	\end{align*}
	By \eqref{Le4-1},
	\begin{align*}
		v(s)=y-\epsilon z(\theta_s\omega).
	\end{align*}
	Using \eqref{The1-2} and
	$z(\theta_r(\theta_s\omega))=z(\theta_{s+r}\omega)$, the function
	$r\mapsto v(s+r)$ solves \eqref{Eq4-trans} on $[0,t]$ over the shifted base point $\theta_s\omega$, with law curve $\ell^\epsilon_\nu$ and initial value $y-\epsilon z(\theta_s\omega)$. Pathwise uniqueness in Lemma~\ref{Le4} gives
	\begin{align*}
		v(s+t)
		=v_{\ell^\epsilon_\nu}
		(t,\theta_s\omega;y-\epsilon z(\theta_s\omega)).
	\end{align*}
	Adding $\epsilon z(\theta_{s+t}\omega)$ to both sides yields
	\begin{align*}
		\varphi^\epsilon_H(t+s,\omega,x,\mu)
		=
		\varphi^\epsilon_H\big(
		t,\theta_s\omega,
		\varphi^\epsilon_H(s,\omega,x,\mu),
		\varphi^\epsilon_P(s)\mu
		\big).
	\end{align*}
	Together with \eqref{Le3-1}, this proves \eqref{The1-1}. Because all objects are defined on the fixed shift-invariant set $\Omega_0$, the cocycle identity holds there for all $s,t\geq0$ and $(x,\mu)\in H\times P$.
\end{proof}

\section{Pullback random attractors in the product space}\label{sec4}

To apply Theorem~\ref{theorem1}, we construct a closed pullback absorbing family and prove pullback asymptotic compactness. All estimates are uniform in $\epsilon\in[0,\epsilon_0]$. Throughout this section, $C>0$ denotes a deterministic constant depending only on the structural data; subscripts indicate additional dependencies, and the value may change from line to line. A symbol of the form $C(\omega)$ denotes a finite random constant.

Retain the exponent $q>2$ fixed in Section~\ref{sec3}. For a nonempty bounded set $B\subset H$, set
$\|B\|_H=\sup_{x\in B}\|x\|$. A state family is denoted by
$B_H=\{B_H(\omega)\}_{\omega\in\Omega_0}$, where every $B_H(\omega)$ is a nonempty bounded subset of $H$. We define $\mathfrak D_H$ as the collection of all such families satisfying, for every $\delta>0$ and $\omega\in\Omega_0$,
\begin{align*}
	\lim_{t\to\infty}
	e^{-\delta t}\|B_H(\theta_{-t}\omega)\|_H^2=0.
\end{align*}
A law set is denoted by $B_P$, and we define
\begin{align*}
	\mathfrak D_P
	=
	\Big\{
	B_P\subset\mathcal P_q(H):
	B_P\neq\emptyset,\quad
	\sup_{\mu\in B_P}M_q(\mu)<\infty
	\Big\}.
\end{align*}
Both $\mathfrak D_H$ and $\mathfrak D_P$ are inclusion closed. Moreover, Lyapunov's inequality gives
\begin{align*}
	\sup_{\mu\in B_P}M_2(\mu)
	\leq
	\Big(\sup_{\mu\in B_P}M_q(\mu)\Big)^{2/q}<\infty,
	\qquad B_P\in\mathfrak D_P.
\end{align*}
The product universe induced by these two collections is
\begin{align*}
	\mathfrak D
	=
	\Big\{
	D=\{D(\omega)\}_{\omega\in\Omega_0}:{}
	&D(\omega)\neq\emptyset
	\text{ for every }\omega\in\Omega_0,
	\text{ and there exist }B_H\in\mathfrak D_H,
	B_P\in\mathfrak D_P
	\\[-1mm]
	&\text{such that }
	D(\omega)\subset B_H(\omega)\times B_P
	\text{ for every }\omega\in\Omega_0
	\Big\},
\end{align*}
in agreement with Section~\ref{sec2}.

\subsection{A uniform pullback absorbing set in \texorpdfstring{$H\times P$}{H x P}}

Recall that $\delta_{0_H}$ is the Dirac measure at $0\in H$. Set
\begin{align*}
	C_P=C_{\rm en}=2\|\psi_1\|_1+\beta_3
	+\frac{2}{\lambda_0}\|g\|^2
	+\epsilon_0^2\sum_{j=1}^{m}\|h_j\|^2,
	\qquad
	R_P^2=\frac{2C_P}{\lambda_0},
\end{align*}
where $C_{\rm en}$ and $\lambda_0$ are defined in Section~\ref{sec3}. We further define
\begin{align}\label{s4.1-1}
	\mathfrak z(\omega)
	=
	1+\|z(\omega)\|_{H^2}^2
	+\|z(\omega)\|_p^p,
	\qquad \omega\in\Omega_0.
\end{align}
By \eqref{h-1}, $z(\omega)\in H^2(\mathbb R^n)\cap\mathbb V$. Since the scalar Ornstein-Uhlenbeck processes are tempered, so are all polynomial functions of their finite-dimensional linear combinations. Consequently, $\mathfrak z$ is finite, measurable and tempered on $\Omega_0$.

\begin{lemma}\label{Le5}
	Assume \textnormal{(A1)-(A4)}. There exist constants
	$\gamma\in(0,\lambda_0/2)$ and $C_H>0$, independent of
	$\epsilon\in[0,\epsilon_0]$, such that, for every $t\geq0$,
	$\omega\in\Omega_0$, $x\in H$ and $\mu\in P$,
	\begin{align}\label{Le5-1}
		\left\|
		\varphi_H^\epsilon
		(t,\theta_{-t}\omega,x,\mu)
		\right\|^2
		\leq
		4e^{-\gamma t}\|x\|^2
		+4\epsilon_0^2e^{-\gamma t}
		\|z(\theta_{-t}\omega)\|^2
		+\frac{2L_G}{\lambda_0/2-\gamma}
		e^{-\gamma t}M_2(\mu)
		+R_H^2(\omega)-1,
	\end{align}
	where
	\begin{align*}
		R_H^2(\omega)
		=1+\frac{2L_GR_P^2}{\gamma}
		+2C_H\int_{-\infty}^{0}
		e^{\gamma s}\mathfrak z(\theta_s\omega)\,ds
		+2\epsilon_0^2\|z(\omega)\|^2.
	\end{align*}
	The mapping $\omega\mapsto R_H(\omega)$ is measurable and tempered.
	Hence
	\begin{align*}
		K_H(\omega)
		=
		\big\{u\in H:\|u\|^2\leq R_H^2(\omega)\big\},
		\qquad \omega\in\Omega_0,
	\end{align*}
	is a closed measurable random set belonging to $\mathfrak D_H$.
	
	Let $B_H\in\mathfrak D_H$, $B_P\in\mathfrak D_P$ and
	$\omega\in\Omega_0$. Then there exists $T=T(B_H,B_P,\omega)>0$
	such that, for all $t\geq T$ and $\epsilon\in[0,\epsilon_0]$,
	\begin{align*}
		\left\{
		\varphi_H^\epsilon
		(t,\theta_{-t}\omega,x,\mu):
		x\in B_H(\theta_{-t}\omega),\ \mu\in B_P
		\right\}
		\subset K_H(\omega).
	\end{align*}
	Thus $K_H$ is a pullback absorbing random set for the state component in
	$H$, uniformly in $\epsilon\in[0,\epsilon_0]$ and in the
	initial laws in $M_2$-bounded subsets of $P$.
\end{lemma}

\begin{proof}
	For the law component, let $u$ solve equation~\eqref{Eq1} with
	initial law $\mu\in P$ and write
	\begin{align*}
		M_2\big(\varphi_P^\epsilon(r)\mu\big)
		=\mathbb E\|u(r)\|^2.
	\end{align*}
	Applying It\^o's formula to $\|u(r)\|^2$, taking expectation,
	and using \eqref{f-1}, \eqref{G-1} and \eqref{gap}, we obtain
	\begin{align*}
		\frac{d}{dr}\mathbb E\|u(r)\|^2
		&+2\mathbb E\|\nabla u(r)\|^2
		+(2\alpha_1-\beta_G)
		\mathbb E\|u(r)\|_p^p
		+\lambda_0\mathbb E\|u(r)\|^2
		\\
		&\leq
		2\|\psi_1\|_1+\beta_3
		+2\mathbb E(g,u(r))
		+\epsilon^2\sum_{j=1}^{m}\|h_j\|^2.
	\end{align*}
	Young's inequality gives
	\begin{align*}
		2\mathbb E(g,u(r))
		\leq
		\frac{\lambda_0}{2}\mathbb E\|u(r)\|^2
		+\frac{2}{\lambda_0}\|g\|^2.
	\end{align*}
	Since $\beta_G<2\alpha_1$ and $\epsilon\leq\epsilon_0$, substitution yields
	\begin{align}\label{Le5-2a}
		&\frac{d}{dr}M_2\big(\varphi_P^\epsilon(r)\mu\big)
		+2\mathbb E\|\nabla u(r)\|^2
		+(2\alpha_1-\beta_G)\mathbb E\|u(r)\|_p^p
		+\frac{\lambda_0}{2}M_2\big(\varphi_P^\epsilon(r)\mu\big)
		\leq C_P.
	\end{align}
	Dropping the nonnegative gradient and $L^p$ terms and applying Gronwall's inequality,
	\begin{align}\label{Le5-2}
		M_2\big(\varphi_P^\epsilon(r)\mu\big)
		\leq
		e^{-\lambda_0r/2}M_2(\mu)+\frac{2C_P}{\lambda_0}=:e^{-\lambda_0r/2}M_2(\mu)+R_P^2,
		\qquad r\geq0,
	\end{align}
	uniformly for $\epsilon\in[0,\epsilon_0]$.
	
	For the pathwise estimate, fix $t>0$, $\omega\in\Omega_0$,
	$x\in H$ and $\mu\in P$. For $s\in[-t,0]$, define the transformed state variable
	\begin{align*}
		v(s)
		=
		\varphi_H^\epsilon
		(s+t,\theta_{-t}\omega,x,\mu)
		-\epsilon z(\theta_s\omega).
	\end{align*}
	Then $v(-t)=x-\epsilon z(\theta_{-t}\omega)$ and, on
	$[-t,0]$,
	\begin{align}\label{Le5-3}
		\frac{dv}{ds}
		-\Delta v+\lambda v
		+f\big(\cdot,v+\epsilon z(\theta_s\omega)\big)
		={}&
		\mathcal G\big(
		v+\epsilon z(\theta_s\omega),
		\varphi_P^\epsilon(s+t)\mu
		\big)
		+g+\epsilon\Delta z(\theta_s\omega).
	\end{align}
	
	We first estimate the term involving $f$. By \eqref{f-1}-\eqref{f-2}, H\"older's inequality and Young's inequality, choosing
	the coefficient of
	$\|f(\cdot,v+\epsilon z(\theta_s\omega))\|_{p'}^{p'}$
	to be $\alpha_1/[2(\alpha_3+1)]$, we have
	\begin{align*}
		\big(
		f(\cdot,v+\epsilon z(\theta_s\omega)),v
		\big)
		&=
		\big(
		f(\cdot,v+\epsilon z(\theta_s\omega)),
		v+\epsilon z(\theta_s\omega)
		\big)
		-\big(
		f(\cdot,v+\epsilon z(\theta_s\omega)),
		\epsilon z(\theta_s\omega)
		\big)
		\\
		&\geq
		\frac{\alpha_1}{2}
		\|v+\epsilon z(\theta_s\omega)\|_p^p
		-\alpha_2
		\|v+\epsilon z(\theta_s\omega)\|^2
		-C\big(1+\|z(\theta_s\omega)\|_p^p\big).
	\end{align*}
	Using
	\begin{align*}
		\|v+\epsilon z(\theta_s\omega)\|_p^p
		\geq
		2^{1-p}\|v\|_p^p
		-\epsilon_0^p\|z(\theta_s\omega)\|_p^p\quad\text{and}\quad
		\|v+\epsilon z(\theta_s\omega)\|^2
		\leq
		2\|v\|^2
		+2\epsilon_0^2\|z(\theta_s\omega)\|^2,
	\end{align*}
	we infer that
	\begin{align*}
		\big(
		f(\cdot,v+\epsilon z(\theta_s\omega)),v
		\big)
		\geq
		\frac{\alpha_1}{2^p}\|v\|_p^p
		-2\alpha_2\|v\|^2
		-C\mathfrak z(\theta_s\omega).
	\end{align*}
	On the other hand, applying \eqref{f-3} pointwise to
	$v+\epsilon z(\theta_s\omega)$ and
	$\epsilon z(\theta_s\omega)$, and then integrating over
	$\mathbb R^n$, we obtain
	\begin{align*}
		\big(
		f(\cdot,v+\epsilon z(\theta_s\omega)),
		v
		\big)
		\geq
		-\alpha_4\|v\|^2
		+\big(
		f(\cdot,\epsilon z(\theta_s\omega)),
		v
		\big).
	\end{align*}
	Moreover, \eqref{f-2} yields
	\begin{align*}
		\left\|
		f(\cdot,\epsilon z(\theta_s\omega))
		\right\|_{p'}^{p'}
		\leq
		\alpha_3\epsilon_0^p
		\|z(\theta_s\omega)\|_p^p
		+\|\psi_2\|_1.
	\end{align*}
	Choose $\vartheta\in(0,1)$ sufficiently small that
	\begin{align*}
		2\lambda-\beta_4-4\vartheta\alpha_2
		-2(1-\vartheta)\alpha_4>0,
	\end{align*}
	which is possible by \eqref{gap}. H\"older's and Young's inequalities
	then give
	\begin{align*}
		\big(
		f(\cdot,\epsilon z(\theta_s\omega)),
		v
		\big)
		&\geq
		-\left\|
		f(\cdot,\epsilon z(\theta_s\omega))
		\right\|_{p'}\|v\|_p\\
		&\geq
		-\frac{\vartheta\alpha_1}
		{2^{p+1}(1-\vartheta)}
		\|v\|_p^p
		-C_\vartheta
		\left\|
		f(\cdot,\epsilon z(\theta_s\omega))
		\right\|_{p'}^{p'}\\
		&\geq
		-\frac{\vartheta\alpha_1}
		{2^{p+1}(1-\vartheta)}
		\|v\|_p^p
		-C_\vartheta\mathfrak z(\theta_s\omega).
	\end{align*}
	Taking the convex combination of the preceding two lower
	bounds for
	$\big(f(\cdot,v+\epsilon z(\theta_s\omega)),
	v\big)$
	with weights $\vartheta$ and $1-\vartheta$, respectively,
	we obtain
	\begin{align}\label{Le5-4}
		2\big(
		f(\cdot,v+\epsilon z(\theta_s\omega)),
		v
		\big)
		\geq
		\frac{\vartheta\alpha_1}{2^p}
		\|v\|_p^p
		-\big(
		4\vartheta\alpha_2
		+2(1-\vartheta)\alpha_4
		\big)\|v\|^2
		-C_\vartheta\mathfrak z(\theta_s\omega).
	\end{align}
	
	Applying \eqref{G-3} to
	\begin{align*}
		\big(
		v+\epsilon z(\theta_s\omega),
		\varphi_P^\epsilon(s+t)\mu
		\big)
		\quad\text{and}\quad
		\big(\epsilon z(\theta_s\omega),\delta_{0_H}\big),
	\end{align*}
	and using
	$\mathcal W_2^2(\varphi_P^\epsilon(s+t)\mu,\delta_{0_H})
	=M_2(\varphi_P^\epsilon(s+t)\mu)$, we obtain
	\begin{align*}
		2\left\langle
		\mathcal G\big(
		v+\epsilon z(\theta_s\omega),
		\varphi_P^\epsilon(s+t)\mu
		\big)-\mathcal G\big(
		\epsilon z(\theta_s\omega), \delta_{0_H}\big),v
		\right\rangle_{\mathbb V^*,\mathbb V}
		\leq \beta_4\|v\|^2
		+L_GM_2\big(\varphi_P^\epsilon(s+t)\mu\big).
	\end{align*}
	It follows that
	\begin{align*}
		2\left\langle
		\mathcal G\big(
		v+\epsilon z(\theta_s\omega),
		\varphi_P^\epsilon(s+t)\mu
		\big),v
		\right\rangle_{\mathbb V^*,\mathbb V}
		\leq
		\beta_4\|v\|^2
		+L_GM_2\big(\varphi_P^\epsilon(s+t)\mu\big)
		+2\left\langle
		\mathcal G(\epsilon z(\theta_s\omega),\delta_{0_H}),
		v
		\right\rangle_{\mathbb V^*,\mathbb V}.
	\end{align*}
	For every $\delta>0$, \eqref{G-2} and Young's inequality imply
	\begin{align*}
		2\Big|
		\left\langle
		\mathcal G(\epsilon z(\theta_s\omega),\delta_{0_H}),
		v
		\right\rangle_{\mathbb V^*,\mathbb V}
		\Big|
		\leq
		\delta\|v\|_p^p
		+C_\delta\mathfrak z(\theta_s\omega).
	\end{align*}
	Substituting this estimate into the preceding inequality gives
	\begin{align}\label{Le5-5}
		2\left\langle
		\mathcal G\big(
		v+\epsilon z(\theta_s\omega),
		\varphi_P^\epsilon(s+t)\mu
		\big),v
		\right\rangle_{\mathbb V^*,\mathbb V}
		\leq
		\beta_4\|v\|^2
		+L_GM_2\big(\varphi_P^\epsilon(s+t)\mu\big)
		+\delta\|v\|_p^p
		+C_\delta\mathfrak z(\theta_s\omega).
	\end{align}
	
	Taking the $H$-inner product of \eqref{Le5-3} with $v$ and
	using \eqref{Le5-4}-\eqref{Le5-5} together with
	\begin{align*}
		2|(g,v)|
		&\leq
		\delta\|v\|^2+C_\delta\|g\|^2,\\
		2\epsilon|\big(\Delta z(\theta_s\omega),v\big)|
		&\leq
		\delta\|v\|^2
		+C_\delta\epsilon_0^2\|\Delta z(\theta_s\omega)\|^2,
	\end{align*}
	we obtain, for sufficiently small $\delta>0$, constants
	$c_1,c_2,c_3>0$, independent of $\epsilon$, such that
	\begin{align}\label{Le5-6}
		\frac{d}{ds}\|v(s)\|^2
		+c_1\|\nabla v(s)\|^2
		+c_2\|v(s)\|_p^p
		+c_3\|v(s)\|^2
		\leq{}&
		L_GM_2\big(\varphi_P^\epsilon(s+t)\mu\big)
		+C_H\mathfrak z(\theta_s\omega).
	\end{align}
	
	Fix $\gamma\in(0,\min\{c_3,\lambda_0/2\})$. Multiplying
	\eqref{Le5-6} by $e^{\gamma s}$ and integrating over $[-t,0]$, we obtain
	\begin{align*}
		\|v(0)\|^2
		\leq
		e^{-\gamma t}
		\|x-\epsilon z(\theta_{-t}\omega)\|^2
		+L_G\int_{-t}^{0}e^{\gamma s}
		M_2\big(\varphi_P^\epsilon(s+t)\mu\big)\,ds
		+C_H\int_{-t}^{0}e^{\gamma s}
		\mathfrak z(\theta_s\omega)\,ds.
	\end{align*}
	By \eqref{Le5-2} and $\gamma<\lambda_0/2$,
	\begin{align*}
		\int_{-t}^{0}e^{\gamma s}
		M_2\big(\varphi_P^\epsilon(s+t)\mu\big)\,ds
		\leq
		\frac{e^{-\gamma t}}{\lambda_0/2-\gamma}M_2(\mu)
		+\frac{R_P^2}{\gamma}.
	\end{align*}
	Combining these estimates, using
	\begin{align*}
		\|x-\epsilon z(\theta_{-t}\omega)\|^2
		\leq
		2\|x\|^2
		+2\epsilon_0^2\|z(\theta_{-t}\omega)\|^2,
	\end{align*}
	and extending the last integral to $(-\infty,0]$ gives
	\begin{align}\label{Le5-7}
		\|v(0)\|^2
		\leq
		2e^{-\gamma t}\|x\|^2
		+2\epsilon_0^2e^{-\gamma t}
		\|z(\theta_{-t}\omega)\|^2
		+\frac{L_G e^{-\gamma t}M_2(\mu)}{\lambda_0/2-\gamma}
		+\frac{L_GR_P^2}{\gamma}
		+C_H\int_{-\infty}^{0}
		e^{\gamma s}\mathfrak z(\theta_s\omega)\,ds.
	\end{align}
	Since
	\begin{align*}
		\varphi_H^\epsilon(t,\theta_{-t}\omega,x,\mu)
		=v(0)+\epsilon z(\omega),
	\end{align*}
	combining \eqref{Le5-7} with
	$\|a+b\|^2\leq2\|a\|^2+2\|b\|^2$ proves \eqref{Le5-1}.
	
	To prove measurability and temperedness of $R_H$, fix
	$\rho\in(0,\gamma)$. The temperedness and continuity of
	$s\mapsto\mathfrak z(\theta_s\omega)$ imply
	\begin{align*}
		C_\rho(\omega)
		=
		\sup_{r\leq0}e^{\rho r}\mathfrak z(\theta_r\omega)<\infty.
	\end{align*}
	Hence
	\begin{align*}
		\int_{-\infty}^{0}e^{\gamma s}
		\mathfrak z(\theta_s\omega)\,ds
		\leq
		\frac{C_\rho(\omega)}{\gamma-\rho}<\infty.
	\end{align*}
	Moreover, for $t\geq0$,
	\begin{align*}
		\int_{-\infty}^{0}e^{\gamma s}
		\mathfrak z\big(\theta_s(\theta_{-t}\omega)\big)\,ds
		=
		e^{\gamma t}\int_{-\infty}^{-t}
		e^{\gamma r}\mathfrak z(\theta_r\omega)\,dr
		\leq
		\frac{C_\rho(\omega)}{\gamma-\rho}e^{\rho t}.
	\end{align*}
	Given any $a>0$, choosing
	$\rho\in(0,\min\{a,\gamma\})$ and using the temperedness of $z$
	therefore yields
	\begin{align*}
		\lim_{t\to\infty}
		e^{-at}R_H^2(\theta_{-t}\omega)=0.
	\end{align*}
	The measurability of $R_H$ follows from the joint measurability of
	$(s,\omega)\mapsto z(\theta_s\omega)$. Since $H$ is separable, a closed
	ball with measurable random radius is a measurable random set. Thus
	$K_H\in\mathfrak D_H$.
	
	Let $B_H\in\mathfrak D_H$ and $B_P\in\mathfrak D_P$. By the
	definitions of $\mathfrak D_H$ and $\mathfrak D_P$, and by the
	temperedness of $z$,
	\begin{align*}
		4e^{-\gamma t}\|B_H(\theta_{-t}\omega)\|_H^2
		+4\epsilon_0^2e^{-\gamma t}
		\|z(\theta_{-t}\omega)\|^2
		+\frac{2L_G}{\lambda_0/2-\gamma}e^{-\gamma t}
		\sup_{\mu\in B_P}M_2(\mu)
		\longrightarrow0
	\end{align*}
	as $t\to\infty$. Hence the sum on the left is at most $1$ for all
	sufficiently large $t$, uniformly in $\epsilon\in[0,\epsilon_0]$.
	Estimate \eqref{Le5-1} then gives the asserted pullback absorption.
\end{proof}

The state estimate requires only a uniform second moment. The following $q$-moment bound completes the product-space absorption argument.

\begin{lemma}\label{Le6}
	Assume \textnormal{(A1)-(A4)}. There exists a constant $C_q>0$, independent of $\epsilon\in[0,\epsilon_0]$, $t$ and $\mu$, such that, with
	\begin{align}\label{Le6-1}
		R_{P,q}^q=\frac{4C_q}{q\lambda_0},
	\end{align}
	one has, for every $t\geq0$ and $\mu\in\mathcal P_q(H)$,
	\begin{align}\label{Le6-2}
		M_q\big(\varphi_P^\epsilon(t)\mu\big)
		\leq
		e^{-q\lambda_0t/4}M_q(\mu)+R_{P,q}^q.
	\end{align}
	In particular,
	$\varphi_P^\epsilon(t)\mathcal P_q(H)\subset\mathcal P_q(H)$
	for all $t\geq0$.
	
	Define
	\begin{align}\label{Le6-3}
		K_P
		&=
		\Big\{
		\nu\in\mathcal P_q(H):
		M_q(\nu)\leq R_{P,q}^q+1
		\Big\},\\
		K(\omega)
		&=K_H(\omega)\times K_P,
		\qquad \omega\in\Omega_0,
	\end{align}
	where $K_H$ is given by Lemma~\ref{Le5}. Then $K_P$ is a nonempty closed subset of $(P,\mathcal W_2)$ and belongs to $\mathfrak D_P$. The family $K=\{K(\omega)\}_{\omega\in\Omega_0}$ is a closed measurable random set in $H\times P$ and belongs to $\mathfrak D$.
	
	Moreover, for every $D\in\mathfrak D$ and $\omega\in\Omega_0$, there exists $T=T(D,\omega)>0$ such that, for all $t\geq T$ and $\epsilon\in[0,\epsilon_0]$,
	\begin{align}\label{Le6-4}
		\Phi^\epsilon(t,\theta_{-t}\omega)
		D(\theta_{-t}\omega)
		\subset K(\omega).
	\end{align}
	Consequently, $K$ is a closed $\mathfrak D$-pullback absorbing random set for $\Phi^\epsilon$ in $H\times P$, uniformly in $\epsilon\in[0,\epsilon_0]$.
\end{lemma}

\begin{proof}
	Let $\mu\in\mathcal P_q(H)$ and realize it by an $H$-valued
	random variable $u_0\in L^q$, independent of an $m$-dimensional Brownian
	motion. Let $u$ solve \eqref{Eq1}. By uniqueness in law, the
	following estimate depends only on $\mu$. Applying It\^o's formula to
	the Galerkin approximations of $\|u(r)\|^q$, localizing and
	passing to the limit with the mixed-energy bounds, gives, for a.e.\ $r>0$,
	\begin{align}\label{Le6-5}
		&\frac{d}{dr}\mathbb E\|u\|^q
		+q\mathbb E\big(
		\|u\|^{q-2}\|\nabla u\|^2
		\big)
		+q\lambda\mathbb E\|u\|^q
		+q\mathbb E\big[
		\|u\|^{q-2}
		(f(\cdot,u),u)
		\big]
		\notag\\
		&\quad=
		q\mathbb E\Big[
		\|u\|^{q-2}
		\left\langle
		\mathcal G\big(u,\varphi_P^\epsilon(r)\mu\big),
		u
		\right\rangle_{\mathbb V^*,\mathbb V}
		\Big]
		+q\mathbb E\big[
		\|u\|^{q-2}(g,u)
		\big]
		\notag\\
		&\qquad
		+\frac{q\epsilon^2}{2}
		\sum_{j=1}^{m}\|h_j\|^2
		\mathbb E\|u\|^{q-2}
		+\frac{q(q-2)\epsilon^2}{2}
		\sum_{j=1}^{m}
		\mathbb E\left[
		\|u\|^{q-4}(u,h_j)^2
		\right].
	\end{align}
	At $u=0$, the last integrand is set to zero. For arbitrary
	$q>2$, the formula follows by applying It\^o's formula to
	$(\|u\|^2+\delta)^{q/2}$ and letting $\delta\downarrow0$.
	
	By \eqref{f-1} and \eqref{G-1},
	\begin{align*}
		&q\mathbb E\big[
		\|u\|^{q-2}
		(f(\cdot,u),u)
		\big]
		-q\mathbb E\Big[
		\|u\|^{q-2}
		\left\langle
		\mathcal G\big(u,\varphi_P^\epsilon(r)\mu\big),
		u
		\right\rangle_{\mathbb V^*,\mathbb V}
		\Big]
		\notag\\
		&\quad\geq
		q\Big(\alpha_1-\frac{\beta_G}{2}\Big)
		\mathbb E\big(
		\|u\|^{q-2}\|u\|_p^p
		\big)
		-q\Big(\alpha_2+\frac{\beta_1}{2}\Big)
		\mathbb E\|u\|^q
		\notag\\
		&\qquad
		-\frac{q\beta_2}{2}
		M_2\big(\varphi_P^\epsilon(r)\mu\big)
		\mathbb E\|u\|^{q-2}
		-q\Big(\|\psi_1\|_1+\frac{\beta_3}{2}\Big)
		\mathbb E\|u\|^{q-2}.
	\end{align*}
	Since $\mathcal L u(r)=\varphi_P^\epsilon(r)\mu$, Lyapunov's inequality gives
	\begin{align*}
		M_2\big(\varphi_P^\epsilon(r)\mu\big)
		\mathbb E\|u(r)\|^{q-2}
		\leq
		M_q\big(\varphi_P^\epsilon(r)\mu\big).
	\end{align*}
	Moreover,
	\begin{align*}
		\sum_{j=1}^{m}
		\mathbb E\left[
		\|u\|^{q-4}(u,h_j)^2
		\right]
		\leq
		\sum_{j=1}^{m}\|h_j\|^2
		\mathbb E\|u\|^{q-2}.
	\end{align*}
	Substituting these estimates into \eqref{Le6-5} and using
	$q\lambda
	-q\big(\alpha_2+\frac{\beta_1}{2}\big)
	-\frac{q\beta_2}{2}
	=\frac{q\lambda_0}{2}$,
	we obtain
	\begin{align}\label{Le6-6}
		&\frac{d}{dr}
		M_q\big(\varphi_P^\epsilon(r)\mu\big)
		+q\mathbb E\big(
		\|u\|^{q-2}\|\nabla u\|^2
		\big)
		+q\Big(\alpha_1-\frac{\beta_G}{2}\Big)
		\mathbb E\big(
		\|u\|^{q-2}\|u\|_p^p
		\big)
		+\frac{q\lambda_0}{2}
		M_q\big(\varphi_P^\epsilon(r)\mu\big)
		\notag\\
		&\quad\leq
		q\|g\|\mathbb E\|u\|^{q-1}
		+q\Big(
		\|\psi_1\|_1+\frac{\beta_3}{2}
		+\frac{(q-1)\epsilon_0^2}{2}
		\sum_{j=1}^{m}\|h_j\|^2
		\Big)
		\mathbb E\|u\|^{q-2}.
	\end{align}
	For every $s\geq0$, Young's inequality gives
	\begin{align*}
		q\|g\|s^{q-1}
		+q\Big(
		\|\psi_1\|_1+\frac{\beta_3}{2}
		+\frac{(q-1)\epsilon_0^2}{2}
		\sum_{j=1}^{m}\|h_j\|^2
		\Big)s^{q-2}
		\leq
		\frac{q\lambda_0}{4}s^q+C_q,
	\end{align*}
	where $C_q>0$ depends only on $q$ and the structural data. Applying this inequality pointwise with $s=\|u(r)\|$ in \eqref{Le6-6} yields
	\begin{align}\label{Le6-7}
		&\frac{d}{dr}
		M_q\big(\varphi_P^\epsilon(r)\mu\big)
		+q\mathbb E\big(
		\|u(r)\|^{q-2}
		\|\nabla u(r)\|^2
		\big)
		\notag\\
		&\quad
		+q\Big(\alpha_1-\frac{\beta_G}{2}\Big)
		\mathbb E\big(
		\|u(r)\|^{q-2}
		\|u(r)\|_p^p
		\big)
		+\frac{q\lambda_0}{4}
		M_q\big(\varphi_P^\epsilon(r)\mu\big)
		\leq C_q.
	\end{align}
	Gronwall's inequality and \eqref{Le6-7} give \eqref{Le6-2}.
	
	The set $K_P$ is nonempty because $\delta_{0_H}\in K_P$, and
	\eqref{Le6-3} gives $K_P\in\mathfrak D_P$. Let $\nu_n\in K_P$ and $\mathcal W_2(\nu_n,\nu)\to0$. Then $\nu_n$ converges weakly to $\nu$, and the Portmanteau theorem gives
	\begin{align*}
		M_q(\nu)
		\leq
		\liminf_{n\to\infty}M_q(\nu_n)
		\leq R_{P,q}^q+1.
	\end{align*}
	Hence $K_P$ is closed in $P$.
	
	By Lemma~\ref{Le5}, $K_H$ is a closed measurable random set and $K_H\in\mathfrak D_H$. Thus \eqref{Le6-3} implies that $K$ is closed and belongs to $\mathfrak D$. For every $(x,\mu)\in H\times P$,
	\begin{align*}
		\inf_{(y,\nu)\in K(\omega)}
		d_{\mathbb X}\big((x,\mu),(y,\nu)\big)
		=
		\inf_{y\in K_H(\omega)}\|x-y\|
		+\inf_{\nu\in K_P}\mathcal W_2(\mu,\nu),
	\end{align*}
	so the measurability of $K$ follows from that of $K_H$.
	
	Let $D\in\mathfrak D$. There exist $B_H\in\mathfrak D_H$ and $B_P\in\mathfrak D_P$ such that
	\begin{align*}
		D(\omega')\subset B_H(\omega')\times B_P,
		\qquad \omega'\in\Omega_0.
	\end{align*}
	Choose $T_P=T_P(B_P)>0$ so that
	\begin{align*}
		e^{-q\lambda_0T_P/4}
		\sup_{\mu\in B_P}M_q(\mu)\leq1.
	\end{align*}
	Then \eqref{Le6-2} and the definition of $K_P$ give
	\begin{align*}
		\varphi_P^\epsilon(t)B_P\subset K_P,
		\qquad t\geq T_P,
	\end{align*}
	uniformly for $\epsilon\in[0,\epsilon_0]$. Since $B_P$ is $M_2$-bounded by Lyapunov's inequality, Lemma~\ref{Le5} gives $T_H=T_H(B_H,B_P,\omega)>0$ such that
	\begin{align*}
		\left\{
		\varphi_H^\epsilon
		(t,\theta_{-t}\omega,x,\mu):
		x\in B_H(\theta_{-t}\omega),\ \mu\in B_P
		\right\}
		\subset K_H(\omega),
		\qquad t\geq T_H,
	\end{align*}
	uniformly for $\epsilon\in[0,\epsilon_0]$. Taking $T=\max\{T_H,T_P\}$ proves \eqref{Le6-4}.
\end{proof}

\begin{remark}\label{Re1}
	Fix $D\in\mathfrak D$ and $\omega\in\Omega_0$. Lemmas~\ref{Le5}-\ref{Le6} give $T_0=T_0(D,\omega)\geq2$ such that, for $t\geq T_0$, $\epsilon\in[0,\epsilon_0]$ and $(x,\mu)\in D(\theta_{-t}\omega)$,
	\begin{align*}
		\varphi_H^\epsilon(t-1,\theta_{-t}\omega,x,\mu)\in K_H(\theta_{-1}\omega),\qquad
		\varphi_H^\epsilon(t,\theta_{-t}\omega,x,\mu)\in K_H(\omega),\qquad
		\varphi_P^\epsilon(t-1)\mu\in K_P.
	\end{align*}
	For $s\in[-t,0]$, let $v(s)$ be the transformed solution in \eqref{Le5-3}. Integrating \eqref{Le5-6} over $[-1,0]$ and \eqref{Le5-2a} over $[t-1,t]$, we obtain constants $C_0(\omega)>0$ and $C_0^P>0$, independent of $t$, $\epsilon$, $x$ and $\mu$, such that
	\begin{align}\label{Re1-1}
		&\int_{-1}^{0}
		\Big[
		\|v(s)\|_V^2
		+\|v(s)\|_p^p
		+M_2\big(\varphi_P^\epsilon(s+t)\mu\big)
		\Big]ds
		\leq C_0(\omega),
		\notag\\[-1mm]
		&\int_0^1\int_H
		\big(\|y\|_V^2+\|y\|_p^p\big)
		\big[\varphi_P^\epsilon(s+t-1)\mu\big](dy)ds
		\leq C_0^P.
	\end{align}
	These bounds depend only on the absorbed state and law at the start of the final unit interval.
\end{remark}

\begin{lemma}[Localized mixed-energy identities]\label{LeLocalized}
	Let $I=[a,b]$ and let $\zeta\in W^{1,\infty}(\mathbb R^n)$ be real valued. Suppose that
	\begin{align*}
		w\in C(I;H)\cap L^2(I;V)\cap L^p(I;L^p(\mathbb R^n))
	\end{align*}
	and, in $\mathcal D'(a,b;\mathbb V^*)$,
	\begin{align*}
		\frac{dw}{dt}=F_2+F_p,
		\qquad
		F_2\in L^2(I;V^*),
		\quad
		F_p\in L^{p'}(I;L^{p'}(\mathbb R^n)).
	\end{align*}
	Then, for $a\leq r\leq s\leq b$,
	\begin{align}\label{LeLocalized-1}
		\|\zeta w(s)\|^2-\|\zeta w(r)\|^2
		=2\int_r^s\Big[
		\langle F_2(\tau),\zeta^2w(\tau)\rangle_{V^*,V}
		+\langle F_p(\tau),\zeta^2w(\tau)\rangle_{L^{p'},L^p}
		\Big]d\tau.
	\end{align}
	If $U$ is progressively measurable,
	\begin{align*}
		U\in L^2(\Omega;C(I;H))
		\cap L^2(\Omega\times I;V)
		\cap L^p(\Omega\times I;L^p(\mathbb R^n)),
	\end{align*}
	and, in $\mathbb V^*$, $\mathbb P$-a.s.,
	\begin{align*}
		U(t)=U(a)+\int_a^t(F_2(s)+F_p(s))ds
		+\sum_{j=1}^{m}\int_a^t\sigma_j(s)dW_j(s),
	\end{align*}
	where $F_2\in L^2(\Omega\times I;V^*)$, $F_p\in L^{p'}(\Omega\times I;L^{p'}(\mathbb R^n))$ and $\sigma_j\in L^2(\Omega\times I;H)$ are progressively measurable, then, up to indistinguishability,
	\begin{align}\label{LeLocalized-2}
		\|\zeta U(t)\|^2
		&=\|\zeta U(a)\|^2
		+2\int_a^t\Big[
		\langle F_2(s),\zeta^2U(s)\rangle_{V^*,V}
		+\langle F_p(s),\zeta^2U(s)\rangle_{L^{p'},L^p}
		\Big]ds\notag\\
		&\quad
		+\sum_{j=1}^{m}\int_a^t\|\zeta\sigma_j(s)\|^2ds
		+2\sum_{j=1}^{m}\int_a^t
		(\zeta U(s),\zeta\sigma_j(s))dW_j(s).
	\end{align}
	The same formula is valid under the corresponding local integrability assumptions after stopping-time localization.
\end{lemma}

\begin{proof}
	Multiplication by $\zeta$ is bounded on $H$, $V$ and $L^p(\mathbb R^n)$. Its transpose maps $V^*$ and $L^{p'}(\mathbb R^n)$ continuously into themselves, and
	\begin{align*}
		\langle M_\zeta^*F,\zeta w\rangle
		=\langle F,\zeta^2w\rangle.
	\end{align*}
	Applying the intersection-space It\^o formula \cite{Gyongy17spde} to $M_\zeta U=\zeta U$ gives \eqref{LeLocalized-2}; \eqref{LeLocalized-1} is its zero-martingale case. Standard energy stopping times yield the localized form.
\end{proof}

\subsection{Pullback asymptotic compactness}

On $\mathbb R^n$, compactness is recovered by combining local spatial regularity with uniform far-field control. The Rellich theorem gives local compactness, while the tail estimate prevents loss of mass at infinity; see \cite{Wang99pd,Bates09jde,Tang16sd,Shi24jde}.

Assumptions \textnormal{(A1)-(A2)} contain the spatial increment identities and local derivative bounds used in the local gradient estimate. For $R>0$, recall that $B_R=\{x\in\mathbb R^n:|x|<R\}$, and choose $\chi_R\in C_c^\infty(\mathbb R^n)$ such that $0\leq\chi_R\leq1$, $\chi_R=1$ on $B_R$, $\chi_R=0$ on $\mathbb R^n\setminus B_{R+1}$, and
$|\nabla\chi_R|+|\Delta\chi_R|\leq c_R$ for a deterministic constant $c_R>0$. Set
\begin{align}\label{s4.2-1}
	\mathfrak z_1(\omega)
	=1+\|z(\omega)\|_{H^2}^2
	+\|z(\omega)\|_p^p+\|\nabla z(\omega)\|_p^p.
\end{align}
By \eqref{h-1}, $\nabla h_j\in L^2(\mathbb R^n)\cap L^\infty(\mathbb R^n)\subset L^p(\mathbb R^n)$ for $1\leq j\leq m$. Hence the finite-dimensional form of $z$ makes $\mathfrak z_1$ finite, measurable and tempered.

\begin{lemma}[Local spatial regularity]\label{Le7}
	Assume \textnormal{(A1)-(A4)}. Let $D\in\mathfrak D$, $\omega\in\Omega_0$ and $R>0$. Then there exist a time $T=T(D,\omega)\geq2$, independent of $R$, a finite random constant $C_R(\omega)>0$ and a deterministic constant $C_R^P>0$ such that, for all $t\geq T$, $\epsilon\in[0,\epsilon_0]$ and $(x,\mu)\in D(\theta_{-t}\omega)$,
	\begin{align}\label{Le7-1}
		\left\|
		\varphi_H^\epsilon
		(t,\theta_{-t}\omega,x,\mu)
		\right\|_{H^1(B_R)}^2
		\leq C_R(\omega),
	\end{align}
	and
	\begin{align}\label{Le7-2}
		\int_H\|y\|_{H^1(B_R)}^2
		\big[\varphi_P^\epsilon(t)\mu\big](dy)
		\leq C_R^P.
	\end{align}
	In \eqref{Le7-2}, the local $H^1$-norm is understood as an extended Borel functional on $H$, equal to $+\infty$ outside $H^1(B_R)$. In particular, the terminal state belongs to $H^1_{\rm loc}(\mathbb R^n)$ and the terminal law is concentrated on $H^1_{\rm loc}(\mathbb R^n)$.
\end{lemma}

\begin{proof}
	For $1\leq k\leq n$ and $0<|h|<1$, write
	\begin{align*}
		\tau_h^kw(\xi)=w(\xi+he_k),
		\qquad
		\delta_h^kw=\frac{\tau_h^kw-w}{h}.
	\end{align*}
	We use the standard difference-quotient estimates
	\begin{align*}
		\|\delta_h^kw\|_r\leq\|\partial_kw\|_r,
		\qquad w\in W^{1,r}(\mathbb R^n),\quad 1\leq r\leq\infty,
	\end{align*}
	and the local difference-quotient characterization of $H^1$; see
	\cite[Section~5.8.2]{Evans10}.
	
	We first prove the pathwise estimate. Let
	$T_0=T_0(D,\omega)\geq2$ be supplied by Remark~\ref{Re1}. Fix
	$t\geq T_0$, $\epsilon\in[0,\epsilon_0]$ and
	$(x,\mu)\in D(\theta_{-t}\omega)$. On $[-1,0]$, set
	\begin{align*}
		u(s)
		=\varphi_H^\epsilon
		(s+t,\theta_{-t}\omega,x,\mu)
		=v(s)+\epsilon z(\theta_s\omega)\quad\text{and}\quad
		\eta(s)=\varphi_P^\epsilon(s+t)\mu.
	\end{align*}
	Fix $k$ and put
	$q_h=\delta_h^kv$ and
	$r_h=\delta_h^kz(\theta_s\omega)$. For
	$\vartheta\in[0,1]$, define the state interpolation
	\begin{align*}
		u_{h,\vartheta}(s,\xi)
		=(1-\vartheta)u(s,\xi)
		+\vartheta\tau_h^ku(s,\xi)
		=u(s,\xi)
		+\vartheta\big(
		\tau_h^ku(s,\xi)-u(s,\xi)\big).
	\end{align*}
	Define, for a.e.\ $(s,\xi)$,
	\begin{align*}
		a_h(s,\xi)
		&=\int_0^1\partial_sf\big(
		\xi+he_k,u_{h,\vartheta}(s,\xi)
		\big)d\vartheta,\\
		F_h(s,\xi)
		&=\int_0^1\partial_kf\big(
		\xi+\vartheta he_k,u(s,\xi)
		\big)d\vartheta,\\
		b_h(s,\xi)
		&=\int_0^1\partial_sG\big(
		\xi+he_k,u_{h,\vartheta}(s,\xi),\eta(s)
		\big)d\vartheta,\\
		H_h(s,\xi)
		&=\int_0^1\partial_kG\big(
		\xi+\vartheta he_k,u(s,\xi),\eta(s)
		\big)d\vartheta.
	\end{align*}
	Splitting each Nemytskii increment into its state and fixed-state
	spatial parts and using \eqref{f-acl} and \eqref{G-acl}, we obtain
	\begin{align*}
		\delta_h^k\big[f(\cdot,u(s))\big]
		&=a_h(s)\big(q_h(s)+\epsilon r_h(s)\big)+F_h(s),\\
		\delta_h^k\big[G(\cdot,u(s),\eta(s))\big]
		&=b_h(s)\big(q_h(s)+\epsilon r_h(s)\big)+H_h(s).
	\end{align*}
	Taking the spatial difference quotient of \eqref{Le5-3} therefore gives,
	in the mixed variational space,
	\begin{align*}
		\frac{dq_h}{ds}-\Delta q_h+\lambda q_h
		+a_h(q_h+\epsilon r_h)+F_h
		=b_h(q_h+\epsilon r_h)+H_h
		+\delta_h^kg+\epsilon\delta_h^k\Delta z(\theta_s\omega).
	\end{align*}
	For fixed $h$, translations preserve $V$ and $L^p$, so $q_h$ has the
	same mixed energy regularity as $v$. Before decomposition,
	translation invariance of $L^{p'}$ gives
	$\delta_h^k[f(\cdot,u)],
	\delta_h^k[G(\cdot,u,\eta)]\in
	L^{p'}(-1,0;L^{p'}(\mathbb R^n))$, while the remaining terms belong to
	$L^2(-1,0;V^*)$. Thus the drift lies in
	\begin{align*}
		L^2(-1,0;V^*)+L^{p'}(-1,0;L^{p'}(\mathbb R^n)),
	\end{align*}
	and Lemma~\ref{LeLocalized} applies. The decomposition into
	$a_h,F_h,b_h,H_h$ is used only after multiplication by $\chi_R$, where
	the local bounds on $\psi_3$ and $\psi_4$ suffice.
	
	By \eqref{f-5}-\eqref{f-7} and \eqref{G-6}-\eqref{G-7},
	\begin{gather*}
		-\alpha_4\leq a_h
		\leq C\big(1+|u|^{p-2}
		+|\tau_h^ku|^{p-2}\big),\\
		2b_h\leq\beta_4,
		\qquad
		|b_h|\leq C\big(1+|u|^{p-2}
		+|\tau_h^ku|^{p-2}\big).
	\end{gather*}
	Moreover,
	$\operatorname{supp}\chi_R+[-1,1]e_k\subset B_{R+2}$, and
	$\psi_3,\psi_4$ are bounded on $B_{R+2}$. Consequently,
	\begin{align}\label{Le7-3}
		\|\chi_RF_h(s)\|^2+\|\chi_RH_h(s)\|^2
		\leq C_R\big(1+\|u(s)\|_p^p
		+M_2(\eta(s))\big),
	\end{align}
	uniformly for $0<|h|<1$.
	
	Apply \eqref{LeLocalized-1} to $q_h$ with multiplier $\chi_R$. For
	$-1\leq r\leq s\leq0$, this gives
	\begin{align*}
		\|\chi_Rq_h(s)\|^2
		&+2\int_r^s(\nabla q_h,\nabla(\chi_R^2q_h))d\tau
		+2\lambda\int_r^s\|\chi_Rq_h\|^2d\tau+2\int_r^s
		\big(a_h(q_h+\epsilon r_h)+F_h,\chi_R^2q_h\big)d\tau\\
		&\quad=\|\chi_Rq_h(r)\|^2
		+2\int_r^s
		\big(b_h(q_h+\epsilon r_h)+H_h,\chi_R^2q_h\big)d\tau\\
		&\qquad+2\int_r^s(\delta_h^kg,\chi_R^2q_h)d\tau
		+2\epsilon\int_r^s
		(\delta_h^k\Delta z(\theta_\tau\omega),\chi_R^2q_h)d\tau,
	\end{align*}
	where all unmarked quantities in the integrands are evaluated at time
	$\tau$. In particular, the localized energy is absolutely continuous.
	The diffusion term satisfies
	\begin{align*}
		2(\nabla q_h,\nabla(\chi_R^2q_h))
		=2\|\chi_R\nabla q_h\|^2
		+4\int_{\mathbb R^n}
		\chi_Rq_h\nabla\chi_R\cdot\nabla q_h\,d\xi
		\geq\|\chi_R\nabla q_h\|^2-C_R\|q_h\|^2.
	\end{align*}
	
	To estimate the state-derivative terms, decompose
	$a_h=a_h^+-a_h^-$ and $b_h=b_h^+-b_h^-$. The identities
	\begin{align*}
		2a_h^+(q_h+\epsilon r_h)q_h
		&=a_h^+\big(
		|q_h+\epsilon r_h|^2+|q_h|^2-\epsilon^2|r_h|^2
		\big),\\
		-2b_h^-(q_h+\epsilon r_h)q_h
		&=b_h^-\big(
		\epsilon^2|r_h|^2-|q_h+\epsilon r_h|^2-|q_h|^2
		\big)
	\end{align*}
	show the favorable signs of $a_h^+$ and $b_h^-$. For
	$p>2$,
	\begin{align*}
		|u|^{p-2}|r_h|^2
		\leq\delta|u|^p+C_\delta|r_h|^p,
	\end{align*}
	with the case $p=2$ immediate, while translation invariance gives
	\begin{align*}
		\|r_h\|_2+\|r_h\|_p
		\leq
		\|\partial_kz(\theta_s\omega)\|_2
		+\|\partial_kz(\theta_s\omega)\|_p.
	\end{align*}
	Since $a_h^-\leq\alpha_4$ and $2b_h^+\leq\beta_4$, the only
	unfavorable quadratic terms have coefficients $2\alpha_4$ and
	$\beta_4$; the $r_h$-terms are controlled by the preceding inequality,
	also after translating $u$. Together with \eqref{Le7-3} and
	Young's inequality, this gives, for every $\delta>0$,
	\begin{align}\label{Le7-4}
		2\big(a_h(q_h+\epsilon r_h)+F_h,\chi_R^2q_h\big)
		&\geq
		-(2\alpha_4+\delta)\|\chi_Rq_h\|^2
		-C_{\delta,R}\big(
		1+\|u\|_p^p+\mathfrak z_1(\theta_s\omega)
		\big),\notag\\
		2\big(b_h(q_h+\epsilon r_h)+H_h,\chi_R^2q_h\big)
		&\leq
		(\beta_4+\delta)\|\chi_Rq_h\|^2
		+C_{\delta,R}\big(
		1+\|u\|_p^p+M_2(\eta)
		+\mathfrak z_1(\theta_s\omega)
		\big).
	\end{align}
	After integrating the $z$-term by parts in space, the forcing terms satisfy
	\begin{gather*}
		2|(\delta_h^kg,\chi_R^2q_h)|
		\leq\delta\|\chi_Rq_h\|^2+C_\delta\|\partial_kg\|^2,\\
		2\epsilon\left|
		(\delta_h^k\Delta z,\chi_R^2q_h)
		\right|
		=2\epsilon\left|
		(\delta_h^k\nabla z,\nabla(\chi_R^2q_h))
		\right|
		\leq\delta\|\chi_R\nabla q_h\|^2
		+C_{\delta,R}\big(
		\|q_h\|^2+\|z(\theta_s\omega)\|_{H^2}^2
		\big).
	\end{gather*}
	Since
	$\|q_h\|\leq\|\partial_kv\|$ and
	\begin{align*}
		\|u(s)\|_p^p
		\leq C_{\epsilon_0}\big(
		\|v(s)\|_p^p+\|z(\theta_s\omega)\|_p^p
		\big),
	\end{align*}
	the positivity of
	$\bar\gamma=2\lambda-2\alpha_4-\beta_4$ allows $\delta$ to be
	chosen so that, for a.e.\ $s\in(-1,0)$,
	\begin{align}\label{Le7-5}
		&\frac{d}{ds}\|\chi_Rq_h(s)\|^2
		+c_4\|\chi_R\nabla q_h(s)\|^2
		+c_5\|\chi_Rq_h(s)\|^2\nonumber\\
		&\qquad\leq C_R\Big[
		1+\|v(s)\|_V^2+\|v(s)\|_p^p
		+M_2(\eta(s))+\mathfrak z_1(\theta_s\omega)
		\Big]=:\mathcal R_R(s),
	\end{align}
	where $c_4,c_5>0$ are independent of
	$R,h,t,\epsilon,x$ and $\mu$.
	
	Integrating from a.e.\ $r\in(-1,0)$ to $0$, dropping the nonnegative terms and then averaging over $r$ gives
	\begin{align*}
		\|\chi_Rq_h(0)\|^2
		&\leq\int_{-1}^0\|\chi_Rq_h(r)\|^2dr
		+\int_{-1}^0\int_r^0\mathcal R_R(s)dsdr\\
		&=\int_{-1}^0\|\chi_Rq_h(r)\|^2dr
		+\int_{-1}^0(s+1)\mathcal R_R(s)ds\\
		&\leq\int_{-1}^0\|\partial_kv(r)\|^2dr
		+\int_{-1}^0\mathcal R_R(s)ds.
	\end{align*}
	Using \eqref{Re1-1} and
	$\int_{-1}^0\mathfrak z_1(\theta_s\omega)ds<\infty$, and recalling
	that $\chi_R=1$ on $B_R$, we obtain
	\begin{align}\label{Le7-6}
		\sup_{0<|h|<1}
		\|\delta_h^kv(0)\|_{L^2(B_R)}^2
		\leq C_{1,R}(\omega),
	\end{align}
	where the right-hand side is independent of
	$t,\epsilon,x$ and $\mu$. The difference-quotient characterization of
	$H^1(B_R)$ yields
	\begin{align*}
		\|\nabla v(0)\|_{L^2(B_R)}^2
		\leq nC_{1,R}(\omega).
	\end{align*}
	Since
	$u(0)=\varphi_H^\epsilon(t,\theta_{-t}\omega,x,\mu)$
	belongs to $K_H(\omega)$ and
	$u(0)=v(0)+\epsilon z(\omega)$,
	\begin{align*}
		\|u(0)\|_{H^1(B_R)}^2
		&\leq R_H^2(\omega)
		+2nC_{1,R}(\omega)
		+2\epsilon_0^2\|\nabla z(\omega)\|_{L^2(B_R)}^2.
	\end{align*}
	This proves \eqref{Le7-1} after enlarging $C_R(\omega)$.
	
	For the law component, put
	$\nu=\varphi_P^\epsilon(t-1)\mu\in K_P$ and let $u$ solve
	\eqref{Eq3} on $[0,1]$ with initial law $\nu$. Then
	\begin{align*}
		\mathcal Lu(s)
		=\varphi_P^\epsilon(s)\nu
		=\varphi_P^\epsilon(s+t-1)\mu,
		\qquad 0\leq s\leq1.
	\end{align*}
	Fix $k$ and set $q_h=\delta_h^ku$. For each fixed $h$, its
	difference-quotient equation has the same mixed variational regularity
	and noise
	$\epsilon\sum_{j=1}^{m}\delta_h^kh_jdW_j$. The spatial increments decompose as
	\begin{align*}
		\delta_h^k[f(\cdot,u)]=a_hq_h+F_h,
		\qquad
		\delta_h^k[G(\cdot,u,\mathcal Lu)]
		=b_hq_h+H_h.
	\end{align*}
	Applying \eqref{LeLocalized-2} with multiplier $\chi_R$, up to a
	standard energy stopping time, yields, in differential notation,
	\begin{align*}
		&d\|\chi_Rq_h\|^2
		+2(\nabla q_h,\nabla(\chi_R^2q_h))ds
		+2\lambda\|\chi_Rq_h\|^2ds
		+2(a_hq_h+F_h,\chi_R^2q_h)ds\\
		&\quad=2(b_hq_h+H_h,\chi_R^2q_h)ds
		+2(\delta_h^kg,\chi_R^2q_h)ds
		+\epsilon^2\sum_{j=1}^{m}
		\|\chi_R\delta_h^kh_j\|^2ds
		+2\epsilon\sum_{j=1}^{m}
		(\chi_Rq_h,\chi_R\delta_h^kh_j)dW_j(s).
	\end{align*}
	The stopped stochastic integral has zero expectation. Using the same
	estimates as above, together with
	$\|\delta_h^kh_j\|\leq\|\partial_kh_j\|$, and then removing the
	stopping by Fatou's lemma, we obtain, uniformly for $0<|h|<1$,
	\begin{align}\label{Le7-7}
		&\frac{d}{ds}\mathbb E\|\chi_Rq_h(s)\|^2
		+c_4\mathbb E\|\chi_R\nabla q_h(s)\|^2
		+c_5\mathbb E\|\chi_Rq_h(s)\|^2\nonumber\\
		&\qquad\leq C_R\Big[
		1+\mathbb E\|u(s)\|_V^2
		+\mathbb E\|u(s)\|_p^p
		+M_2(\varphi_P^\epsilon(s)\nu)
		+\epsilon_0^2\sum_{j=1}^{m}\|\partial_kh_j\|^2
		\Big].
	\end{align}
	Denote the right-hand side by $\mathcal R_R^P(s)$. Integrating from
	a.e.\ $r\in(0,1)$ to $1$ and averaging in $r$ gives
	\begin{align*}
		\mathbb E\|\chi_Rq_h(1)\|^2
		&\leq\int_0^1\mathbb E\|\chi_Rq_h(r)\|^2dr
		+\int_0^1\int_r^1\mathcal R_R^P(s)dsdr\\
		&\leq\int_0^1\mathbb E\|\partial_ku(r)\|^2dr
		+\int_0^1\mathcal R_R^P(s)ds.
	\end{align*}
	The second estimate in \eqref{Re1-1} therefore yields
	\begin{align*}
		\sup_{0<|h|<1}
		\mathbb E\|\delta_h^ku(1)\|_{L^2(B_R)}^2
		\leq C_{2,R}^P.
	\end{align*}
	Regard $u(1)$ as the Bochner map
	$B_{R+1}\to L^2(\Omega)$, $\xi\mapsto u(1,\cdot,\xi)$.
	The uniform difference-quotient and $L^2$ bounds permit the
	vector-valued criterion \cite[Theorem~2.2]{Arendt18sm} on $B_R$.
	Because $L^2(\Omega)$ has the Radon-Nikod\'{y}m property, Fubini's
	theorem identifies the resulting Bochner-Sobolev space and gives
	\begin{align*}
		u(1)\in L^2(\Omega;H^1(B_R)),
		\qquad
		\mathbb E\|\nabla u(1)\|_{L^2(B_R)}^2
		\leq nC_{2,R}^P.
	\end{align*}
	The $L^2(B_R)$ part is uniformly bounded by \eqref{Le6-2} and
	Lyapunov's inequality. Hence, after enlarging $C_R^P$,
	\begin{align*}
		\mathbb E\|u(1)\|_{H^1(B_R)}^2
		\leq C_R^P.
	\end{align*}
	Because $\mathcal Lu(1)=\varphi_P^\epsilon(t)\mu$,
	\begin{align*}
		\int_H\|y\|_{H^1(B_R)}^2
		\big[\varphi_P^\epsilon(t)\mu\big](dy)
		=\mathbb E\|u(1)\|_{H^1(B_R)}^2
		\leq C_R^P,
	\end{align*}
	which is \eqref{Le7-2}.
	
	Because $T_0$ is independent of $R$, applying the estimates at all
	integer radii and taking the countable intersection proves the stated
	$H^1_{\rm loc}$ conclusions.
\end{proof}

For the far-field estimates, the cut-off in \eqref{G-4} satisfies
\begin{align*}
	\mathbf 1_{\{|\xi|>\sqrt{2}R\}}
	\leq\rho_R(\xi)
	\leq\mathbf 1_{\{|\xi|>R\}},
	\qquad
	|\nabla\rho_R(\xi)|\leq\frac{c}{R},
	\qquad R\geq1.
\end{align*}
The term $T_R(\mu)$ in \eqref{G-4} is handled by radius doubling and is not replaced by $\int\rho_R|u|^2$.

\begin{lemma}[Uniform far-field tails]\label{Le8}
	Assume \textnormal{(A1)-(A4)}. Let $D\in\mathfrak D$, $\omega\in\Omega_0$ and $\delta>0$. Then there exist
	$T=T(D,\omega,\delta)\geq2$ and
	$R=R(D,\omega,\delta)\geq1$ such that, for all
	$t\geq T$, $\epsilon\in[0,\epsilon_0]$ and
	$(x,\mu)\in D(\theta_{-t}\omega)$,
	\begin{align}\label{Le8-1}
		\int_{|\xi|>R}
		\left|
		\varphi_H^\epsilon
		(t,\theta_{-t}\omega,x,\mu)(\xi)
		\right|^2d\xi
		<\delta,
	\end{align}
	and
	\begin{align}\label{Le8-2}
		\int_H\int_{|\xi|>R}|y(\xi)|^2d\xi
		\big[\varphi_P^\epsilon(t)\mu\big](dy)
		<\delta.
	\end{align}
\end{lemma}

\begin{proof}
	By the definition of the product universe, there exist
	$B_H\in\mathfrak D_H$ and $B_P\in\mathfrak D_P$ such that
	\begin{align*}
		D(\omega')\subset B_H(\omega')\times B_P,
		\qquad \omega'\in\Omega_0.
	\end{align*}
	Set
	$M_D=\sup_{\nu\in B_P}M_2(\nu)<\infty$ and, for $R\geq1$,
	\begin{align*}
		\mathcal E_R
		=\frac1R+|r_G(R)|
		+\int_{|\xi|>R}
		\Big(|\psi_1(\xi)|+|\psi_2(\xi)|+|g(\xi)|^2
		+\sum_{j=1}^{m}|h_j(\xi)|^2\Big)d\xi.
	\end{align*}
	By \eqref{h-1}-\eqref{g-1} and $r_G(R)\to0$,
	$\mathcal E_R\to0$ as $R\to\infty$.
	
	\smallskip
	\noindent\emph{Step 1. Uniform tails of the law component.}
	Let $u$ solve \eqref{Eq3} with initial law $\nu\in B_P$, and set
	\begin{align*}
		Y_R(r)
		=\mathbb E\int_{\mathbb R^n}
		\rho_R(\xi)|u(r,\xi)|^2d\xi.
	\end{align*}
	Set $\zeta_R(\xi)=\widehat\rho(|\xi|^2/R^2)$, so that $\rho_R=\zeta_R^2$. Applying \eqref{LeLocalized-2} first up to the standard energy stopping times and then taking expectation gives
	\begin{align*}
		&\frac{d}{dr}Y_R(r)
		+2\mathbb E\int\rho_R|\nabla u|^2d\xi
		+2\mathbb E\int u\nabla\rho_R\cdot\nabla u d\xi
		+2\lambda Y_R(r)
		+2\mathbb E\big(f(\cdot,u),\rho_Ru\big)
		\notag\\
		&\quad=
		2\mathbb E\left\langle
		\mathcal G\big(u,\varphi_P^\epsilon(r)\nu\big),
		\rho_Ru
		\right\rangle_{\mathbb V^*,\mathbb V}
		+2\mathbb E\big(g,\rho_Ru\big)
		+\epsilon^2\sum_{j=1}^{m}\int\rho_R|h_j|^2d\xi.
	\end{align*}
	The cut-off derivative, \eqref{f-1}, \eqref{G-4} and Young's inequality yield, respectively,
	\begin{gather*}
		2\mathbb E\int\rho_R|\nabla u|^2d\xi
		+2\mathbb E\int u\nabla\rho_R\cdot\nabla u d\xi
		\geq
		-\frac{C}{R}\mathbb E\|u(r)\|_V^2,\\
		2\mathbb E\big(f(\cdot,u),\rho_Ru\big)
		\geq
		2\alpha_1\mathbb E\int\rho_R|u|^pd\xi
		-2\alpha_2\mathbb E\int\rho_R|u|^2d\xi
		-2\int_{|\xi|>R}|\psi_1(\xi)|d\xi,\\
		2\mathbb E\left\langle
		\mathcal G\big(u,\varphi_P^\epsilon(r)\nu\big),
		\rho_Ru
		\right\rangle_{\mathbb V^*,\mathbb V}
		\leq
		\beta_G\mathbb E\int\rho_R|u|^pd\xi
		+\beta_1\mathbb E\int\rho_R|u|^2d\xi
		+c_TT_R\big(\varphi_P^\epsilon(r)\nu\big)
		\\
		\qquad+|r_G(R)|
		\left(1+M_2\big(\varphi_P^\epsilon(r)\nu\big)\right),\\
		2\mathbb E\big(g,\rho_Ru\big)
		\leq
		\frac{\lambda_{\rm tail}}{2}Y_R(r)
		+\frac{2}{\lambda_{\rm tail}}
		\int_{|\xi|>R}|g(\xi)|^2d\xi.
	\end{gather*}
	Since
	$2\lambda-2\alpha_2-\beta_1-\lambda_{\rm tail}/2
	=c_T+\lambda_{\rm tail}/2$ by \eqref{tail-gap}, these estimates imply
	\begin{align}\label{Le8-3}
		\frac{d}{dr}Y_R(r)
		&+(2\alpha_1-\beta_G)
		\mathbb E\int\rho_R|u(r)|^pd\xi
		+(c_T+\lambda_{\rm tail}/2)Y_R(r)
		\notag\\
		&\leq
		c_TT_R\big(\varphi_P^\epsilon(r)\nu\big)
		+\frac{C}{R}\mathbb E\|u(r)\|_V^2
		+C\mathcal E_R
		\left(1+M_2\big(\varphi_P^\epsilon(r)\nu\big)\right).
	\end{align}
	
	Set
	$d_0=\min\big\{2,\frac{\lambda_0}{2}\big\}>0$.
	Dropping the $L^p$-term in the global law-energy inequality \eqref{Le5-2a} and using the definition of $d_0$, we obtain
	\begin{align*}
		\frac{d}{dr}M_2\big(\varphi_P^\epsilon(r)\nu\big)
		+d_0\mathbb E\|u(r)\|_V^2
		\leq C_P.
	\end{align*}
	Choose $\vartheta\in\big(0,\min\{c_T+\frac{\lambda_{\rm tail}}{2},d_0\}\big)$. Multiplying the preceding inequality at time $s$ by
	$e^{-\vartheta(r-s)}$ and integrating over $s\in(0,r)$ give, after
	integration by parts,
	\begin{align*}
		M_2\big(\varphi_P^\epsilon(r)\nu\big)
		-e^{-\vartheta r}M_2(\nu)
		+d_0\int_0^r e^{-\vartheta(r-s)}
		\mathbb E\|u(s)\|_V^2ds
		-\vartheta\int_0^r e^{-\vartheta(r-s)}
		M_2\big(\varphi_P^\epsilon(s)\nu\big)ds
		\leq\frac{C_P}{\vartheta}.
	\end{align*}
	Since $M_2(\varphi_P^\epsilon(s)\nu)\leq\mathbb E\|u(s)\|_V^2$, dropping the nonnegative terminal moment yields
	\begin{align}\label{Le8-3a}
		(d_0-\vartheta)
		\int_0^r e^{-\vartheta(r-s)}
		\mathbb E\|u(s)\|_V^2ds
		\leq
		e^{-\vartheta r}M_2(\nu)+\frac{C_P}{\vartheta}.
	\end{align}
	Integrating \eqref{Le8-3} with the factor $e^{-(c_T+\lambda_{\rm tail}/2)(r-s)}$ gives
	\begin{align}\label{Le8-4a}
		&Y_R(r)
		\leq
		e^{-(c_T+\lambda_{\rm tail}/2)r}Y_R(0)
		+c_T\int_0^r e^{-(c_T+\lambda_{\rm tail}/2)(r-s)}
		T_R\big(\varphi_P^\epsilon(s)\nu\big)ds
		\notag\\
		&+\frac{C}{R}\int_0^r e^{-(c_T+\lambda_{\rm tail}/2)(r-s)}
		\mathbb E\|u(s)\|_V^2ds
		+C\mathcal E_R\int_0^r e^{-(c_T+\lambda_{\rm tail}/2)(r-s)}
		\left(1+M_2\big(\varphi_P^\epsilon(s)\nu\big)\right)ds.
	\end{align}
	Here $Y_R(0)\leq M_D$. Moreover, \eqref{Le8-3a}, $(c_T+\lambda_{\rm tail}/2)>\vartheta$ and \eqref{Le5-2} uniformly bound the last two convolutions. Since $1/R\leq\mathcal E_R$, there exists $C_D>0$, independent of $R$, $r$, $\epsilon$ and $\nu\in B_P$, such that
	\begin{align}\label{Le8-4}
		Y_R(r)
		\leq
		e^{-(c_T+\lambda_{\rm tail}/2)r}M_D
		+c_T\int_0^r e^{-(c_T+\lambda_{\rm tail}/2)(r-s)}
		T_R\big(\varphi_P^\epsilon(s)\nu\big)ds
		+C_D\mathcal E_R.
	\end{align}
	
	Define
	\begin{align*}
		\mathcal T(R)
		=\limsup_{r\to\infty}
		\sup_{\epsilon\in[0,\epsilon_0]}
		\sup_{\nu\in B_P}
		T_R\big(\varphi_P^\epsilon(r)\nu\big).
	\end{align*}
	By \eqref{Le5-2}, $\mathcal T(R)$ is finite and nonincreasing in $R$. For fixed $S>0$ and $r>S$, splitting the convolution in \eqref{Le8-4} at $S$ gives
	\begin{align*}
		\int_0^r e^{-(c_T+\lambda_{\rm tail}/2)(r-s)}
		T_R\big(\varphi_P^\epsilon(s)\nu\big)ds
		&\leq
		\frac{M_D+R_P^2}{(c_T+\lambda_{\rm tail}/2)}e^{-(c_T+\lambda_{\rm tail}/2)(r-S)}
		\notag\\
		&\quad+\frac1{(c_T+\lambda_{\rm tail}/2)}
		\sup_{s\geq S}
		\sup_{\epsilon'\in[0,\epsilon_0]}
		\sup_{\nu'\in B_P}
		T_R\big(\varphi_P^{\epsilon'}(s)\nu'\big).
	\end{align*}
	Since $T_{\sqrt2R}(\varphi_P^\epsilon(r)\nu)\leq Y_R(r)$, first letting $r\to\infty$ and then $S\to\infty$ yields
	\begin{align}\label{Le8-5}
		\mathcal T(\sqrt2R)
		\leq
		\frac{c_T}{c_T+\lambda_{\rm tail}/2}\mathcal T(R)
		+C_D\mathcal E_R,
		\qquad R\geq1.
	\end{align}
	Since $\mathcal T$ is nonincreasing,
	$\mathcal T_\infty=\lim_{R\to\infty}\mathcal T(R)$ exists. Letting
	$R\to\infty$ in \eqref{Le8-5} and using $\mathcal E_R\to0$ gives
	\begin{align*}
		\mathcal T_\infty
		\leq
		\frac{c_T}{c_T+\lambda_{\rm tail}/2}\mathcal T_\infty.
	\end{align*}
	The coefficient is strictly smaller than one, so $\mathcal T_\infty=0$, namely,
	\begin{align}\label{Le8-5a}
		\lim_{R\to\infty}
		\limsup_{r\to\infty}
		\sup_{\epsilon\in[0,\epsilon_0]}
		\sup_{\nu\in B_P}
		T_R\big(\varphi_P^\epsilon(r)\nu\big)=0.
	\end{align}
	
	\smallskip
	\noindent\emph{Step 2. Localized pathwise energy estimate.}
	For $R\geq1$ and $\omega'\in\Omega_0$, set
	\begin{align*}
		\mathfrak r_R(\omega')
		=\int_{\mathbb R^n}\rho_R(\xi)
		\left(
		|z(\omega')|^2+|z(\omega')|^p
		+|\Delta z(\omega')|^2
		\right)d\xi
		+\|\rho_Rz(\omega')\|_p.
	\end{align*}
	For every $\omega'\in\Omega_0$, dominated convergence gives
	\begin{align*}
		\mathfrak r_R(\omega')\longrightarrow0
		\quad\text{as }R\to\infty,
		\qquad
		\mathfrak r_R(\omega')\leq C\mathfrak z(\omega'),
	\end{align*}
	where $\mathfrak z$ is defined in \eqref{s4.1-1}. Consequently, for every $\gamma>0$, the temperedness of $\mathfrak z$ and dominated convergence imply
	\begin{align}\label{Le8-5b}
		\lim_{R\to\infty}
		\int_{-\infty}^{0}e^{\gamma s}
		\mathfrak r_R(\theta_s\omega)ds=0.
	\end{align}
	
	Fix $t>0$, $\epsilon\in[0,\epsilon_0]$ and
	$(x,\mu)\in D(\theta_{-t}\omega)$. On $[-t,0]$, write
	\begin{align*}
		u(s)=v(s)+\epsilon z(\theta_s\omega),
		\qquad
		\eta(s)=\varphi_P^\epsilon(s+t)\mu.
	\end{align*}
	Applying the deterministic identity \eqref{LeLocalized-1} to \eqref{Le5-3} with multiplier $\zeta_R$ gives the localized identity
	\begin{align}\label{Le8-6a}
		&\frac{d}{ds}\int_{\mathbb R^n}\rho_R|v|^2d\xi
		+2\int_{\mathbb R^n}\rho_R|\nabla v|^2d\xi
		+2\int_{\mathbb R^n}v\nabla\rho_R\cdot\nabla v d\xi
		+2\lambda\int_{\mathbb R^n}\rho_R|v|^2d\xi
		\notag\\
		&\quad
		+2\big(f(\cdot,u),\rho_Rv\big)
		=2\left\langle
		\mathcal G(u,\eta),
		\rho_Rv
		\right\rangle_{\mathbb V^*,\mathbb V}
		+2\big(g,\rho_Rv\big)
		+2\epsilon\big(\Delta z(\theta_s\omega),\rho_Rv\big).
	\end{align}
	The diffusion terms satisfy
	\begin{align*}
		2\int\rho_R|\nabla v|^2d\xi
		+2\int v\nabla\rho_R\cdot\nabla v d\xi
		\geq
		-\frac{C}{R}\|v(s)\|_V^2.
	\end{align*}
	For every sufficiently small $\zeta>0$, the decomposition
	\begin{align*}
		2\big(f(\cdot,u),\rho_Rv\big)
		=2\big(f(\cdot,u),\rho_Ru\big)
		-2\epsilon\big(f(\cdot,u),\rho_Rz(\theta_s\omega)\big),
	\end{align*}
	together with \eqref{f-1}-\eqref{f-2} and Young's inequality, gives
	\begin{align*}
		2\big(f(\cdot,u),\rho_Rv\big)
		&\geq
		(2\alpha_1-\zeta)\int\rho_R|u|^pd\xi
		-2\alpha_2\int\rho_R|u|^2d\xi
		\notag\\
		&\quad-C_\zeta\int_{|\xi|>R}
		\big(|\psi_1(\xi)|+|\psi_2(\xi)|\big)d\xi
		-C_\zeta\int\rho_R|z(\theta_s\omega)|^pd\xi.
	\end{align*}
	
	For the McKean-Vlasov term, first write
	\begin{align*}
		2\left\langle
		\mathcal G(u,\eta),\rho_Rv
		\right\rangle
		=2\left\langle
		\mathcal G(u,\eta),\rho_Ru
		\right\rangle
		-2\epsilon\left\langle
		\mathcal G(u,\eta),\rho_Rz(\theta_s\omega)
		\right\rangle.
	\end{align*}
	By \eqref{G-6} and the fundamental theorem of calculus,
	\begin{align*}
		|G(\xi,r,\nu)-G(\xi,0,\nu)|
		\leq
		c_G\int_0^{|r|}\big(1+s^{p-2}\big)ds
		\leq
		c_G\big(|r|+|r|^{p-1}\big),
	\end{align*}
	while \eqref{G-2} at $u=0$ gives
	\begin{align*}
		\|\mathcal G(0,\nu)\|_{p'}
		\leq C\left(1+M_2(\nu)^{(p-1)/2}\right).
	\end{align*}
	Using these bounds and Young's inequality in the second term, and \eqref{G-4} in the first, we obtain
	\begin{align*}
		&2\left\langle
		\mathcal G(u,\eta),
		\rho_Rv
		\right\rangle_{\mathbb V^*,\mathbb V}
		\notag\\
		&\quad\leq
		(\beta_G+\zeta)\int\rho_R|u|^pd\xi
		+(\beta_1+\zeta)\int\rho_R|u|^2d\xi
		+c_TT_R(\eta)
		+|r_G(R)|(1+M_2(\eta))
		\notag\\
		&\qquad
		+C_\zeta\int\rho_R
		\big(|z(\theta_s\omega)|^2+|z(\theta_s\omega)|^p\big)d\xi
		+C\left(1+M_2(\eta)^{(p-1)/2}\right)
		\|\rho_Rz(\theta_s\omega)\|_p.
	\end{align*}
	The remaining terms satisfy
	\begin{align*}
		2\left|\big(g,\rho_Rv\big)\right|
		+2\epsilon\left|\big(\Delta z(\theta_s\omega),
		\rho_Rv\big)\right|
		\leq
		\zeta\int\rho_R|v|^2d\xi
		+C_\zeta\int\rho_R
		\big(|g|^2+|\Delta z(\theta_s\omega)|^2\big)d\xi,
	\end{align*}
	and
	\begin{align*}
		\int\rho_R|u|^2d\xi
		\leq
		(1+\zeta)\int\rho_R|v|^2d\xi
		+C_\zeta\int\rho_R|z(\theta_s\omega)|^2d\xi.
	\end{align*}
	By \eqref{Le5-2},
	$M_2(\eta(s))\leq M_D+R_P^2$ for $-t\leq s\leq0$. Since $2\alpha_1>\beta_G$ and
	$2\lambda-2\alpha_2-\beta_1=c_T+\lambda_{\rm tail}$, we may choose $\zeta>0$ sufficiently small in the preceding estimates. Substitution into \eqref{Le8-6a} then yields constants $a_H>c_T$ and $c_H>0$, independent of $R$, $t$, $\epsilon$, $x$ and $\mu$, such that
	\begin{align}\label{Le8-6}
		&\frac{d}{ds}\int_{\mathbb R^n}
		\rho_R(\xi)|v(s,\xi)|^2d\xi
		+a_H\int\rho_R|v|^2d\xi
		+c_H\int_{\mathbb R^n}\rho_R|u(s)|^pd\xi
		\notag\\
		&\quad\leq
		c_TT_R(\eta(s))
		+\frac{C}{R}\|v(s)\|_V^2
		+C_D\mathcal E_R
		+C_D\mathfrak r_R(\theta_s\omega).
	\end{align}
	
	\smallskip
	\noindent\emph{Step 3. Passage to the pullback tail limit.}
	Choose
	$\gamma_*\in(0,\min\{a_H,c_3,\lambda_0/2\})$ and set
	$c_E=\min\{c_1,c_3-\gamma_*\}>0$, where $c_1,c_3$ are from \eqref{Le5-6}. Multiplying \eqref{Le5-6} by $e^{\gamma_*s}$ and integrating over $[-t,0]$, we obtain
	\begin{align*}
		c_E\int_{-t}^{0}e^{\gamma_*s}\|v(s)\|_V^2ds
		\leq
		e^{-\gamma_*t}\|x-\epsilon z(\theta_{-t}\omega)\|^2
		+L_G\int_{-t}^{0}e^{\gamma_*s}
		M_2\big(\varphi_P^\epsilon(s+t)\mu\big)ds
		+C_H\int_{-t}^{0}e^{\gamma_*s}
		\mathfrak z(\theta_s\omega)ds.
	\end{align*}
	The first term is uniformly bounded for all sufficiently large $t$ by the temperedness of $B_H$ and $z$. Moreover, \eqref{Le5-2} and $\gamma_*<\lambda_0/2$ give
	\begin{align*}
		\int_{-t}^{0}e^{\gamma_*s}
		M_2\big(\varphi_P^\epsilon(s+t)\mu\big)ds
		\leq
		\frac{e^{-\gamma_*t}}{\lambda_0/2-\gamma_*}M_D
		+\frac{R_P^2}{\gamma_*}.
	\end{align*}
	Since the last weighted random integral is finite, there exist
	$C_E(D,\omega)>0$ and $T_E=T_E(D,\omega)$ such that
	\begin{align}\label{Le8-6b}
		\int_{-t}^{0}e^{\gamma_*s}\|v(s)\|_V^2ds
		\leq C_E(D,\omega),
		\qquad t\geq T_E,
	\end{align}
	uniformly for $\epsilon\in[0,\epsilon_0]$ and
	$(x,\mu)\in D(\theta_{-t}\omega)$.
	
	Multiplying \eqref{Le8-6} by $e^{\gamma_*s}$, integrating from $-t$ to $0$ and using $\gamma_*<a_H$, we find
	\begin{align}\label{Le8-7}
		\int_{\mathbb R^n}
		\rho_R(\xi)|v(0,\xi)|^2d\xi
		&\leq
		e^{-\gamma_*t}
		\|x-\epsilon z(\theta_{-t}\omega)\|^2
		+c_T\int_{-t}^{0}e^{\gamma_*s}
		T_R\big(\varphi_P^\epsilon(s+t)\mu\big)ds
		\notag\\
		&\quad+\frac{C}{R}\int_{-t}^{0}e^{\gamma_*s}
		\|v(s)\|_V^2ds
		+\frac{C_D}{\gamma_*}\mathcal E_R
		+C_D\int_{-t}^{0}e^{\gamma_*s}
		\mathfrak r_R(\theta_s\omega)ds.
	\end{align}
	For the law-tail convolution, the change of variables $r=s+t$ gives
	\begin{align*}
		\int_{-t}^{0}e^{\gamma_*s}
		T_R\big(\varphi_P^\epsilon(s+t)\mu\big)ds
		=
		\int_0^t e^{-\gamma_*(t-r)}
		T_R\big(\varphi_P^\epsilon(r)\mu\big)dr.
	\end{align*}
	For fixed $S>0$ and $t>S$, splitting the last integral at $S$ yields
	\begin{align*}
		\int_0^t e^{-\gamma_*(t-r)}
		T_R\big(\varphi_P^\epsilon(r)\mu\big)dr
		\leq
		\frac{M_D+R_P^2}{\gamma_*}e^{-\gamma_*(t-S)}
		+\frac1{\gamma_*}
		\sup_{r\geq S}
		\sup_{\epsilon'\in[0,\epsilon_0]}
		\sup_{\nu'\in B_P}
		T_R\big(\varphi_P^{\epsilon'}(r)\nu'\big).
	\end{align*}
	Therefore
	\begin{align}\label{Le8-7a}
		\limsup_{t\to\infty}
		\sup_{\epsilon\in[0,\epsilon_0]}
		\sup_{\mu\in B_P}
		\int_{-t}^{0}e^{\gamma_*s}
		T_R\big(\varphi_P^\epsilon(s+t)\mu\big)ds
		\leq
		\frac{\mathcal T(R)}{\gamma_*}.
	\end{align}
	The first term on the right-hand side of \eqref{Le8-7} tends to zero uniformly over the pullback initial data by temperedness. Taking the corresponding upper limits, using \eqref{Le8-5a}, \eqref{Le8-5b}, \eqref{Le8-6b}, \eqref{Le8-7a} and $\mathcal E_R\to0$, and then letting $R\to\infty$, we obtain
	\begin{align*}
		\lim_{R\to\infty}
		\limsup_{t\to\infty}
		\sup_{\epsilon\in[0,\epsilon_0]}
		\sup_{(x,\mu)\in D(\theta_{-t}\omega)}
		\int_{\mathbb R^n}
		\rho_R(\xi)|v(0,\xi)|^2d\xi=0.
	\end{align*}
	Finally, since $\rho_R=1$ on $\{|\xi|>\sqrt2R\}$ and
	$\varphi_H^\epsilon(t,\theta_{-t}\omega,x,\mu)
	=v(0)+\epsilon z(\omega)$,
	\begin{align*}
		\int_{|\xi|>\sqrt2R}
		\left|
		\varphi_H^\epsilon
		(t,\theta_{-t}\omega,x,\mu)(\xi)
		\right|^2d\xi
		\leq
		2\int_{\mathbb R^n}\rho_R|v(0)|^2d\xi
		+2\epsilon_0^2
		\int_{|\xi|>\sqrt2R}|z(\omega)(\xi)|^2d\xi.
	\end{align*}
	This proves the uniform state-tail estimate. Together with \eqref{Le8-5a}, it
	yields \eqref{Le8-1}-\eqref{Le8-2} after enlarging the radius and
	pullback time. Lemma~\ref{LeLocalized} justifies the localized
	identities at the variational level, and Fatou's lemma removes the
	stopping-time localization.
\end{proof}

We combine the $q$-moment, local regularity and far-field estimates of Lemmas~\ref{Le6}-\ref{Le8} with Prokhorov compactness and the characterization of $\mathcal W_2$ convergence; see \cite{Villani09,Ambrosio08lmzb}.

\begin{lemma}[Pullback asymptotic compactness]\label{Le9}
	Assume \textnormal{(A1)-(A4)}, with $q>2$ fixed above. Then the family
	$\{\Phi^\epsilon\}_{\epsilon\in[0,\epsilon_0]}$ is uniformly
	$\mathfrak D$-pullback asymptotically compact in $H\times P$ in the
	following sequential sense. Let $D\in\mathfrak D$, $\omega\in\Omega_0$,
	$t_n\to\infty$, $\epsilon_n\in[0,\epsilon_0]$ and
	$(x_n,\mu_n)\in D(\theta_{-t_n}\omega)$. Then there exist a subsequence,
	still indexed by $n$, and $(u,\nu)\in H\times P$ such that
	\begin{align}\label{Le9-1}
		d_{\mathbb X}\left(
		\Phi^{\epsilon_n}(t_n,\theta_{-t_n}\omega)(x_n,\mu_n),
		(u,\nu)
		\right)
		\longrightarrow0.
	\end{align}
	Consequently, for every fixed $\epsilon\in[0,\epsilon_0]$, the cocycle
	$\Phi^\epsilon$ is $\mathfrak D$-pullback asymptotically compact.
\end{lemma}

\begin{proof}
	Set
	\begin{align}\label{Le9-2}
		u_n
		=\varphi_H^{\epsilon_n}
		(t_n,\theta_{-t_n}\omega,x_n,\mu_n)\quad\text{and}\quad
		\nu_n
		=\varphi_P^{\epsilon_n}(t_n)\mu_n.
	\end{align}
	By Lemma~\ref{Le6}, after discarding finitely many terms there exists
	$N_0\in\mathbb N$ such that
	\begin{align}\label{Le9-3}
		(u_n,\nu_n)
		\in K_H(\omega)\times K_P,
		\quad
		M_q(\nu_n)
		\leq R_{P,q}^q+1,
		\qquad n\geq N_0.
	\end{align}
	Increasing $N_0$ if necessary, Lemma~\ref{Le7} gives, for every $R>0$,
	\begin{align}\label{Le9-4}
		\sup_{n\geq N_0}\|u_n\|_{H^1(B_R)}^2
		\leq C_R(\omega)
		\quad\text{and}\quad
		\sup_{n\geq N_0}
		\int_H\|y\|_{H^1(B_R)}^2\,\nu_n(dy)
		\leq C_R^P.
	\end{align}
	
	We next make the tail estimate uniform over the sequence
	$\{(u_n,\nu_n):n\geq N_0\}$. Let $\delta>0$. Lemma~\ref{Le8} yields
	$N_\delta\geq N_0$ and $R_\delta>0$ such that both tails are less than
	$\delta$ for $n\geq N_\delta$. For each of the finitely many indices
	$N_0\leq n<N_\delta$, the state tail tends to zero because $u_n\in H$,
	and the law tail tends to zero by dominated convergence because
	$\nu_n\in P$. Enlarging $R_\delta$ therefore gives
	\begin{align}\label{Le9-5}
		\lim_{R\to\infty}
		\sup_{n\geq N_0}
		\Big[
		\int_{|\xi|>R}|u_n(\xi)|^2d\xi
		+
		\int_H\int_{|\xi|>R}|y(\xi)|^2d\xi\,\nu_n(dy)
		\Big]
		=0.
	\end{align}
	
	We extract a convergent subsequence of the state component. For
	every integer $m\geq1$, the first estimate in \eqref{Le9-4} and the
	compact embedding $H^1(B_m)\hookrightarrow L^2(B_m)$ yield a subsequence
	converging in $L^2(B_m)$. A diagonal selection gives a subsequence, still
	denoted by $u_n$, which is Cauchy in $L^2(B_m)$ for every $m$. For $R>0$,
	\begin{align*}
		\|u_k-u_\ell\|^2
		\leq
		\|u_k-u_\ell\|_{L^2(B_R)}^2
		+2\int_{|\xi|>R}|u_k(\xi)|^2d\xi
		+2\int_{|\xi|>R}|u_\ell(\xi)|^2d\xi.
	\end{align*}
	Using \eqref{Le9-5}, first choose $R$ and then $k,\ell$. Hence the
	selected subsequence is Cauchy in $H$, and there exists $u\in H$ such
	that $u_n\to u$ in $H$.
	
	It remains to prove tightness of the corresponding laws. Fix $\eta\in(0,1)$.
	By \eqref{Le9-5}, choose an increasing sequence $R_j\geq j$ such that
	\begin{align*}
		\sup_{n\geq N_0}
		\int_H\int_{|\xi|>R_j}|y(\xi)|^2d\xi\,\nu_n(dy)
		\leq\eta\,2^{-2j-2},
		\qquad j\geq1.
	\end{align*}
	Define
	\begin{align*}
		\mathcal K_\eta
		=\bigcap_{j=1}^{\infty}
		\Big\{
		y\in H:
		\|y\|_{H^1(B_{R_j})}^2
		\leq\frac{2^{j+2}(1+C_{R_j}^P)}{\eta},
		\quad
		\int_{|\xi|>R_j}|y(\xi)|^2d\xi\leq2^{-j}
		\Big\}.
	\end{align*}
	The set $\mathcal K_\eta$ is compact in $H$. To see this, every sequence in
	$\mathcal K_\eta$ has, by the local $H^1$ bounds and a diagonal use of
	the Rellich theorem, a subsequence converging in $L^2(B_{R_j})$ for every
	$j$. For two elements of this subsequence,
	\begin{align*}
		\|y_k-y_\ell\|^2
		\leq
		\|y_k-y_\ell\|_{L^2(B_{R_j})}^2
		+2\int_{|\xi|>R_j}|y_k(\xi)|^2d\xi
		+2\int_{|\xi|>R_j}|y_\ell(\xi)|^2d\xi.
	\end{align*}
	Choosing $j$ large and then $k,\ell$ large shows that the subsequence
	is Cauchy in $H$. Its limit belongs to $\mathcal K_\eta$ by
	weak lower semicontinuity of the local $H^1$-norm and continuity of the
	$H$-tail functional. Thus $\mathcal K_\eta$ is sequentially compact and
	hence compact in $H$.
	
	By Markov's inequality and \eqref{Le9-4}, for $n\geq N_0$,
	\begin{align*}
		\nu_n(H\setminus\mathcal K_\eta)
		\leq
		\sum_{j=1}^{\infty}
		\Big[
		\frac{\eta}{2^{j+2}(1+C_{R_j}^P)}
		\int_H\|y\|_{H^1(B_{R_j})}^2\,\nu_n(dy)
		+2^j
		\int_H\int_{|\xi|>R_j}|y(\xi)|^2d\xi\,\nu_n(dy)
		\Big]
		\leq\eta.
	\end{align*}
	Hence $\{\nu_n:n\geq N_0\}$ is tight in $H$. Since $H$ is Polish,
	Prokhorov's theorem yields a further subsequence and a probability
	measure $\nu$ such that $\nu_n\rightharpoonup\nu$.
	
	The map $y\mapsto\|y\|^q$ is nonnegative and lower semicontinuous.
	Therefore \eqref{Le9-3} and the Portmanteau theorem give
	$M_q(\nu)\leq R_{P,q}^q+1$. Consequently, for every $M>0$,
	\begin{align}\label{Le9-6}
		\sup_{n\geq N_0}
		\int_{\{\|y\|>M\}}\|y\|^2\,\nu_n(dy)
		+
		\int_{\{\|y\|>M\}}\|y\|^2\,\nu(dy)
		\leq
		2M^{2-q}\big(R_{P,q}^q+1\big)
		\longrightarrow0
		\quad\text{as }M\to\infty.
	\end{align}
	For $M>0$, the function
	$y\mapsto\min\{\|y\|^2,M^2\}$ is bounded and continuous on $H$.
	Weak convergence and \eqref{Le9-6} therefore imply
	\begin{align*}
		\int_H\|y\|^2\,\nu_n(dy)
		\longrightarrow
		\int_H\|y\|^2\,\nu(dy).
	\end{align*}
	The characterization of $\mathcal W_2$ convergence by weak convergence
	and convergence of the second moments now yields
	\begin{align}\label{Le9-7}
		u_n&\longrightarrow u\quad\text{in }H,
		&
		\nu_n&\rightharpoonup\nu,
		&
		\int_H\|y\|^2\,\nu_n(dy)
		&\longrightarrow
		\int_H\|y\|^2\,\nu(dy),
		&
		\mathcal W_2(\nu_n,\nu)&\longrightarrow0.
	\end{align}
	In particular, $\nu\in P$. By \eqref{Le9-2}, \eqref{Le9-7} and the
	definition of $d_{\mathbb X}$, relation \eqref{Le9-1} follows. The
	argument allows arbitrary sequences $\epsilon_n\in[0,\epsilon_0]$;
	taking $\epsilon_n\equiv\epsilon$ proves the last assertion.
\end{proof}

\subsection{Existence of the pullback random attractor}

\begin{theorem}\label{The2}
	Assume \textnormal{(A1)-(A4)}. For every $\epsilon\in[0,\epsilon_0]$, the cocycle $\Phi^\epsilon$ possesses a unique $\mathfrak D$-pullback random attractor
	$\mathcal A^\epsilon=\{\mathcal A^\epsilon(\omega)\}_{\omega\in\Omega_0}$ in $H\times P$. It is given by
	\begin{align}\label{The2-1}
		\mathcal A^\epsilon(\omega)
		=
		\bigcap_{\tau\geq0}
		\overline{
			\bigcup_{t\geq\tau}
			\Phi^\epsilon(t,\theta_{-t}\omega)
			K(\theta_{-t}\omega)
		}^{H\times P},
		\qquad \omega\in\Omega_0,
	\end{align}
	and $\mathcal A^\epsilon(\omega)\subset K(\omega)$ for every $\omega\in\Omega_0$.
\end{theorem}

\begin{proof}
	By Theorem~\ref{The1}, $\Phi^\epsilon$ is a continuous random dynamical system on $H\times P$. Lemma~\ref{Le6} provides the closed measurable $\mathfrak D$-pullback absorbing family $K\in\mathfrak D$, and Lemma~\ref{Le9} gives $\mathfrak D$-pullback asymptotic compactness. Theorem~\ref{theorem1} therefore yields existence, uniqueness, representation \eqref{The2-1} and the inclusion $\mathcal A^\epsilon(\omega)\subset K(\omega)$.
\end{proof}

\begin{corollary}[Law attractor and pathwise fibers]\label{CorLaw}
	Assume \textnormal{(A1)-(A4)}. For every $\epsilon\in[0,\epsilon_0]$, the law semiflow $\varphi_P^\epsilon$ possesses a unique $\mathfrak D_P$-global attractor
	\begin{align}\label{CorLaw-1}
		\mathcal A_P^\epsilon
		=
		\bigcap_{\tau\geq0}
		\overline{
			\bigcup_{t\geq\tau}\varphi_P^\epsilon(t)K_P
		}^{P}.
	\end{align}
	Moreover, $\mathcal A_P^\epsilon\subset K_P$, so $\mathcal A_P^\epsilon\in\mathfrak D_P$, and
	\begin{align}\label{CorLaw-2}
		\Pi_P\mathcal A^\epsilon(\omega)=\mathcal A_P^\epsilon,
		\qquad \omega\in\Omega_0.
	\end{align}
	For $\mu\in\mathcal A_P^\epsilon$, define
	\begin{align*}
		\mathcal A_H^\epsilon(\omega,\mu)
		=\big\{x\in H:(x,\mu)\in\mathcal A^\epsilon(\omega)\big\}
	\end{align*}
	and
	\begin{align*}
		\Theta_{-t}^\epsilon\mu
		=\big\{\nu\in\mathcal A_P^\epsilon:
		\varphi_P^\epsilon(t)\nu=\mu\big\}.
	\end{align*}
	Then, for every $t\geq0$,
	\begin{align}\label{CorLaw-3}
		\mathcal A_H^\epsilon(\omega,\mu)
		=
		\bigcup_{\nu\in\Theta_{-t}^\epsilon\mu}
		\varphi_H^\epsilon\big(
		t,\theta_{-t}\omega,
		\mathcal A_H^\epsilon(\theta_{-t}\omega,\nu),\nu
		\big),
	\end{align}
	and, for every $B_H\in\mathfrak D_H$,
	\begin{align}\label{CorLaw-4}
		\lim_{t\to\infty}
		\sup_{\nu\in\Theta_{-t}^\epsilon\mu}
		\operatorname{dist}_H\left(
		\varphi_H^\epsilon\big(
		t,\theta_{-t}\omega,B_H(\theta_{-t}\omega),\nu
		\big),
		\mathcal A_H^\epsilon(\omega,\mu)
		\right)=0.
	\end{align}
\end{corollary}

\begin{proof}
	The closed set $K_P\in\mathfrak D_P$ absorbs every $B_P\in\mathfrak D_P$ by Lemma~\ref{Le6}. We verify asymptotic compactness of the law semiflow. Let $t_n\to\infty$ and $\mu_n\in B_P$. Fix $\omega\in\Omega_0$, let $B_0(\omega')=\{0\}$, and apply Lemma~\ref{Le9} to $D(\omega')=B_0(\omega')\times B_P$, with $\epsilon_n\equiv\epsilon$. The law coordinates $\varphi_P^\epsilon(t_n)\mu_n$ then have a convergent subsequence in $P$. The omega-limit construction, equivalently Theorem~\ref{theorem1} over the trivial base flow, gives the unique compact global attractor and \eqref{CorLaw-1}. Since $K_P$ absorbs itself and is closed, its omega-limit set is contained in $K_P$; hence $\mathcal A_P^\epsilon\in\mathfrak D_P$. Proposition~\ref{PrLawProjection} yields \eqref{CorLaw-2}, while Theorem~\ref{theorem2} gives \eqref{CorLaw-3}-\eqref{CorLaw-4}.
\end{proof}

\subsection{Invariant-law reduction and the random-equilibrium fiber}

Corollary~\ref{CorLaw} describes the law projection and pathwise fibers without a contraction assumption. Under \textnormal{(A5)}, the law attractor and its corresponding state fiber reduce to singletons.

\begin{corollary}\label{Cor2}
	Assume \textnormal{(A1)-(A5)}. For every $\epsilon\in[0,\epsilon_0]$, the invariant law $\mu_\epsilon^*$ from Corollary~\ref{Cor1} belongs to $\mathcal P_q(H)$, and $\{\mu_\epsilon^*\}$ is the $\mathfrak D_P$-global attractor of $\varphi_P^\epsilon$. Define
	\begin{align}\label{Cor2-1}
		\Psi^\epsilon(t,\omega)x
		=\varphi_H^\epsilon(t,\omega,x,\mu_\epsilon^*),
	\end{align}
	let $\Pi_H(x,\mu)=x$, and set $\mathcal A_H^\epsilon=\Pi_H\mathcal A^\epsilon$. Then $\Psi^\epsilon$ is a continuous random dynamical system on $H$, $\mathcal A_H^\epsilon$ is its unique $\mathfrak D_H$-pullback random attractor, and there exists a measurable random equilibrium $\xi_\epsilon:\Omega_0\to H$ such that
	\begin{align}\label{Cor2-2}
		\mathcal A_H^\epsilon(\omega)
		&=\{\xi_\epsilon(\omega)\},\\
		\mathcal A^\epsilon(\omega)
		&=\mathcal A_H^\epsilon(\omega)
		\times\{\mu_\epsilon^*\}
		=\big\{(\xi_\epsilon(\omega),\mu_\epsilon^*)\big\}.
	\end{align}
	Moreover,
	\begin{align}\label{Cor2-3}
		\Psi^\epsilon(t,\omega)\xi_\epsilon(\omega)
		&=\xi_\epsilon(\theta_t\omega),
		\qquad t\geq0,\\
		\mathbb P\circ\xi_\epsilon^{-1}
		&=\mu_\epsilon^*.
	\end{align}
	Consequently, $u_\epsilon(t,\omega)=\xi_\epsilon(\theta_t\omega)$ is a stationary variational solution of \eqref{Eq3} with one-time law $\mu_\epsilon^*$.
\end{corollary}

\begin{proof}
	\emph{Step 1: identification of the global law attractor.}
	Corollary~\ref{Cor1} gives
	\begin{align*}
		\varphi_P^\epsilon(k)\delta_{0_H}
		\longrightarrow\mu_\epsilon^*
		\quad\text{in }\mathcal W_2
		\quad\text{as }k\to\infty.
	\end{align*}
	On the other hand, \eqref{Le6-2} yields
	\begin{align*}
		M_q\big(
		\varphi_P^\epsilon(k)\delta_{0_H}
		\big)
		\leq R_{P,q}^q,
		\qquad k\geq0.
	\end{align*}
	Convergence in $\mathcal W_2$ implies weak convergence, and $v\mapsto\|v\|^q$ is nonnegative and lower semicontinuous. The Portmanteau theorem therefore gives
	\begin{align*}
		M_q(\mu_\epsilon^*)
		\leq R_{P,q}^q.
	\end{align*}
	Thus $\mu_\epsilon^*\in\mathcal P_q(H)$ and $\{\mu_\epsilon^*\}\in\mathfrak D_P$. If $B_P\in\mathfrak D_P$, then Lyapunov's inequality and the triangle inequality give
	\begin{align*}
		\sup_{\mu\in B_P}\mathcal W_2(\mu,\mu_\epsilon^*)
		\leq
		\sup_{\mu\in B_P}M_2(\mu)^{1/2}
		+M_2(\mu_\epsilon^*)^{1/2}<\infty.
	\end{align*}
	Estimate \eqref{Cor1-1} is therefore uniform over $\mu\in B_P$. Together with the invariance of $\mu_\epsilon^*$, this proves that $\{\mu_\epsilon^*\}$ is the $\mathfrak D_P$-global attractor. By the uniqueness in Corollary~\ref{CorLaw},
	\begin{align*}
		\mathcal A_P^\epsilon=\{\mu_\epsilon^*\}.
	\end{align*}
	
	\emph{Step 2: reduction to the invariant-law state cocycle.}
	The product attractor $\mathcal A^\epsilon$ belongs to $\mathfrak D$, and $\mathfrak D_H$ contains the constant singleton family $B_0=\{B_0(\omega)\}_{\omega\in\Omega_0}$ defined by $B_0(\omega)=\{0\}$. Proposition~\ref{PrLawProjection} and Corollary~\ref{corollary1}, applied with $X=H$ and $\mathcal A_P=\{\mu_\epsilon^*\}$, show that \eqref{Cor2-1} defines a continuous random dynamical system, that
	\begin{align*}
		\mathcal A^\epsilon(\omega)
		=\mathcal A_H^\epsilon(\omega)\times\{\mu_\epsilon^*\},
	\end{align*}
	and that $\mathcal A_H^\epsilon$ is its unique $\mathfrak D_H$-pullback random attractor. Setting $\mu=\nu=\mu_\epsilon^*$ in \eqref{Le4-2} gives
	\begin{align*}
		\|\Psi^\epsilon(t,\omega)x
		-\Psi^\epsilon(t,\omega)y\|^2
		\leq e^{-\bar\gamma t}\|x-y\|^2.
	\end{align*}
	For $B\in\mathfrak D_H$, write
	$\operatorname{diam}_H B(\omega)=\sup\{\|x-y\|:x,y\in B(\omega)\}$. Then
	\begin{align*}
		e^{-\bar\gamma t}
		\operatorname{diam}_H^2 B(\theta_{-t}\omega)
		\leq
		4e^{-\bar\gamma t}
		\|B(\theta_{-t}\omega)\|_H^2
		\longrightarrow0.
	\end{align*}
	Hence the hypotheses of Corollary~\ref{corollary2} hold with $\alpha=\bar\gamma$. That corollary yields the singleton representation in \eqref{Cor2-2} and the random-equilibrium identity in \eqref{Cor2-3}.
	
	\emph{Step 3: identification of the equilibrium law.}
	For $\nu\in P$, let $Q_t^{\epsilon,*}\nu$ be the law of $\Psi^\epsilon(t,\cdot)Y$ on an independent extension of the Wiener space, where $Y$ has law $\nu$ and is independent of the Wiener process. Given any coupling $(Y,\widetilde Y)$ of $(\nu,\widetilde\nu)$, independent of the same Wiener process, the preceding pathwise contraction gives
	\begin{align*}
		\mathbb E\|\Psi^\epsilon(t,\cdot)Y
		-\Psi^\epsilon(t,\cdot)\widetilde Y\|^2
		\leq
		e^{-\bar\gamma t}\mathbb E\|Y-\widetilde Y\|^2.
	\end{align*}
	Taking the infimum over all couplings proves
	\begin{align}\label{Cor2-4}
		\mathcal W_2\big(
		Q_t^{\epsilon,*}\nu,
		Q_t^{\epsilon,*}\widetilde\nu
		\big)
		\leq
		e^{-\bar\gamma t/2}
		\mathcal W_2(\nu,\widetilde\nu).
	\end{align}
	In particular, $Q_t^{\epsilon,*}$ is well defined independently of the realization of the initial variable.
	
	Since $\varphi_P^\epsilon(t)\mu_\epsilon^*=\mu_\epsilon^*$, a self-consistent solution with initial law $\mu_\epsilon^*$ evolves along the constant law curve. Uniqueness for the decoupled equation therefore gives
	\begin{align*}
		Q_t^{\epsilon,*}\mu_\epsilon^*=\mu_\epsilon^*,
		\qquad t\geq0.
	\end{align*}
	Applying \eqref{Cor2-4} with $\nu=\delta_{0_H}$ and $\widetilde\nu=\mu_\epsilon^*$ yields
	\begin{align*}
		\mathcal W_2\big(
		Q_t^{\epsilon,*}\delta_{0_H},\mu_\epsilon^*
		\big)
		\leq
		e^{-\bar\gamma t/2}
		\mathcal W_2(\delta_{0_H},\mu_\epsilon^*)
		\longrightarrow0.
	\end{align*}
	On the other hand, pullback attraction of $\{\xi_\epsilon(\omega)\}$ gives
	\begin{align*}
		\Psi^\epsilon(t,\theta_{-t}\omega)0
		\longrightarrow\xi_\epsilon(\omega)
		\quad\text{in }H.
	\end{align*}
	Because $\theta_{-t}$ preserves $\mathbb P$, the law of the left-hand side is $Q_t^{\epsilon,*}\delta_{0_H}$. The pathwise convergence implies weak convergence to $\mathbb P\circ\xi_\epsilon^{-1}$, whereas the preceding Wasserstein convergence implies weak convergence to $\mu_\epsilon^*$. Uniqueness of the weak limit proves the second identity in \eqref{Cor2-3}.
	
	\emph{Step 4: stationarity and self-consistency.}
	For each integer $k\geq1$, the pullback approximation
	$\Psi^\epsilon(k,\theta_{-k}\cdot)0$ depends only on the Wiener
	increments on $[-k,0]$ and is therefore $\mathcal F_0$-measurable.
	Since $\mathcal F_0$ is complete, its almost-sure limit $\xi_\epsilon$
	is $\mathcal F_0$-measurable.
	On the canonical Wiener space,
	$\xi_\epsilon\circ\theta_t$ is $\mathcal F_t$-measurable. Together with
	the random-equilibrium identity, this shows that
	$t\mapsto\xi_\epsilon(\theta_t\omega)$ is an adapted stationary
	solution with frozen law $\mu_\epsilon^*$. By Step~3, its one-time law is
	$\mu_\epsilon^*$; thus the solution is self-consistent and solves
	\eqref{Eq1}.
\end{proof}

\section{Zero-noise stability of the state-law attractors}\label{sec5}

For each $\epsilon\in[0,\epsilon_0]$, let $\Phi^\epsilon$ be the product cocycle constructed in Theorem~\ref{The1}, and let $\mathcal A^\epsilon$ be its unique pullback random attractor from Theorem~\ref{The2}. The symbols $\Phi^0$ and $\mathcal A^0$ denote the corresponding zero-noise objects.

We prove upper semicontinuity at $\epsilon=0$ under \textnormal{(A1)-(A4)}, without assuming \textnormal{(A5)}. By Theorem~\ref{theorem3}, it suffices to prove finite-time convergence of the product cocycles, uniformly on compact subsets, and precompactness of the small-noise attractors. Section~\ref{sec4} provides the required precompactness. Synchronous coupling controls the law coordinate in the finite-time comparison, while the Ornstein-Uhlenbeck transformation handles the state coordinate. This is the standard finite-time convergence and precompactness strategy for random attractors, here applied in the product space; see \cite{Caraballo98cpde,Wang09jde}. Under \textnormal{(A5)}, the singleton representation further yields quantitative convergence of the invariant laws and random equilibria.

We first compare the law semiflows.

For $\kappa\in\mathbb R$ and $t\geq0$, set
\begin{align}\label{Jkappa}
	\mathfrak J_\kappa(t)
	=\int_0^t e^{-\kappa(t-r)}dr
	=
	\begin{cases}
		\displaystyle\frac{1-e^{-\kappa t}}{\kappa},&\kappa\neq0,\\[2mm]
		t,&\kappa=0.
	\end{cases}
\end{align}

\begin{lemma}\label{Le10}
	Assume \textnormal{(A1)-(A4)}. For every $\epsilon\in[0,\epsilon_0]$, $t\geq0$ and $\mu,\nu\in P$,
	\begin{align}\label{Le10-1}
		\mathcal W_2^2\big(
		\varphi_P^\epsilon(t)\mu,
		\varphi_P^0(t)\nu
		\big)
		\leq
		e^{(L_G-\bar\gamma)t}\mathcal W_2^2(\mu,\nu)
		+\epsilon^2\mathfrak J_{\bar\gamma-L_G}(t)
		\sum_{j=1}^{m}\|h_j\|^2.
	\end{align}
	Consequently, for every $T>0$,
	\begin{align}\label{Le10-2}
		\sup_{0\leq t\leq T}\sup_{\mu\in P}
		\mathcal W_2\big(
		\varphi_P^\epsilon(t)\mu,
		\varphi_P^0(t)\mu
		\big)
		\leq
		\epsilon
		\Big(
		\mathfrak J_{\bar\gamma-L_G}(T)
		\sum_{j=1}^{m}\|h_j\|^2
		\Big)^{1/2}.
	\end{align}
	If \textnormal{(A5)} also holds, then
	\begin{align}\label{Le10-3}
		\sup_{t\geq0}\sup_{\mu\in P}
		\mathcal W_2\big(
		\varphi_P^\epsilon(t)\mu,
		\varphi_P^0(t)\mu
		\big)
		\leq
		\epsilon
		\Big(
		\frac1{\kappa_0}\sum_{j=1}^{m}\|h_j\|^2
		\Big)^{1/2}.
	\end{align}
\end{lemma}

\begin{proof}
	Let $(\xi,\zeta)$ be an arbitrary coupling of $(\mu,\nu)$. Realize this coupling on a probability space independent of the canonical Wiener space and pass to the product stochastic basis. On this basis, let $u^\epsilon$ and $u^0$ solve \eqref{Eq3} with noise amplitudes $\epsilon$ and $0$, and with initial data $\xi$ and $\zeta$, respectively. The intersection-space It\^o formula used in Lemma~\ref{Le1}, together with \eqref{f-3} and \eqref{G-3}, yields, for a.e.\ $t>0$,
	\begin{align*}
		\frac{d}{dt}\mathbb E\|u^\epsilon(t)-u^0(t)\|^2
		&+2\mathbb E\|\nabla(u^\epsilon(t)-u^0(t))\|^2
		+\bar\gamma\mathbb E\|u^\epsilon(t)-u^0(t)\|^2\\
		&\leq
		L_G\mathcal W_2^2\big(
		\mathcal Lu^\epsilon(t),\mathcal Lu^0(t)
		\big)
		+\epsilon^2\sum_{j=1}^{m}\|h_j\|^2.
	\end{align*}
	Since $(u^\epsilon(t),u^0(t))$ is a coupling of its marginal laws, we obtain
	\begin{align*}
		\frac{d}{dt}\mathbb E\|u^\epsilon(t)-u^0(t)\|^2
		+(\bar\gamma-L_G)\mathbb E\|u^\epsilon(t)-u^0(t)\|^2
		\leq\epsilon^2\sum_{j=1}^{m}\|h_j\|^2.
	\end{align*}
	Variation of constants, with no sign restriction on $\bar\gamma-L_G$, yields
	\begin{align*}
		\mathbb E\|u^\epsilon(t)-u^0(t)\|^2
		\leq
		e^{(L_G-\bar\gamma)t}\mathbb E\|\xi-\zeta\|^2
		+\epsilon^2\mathfrak J_{\bar\gamma-L_G}(t)
		\sum_{j=1}^{m}\|h_j\|^2.
	\end{align*}
	Taking the infimum over all couplings proves \eqref{Le10-1}. Relation \eqref{Le10-2} follows by taking $\nu=\mu$ and using the monotonicity of $\mathfrak J_{\bar\gamma-L_G}$. Under \textnormal{(A5)}, $\bar\gamma-L_G=\kappa_0$ and $\mathfrak J_{\kappa_0}(t)\leq1/\kappa_0$, which gives \eqref{Le10-3}.
\end{proof}

\begin{corollary}\label{Cor3}
	Assume \textnormal{(A1)-(A5)}. The invariant laws from Corollary~\ref{Cor1} satisfy
	\begin{align}\label{Cor3-1}
		\mathcal W_2(\mu_\epsilon^*,\mu_0^*)
		\leq
		\epsilon
		\Big(
		\frac1{\kappa_0}
		\sum_{j=1}^{m}\|h_j\|^2
		\Big)^{1/2},
		\qquad \epsilon\in[0,\epsilon_0].
	\end{align}
\end{corollary}

\begin{proof}
	Apply \eqref{Le10-1} with $\mu=\mu_\epsilon^*$ and $\nu=\mu_0^*$. Invariance gives, for every $t>0$,
	\begin{align*}
		\mathcal W_2^2(\mu_\epsilon^*,\mu_0^*)
		\leq
		e^{-\kappa_0t}\mathcal W_2^2(\mu_\epsilon^*,\mu_0^*)
		+\epsilon^2\mathfrak J_{\kappa_0}(t)
		\sum_{j=1}^{m}\|h_j\|^2.
	\end{align*}
	Since $\mathfrak J_{\kappa_0}(t)=(1-e^{-\kappa_0t})/\kappa_0$ under \textnormal{(A5)}, cancellation of $1-e^{-\kappa_0t}$ proves \eqref{Cor3-1}.
\end{proof}

We next compare the pathwise state components. By \eqref{Le4-1}, each
$\varphi_H^\epsilon$ is reconstructed from the transformed decoupled
solution along its law trajectory. We compare those transformed
solutions first and add the Ornstein-Uhlenbeck shift back at the end;
no identification with the self-consistent process $u^\epsilon$ is used.

\begin{lemma}\label{Le11}
	Assume \textnormal{(A1)-(A4)}. Let $T>0$, $R>0$ and $\omega\in\Omega_0$. There exists $C_{T,\omega,R}>0$ such that, for every $\epsilon\in(0,\epsilon_0]$, $x_\epsilon,x_0\in H$ and $\mu_\epsilon,\mu_0\in P$ satisfying
	\begin{align*}
		\|x_\epsilon\|^2+\|x_0\|^2
		+M_2(\mu_\epsilon)+M_2(\mu_0)\leq R,
	\end{align*}
	one has
	\begin{align}\label{Le11-1}
		\sup_{0\leq t\leq T}
		\left\|
		\varphi_H^\epsilon(t,\omega,x_\epsilon,\mu_\epsilon)
		-\varphi_H^0(t,\omega,x_0,\mu_0)
		\right\|^2
		\leq
		C_{T,\omega,R}
		\left(
		\|x_\epsilon-x_0\|^2
		+\mathcal W_2^2(\mu_\epsilon,\mu_0)
		+\epsilon
		\right).
	\end{align}
	Consequently, for every compact set $\mathcal C\subset\mathbb X$,
	\begin{align}\label{Le11-2}
		\lim_{\epsilon\to0}
		\sup_{(x,\mu)\in \mathcal C}\sup_{0\leq t\leq T}
		d_{\mathbb X}\big(
		\Phi^\epsilon(t,\omega)(x,\mu),
		\Phi^0(t,\omega)(x,\mu)
		\big)=0.
	\end{align}
\end{lemma}

\begin{proof}
	Let
	\begin{align*}
		v^\epsilon(t)
		=v^\epsilon_{\ell^\epsilon_{\mu_\epsilon}}
		\big(t,\omega;x_\epsilon-\epsilon z(\omega)\big)\quad\text{and}\quad
		v^0(t)
		=v^0_{\ell^0_{\mu_0}}(t,\omega;x_0),
	\end{align*}
	where
	$\ell^\epsilon_{\mu_\epsilon}(t)=\varphi_P^\epsilon(t)\mu_\epsilon$
	and $\ell^0_{\mu_0}(t)=\varphi_P^0(t)\mu_0$. By \eqref{Le4-1}, we have
	\begin{align}\label{Le11-2a}
		\varphi_H^\epsilon(t,\omega,x_\epsilon,\mu_\epsilon)
		=v^\epsilon(t)+\epsilon z(\theta_t\omega)\quad\text{and}\quad
		\varphi_H^0(t,\omega,x_0,\mu_0)
		=v^0(t).
	\end{align}
	
	\emph{Step 1. Uniform finite-time estimates.}
	Applying the estimate leading to \eqref{Le5-6} to the two prescribed law curves $\ell^\epsilon_{\mu_\epsilon}$ and $\ell^0_{\mu_0}$, and then adding the resulting inequalities, gives, for a.e.\ $t\in(0,T)$,
	\begin{align*}
		&\frac{d}{dt}
		\big(\|v^\epsilon(t)\|^2+\|v^0(t)\|^2\big)
		+c_1\big(\|\nabla v^\epsilon(t)\|^2
		+\|\nabla v^0(t)\|^2\big)
		\notag\\
		&\quad
		+c_2\big(\|v^\epsilon(t)\|_p^p
		+\|v^0(t)\|_p^p\big)
		+c_3\big(\|v^\epsilon(t)\|^2
		+\|v^0(t)\|^2\big)
		\notag\\
		&\qquad\leq
		L_G\left[
		M_2\big(\varphi_P^\epsilon(t)\mu_\epsilon\big)
		+M_2\big(\varphi_P^0(t)\mu_0\big)
		\right]
		+2C_H\mathfrak z(\theta_t\omega).
	\end{align*}
	By \eqref{Le5-2}, we have
	\begin{align*}
		M_2\big(\varphi_P^\epsilon(t)\mu_\epsilon\big)
		+M_2\big(\varphi_P^0(t)\mu_0\big)
		\leq
		e^{-\lambda_0t/2}
		\big(M_2(\mu_\epsilon)+M_2(\mu_0)\big)
		+2R_P^2
		\leq R+2R_P^2,
	\end{align*}
	and
	\begin{align*}
		\|v^\epsilon(0)\|^2+\|v^0(0)\|^2
		\leq
		2R+2\epsilon_0^2\|z(\omega)\|^2.
	\end{align*}
	Since $\int_0^T\mathfrak z(\theta_t\omega)dt<\infty$, integration
	yields
	\begin{align}\label{Le11-3}
		\sup_{0\leq t\leq T}
		\big(\|v^\epsilon(t)\|^2+\|v^0(t)\|^2\big)
		+
		\int_0^T
		\left[
		\|v^\epsilon(t)\|_V^2+\|v^0(t)\|_V^2
		+\|v^\epsilon(t)\|_p^p+\|v^0(t)\|_p^p
		\right]dt
		\leq C_{T,\omega,R}.
	\end{align}
	
	\emph{Step 2. Difference estimate.}
	Set $w(t)=v^\epsilon(t)-v^0(t)$. Subtracting the two transformed
	equations gives
	\begin{align}\label{Le11-4}
		&\frac{dw}{dt}-\Delta w+\lambda w
		+f\big(\cdot,v^\epsilon+\epsilon z(\theta_t\omega)\big)
		-f(\cdot,v^0)\nonumber\\
		&\qquad=
		\mathcal G\big(
		v^\epsilon+\epsilon z(\theta_t\omega),
		\varphi_P^\epsilon(t)\mu_\epsilon
		\big)-\mathcal G\big(v^0,\varphi_P^0(t)\mu_0\big)
		+\epsilon\Delta z(\theta_t\omega),
	\end{align}
	with $w(0)=x_\epsilon-x_0-\epsilon z(\omega)$. Insert
	$f(\cdot,v^0+\epsilon z(\theta_t\omega))$ in the $f$-difference
	and
	$\mathcal G(v^0+\epsilon z(\theta_t\omega),
	\varphi_P^0(t)\mu_0)$ in the $\mathcal G$-difference. Since the
	state difference in the first pair of terms is $w$, \eqref{f-3}
	and \eqref{G-3} give
	\begin{align*}
		&2\big(
		f\big(\cdot,v^\epsilon+\epsilon z(\theta_t\omega)\big)
		-f\big(\cdot,v^0+\epsilon z(\theta_t\omega)\big),w
		\big)
		\geq-2\alpha_4\|w\|^2,
		\\
		&2\left\langle
		\mathcal G\big(
		v^\epsilon+\epsilon z(\theta_t\omega),
		\varphi_P^\epsilon(t)\mu_\epsilon
		\big)
		-\mathcal G\big(
		v^0+\epsilon z(\theta_t\omega),
		\varphi_P^0(t)\mu_0
		\big),w
		\right\rangle_{\mathbb V^*,\mathbb V}
		\\
		&\qquad\leq
		\beta_4\|w\|^2
		+L_G\mathcal W_2^2\big(
		\varphi_P^\epsilon(t)\mu_\epsilon,
		\varphi_P^0(t)\mu_0
		\big).
	\end{align*}
	The intersection-space integration-by-parts formula used in Lemma~\ref{Le1}, or equivalently the corresponding Galerkin identity followed by passage to the limit, justifies testing \eqref{Le11-4} by $2w$; see \cite{Gyongy17spde}. Thus, for a.e.\ $t\in(0,T)$,
	\begin{align}\label{Le11-5}
		&\frac{d}{dt}\|w\|^2
		+2\|\nabla w\|^2+\bar\gamma\|w\|^2
		\leq
		L_G\mathcal W_2^2\big(
		\varphi_P^\epsilon(t)\mu_\epsilon,
		\varphi_P^0(t)\mu_0
		\big)
		+2\left|\big(
		f\big(\cdot,v^0+\epsilon z(\theta_t\omega)\big)
		-f(\cdot,v^0),w
		\big)\right|
		\notag\\
		&\qquad
		+2\left|\left\langle
		\mathcal G\big(
		v^0+\epsilon z(\theta_t\omega),
		\varphi_P^0(t)\mu_0
		\big)
		-\mathcal G\big(v^0,\varphi_P^0(t)\mu_0\big),w
		\right\rangle_{\mathbb V^*,\mathbb V}\right|
		+2\epsilon\left|\big(\Delta z(\theta_t\omega),w\big)\right|.
	\end{align}
	By \eqref{f-5}-\eqref{f-6},
	$|\partial_sf(x,s)|\leq(\alpha_4+c_f)(1+|s|^{p-2})$, while
	\eqref{G-6}
	gives the analogous bound for $\partial_sG$. Thus, for a.e.\ $x\in\mathbb R^n$, the integral form of the mean value theorem gives
	\begin{align*}
		f\big(x,v^0+\epsilon z(\theta_t\omega)\big)-f(x,v^0)
		&=\epsilon z(\theta_t\omega)
		\int_0^1\partial_sf\big(x,v^0+r\epsilon z(\theta_t\omega)\big)dr,
		\\
		G\big(x,v^0+\epsilon z(\theta_t\omega),\varphi_P^0(t)\mu_0\big)
		-G\big(x,v^0,\varphi_P^0(t)\mu_0\big)
		&=\epsilon z(\theta_t\omega)
		\int_0^1\partial_sG\big(x,v^0+r\epsilon z(\theta_t\omega),
		\varphi_P^0(t)\mu_0\big)dr.
	\end{align*}
	Since $|a+b|^{p-2}\leq c_p(|a|^{p-2}+|b|^{p-2})$ and
	$r\epsilon\leq\epsilon_0$ for $p>2$ (the case $p=2$ being direct),
	we obtain
	\begin{align*}
		&\left|f\big(x,v^0+\epsilon z(\theta_t\omega)\big)-f(x,v^0)\right|
		+\left|G\big(x,v^0+\epsilon z(\theta_t\omega),\varphi_P^0(t)\mu_0\big)
		-G\big(x,v^0,\varphi_P^0(t)\mu_0\big)\right|
		\\
		&\quad\leq C\epsilon|z(\theta_t\omega)|
		\int_0^1\left[1+|v^0+r\epsilon z(\theta_t\omega)|^{p-2}\right]dr
		\leq C\epsilon|z(\theta_t\omega)|
		\left(1+|v^0|^{p-2}+|z(\theta_t\omega)|^{p-2}\right),
	\end{align*}
	where $C$ depends only on the structural constants and $\epsilon_0$.
	For $p>2$, H\"older's and Young's inequalities give
	\begin{gather*}
		\int_{\mathbb R^n}|v^0|^{p-2}|z(\theta_t\omega)||w|dx
		\leq
		\|v^0\|_p^{p-2}\|z(\theta_t\omega)\|_p\|w\|_p
		\leq
		C\big(\|v^0\|_p^p+\|z(\theta_t\omega)\|_p^p+\|w\|_p^p\big),
		\\
		\int_{\mathbb R^n}|z(\theta_t\omega)|^{p-1}|w|dx
		\leq
		\|z(\theta_t\omega)\|_p^{p-1}\|w\|_p
		\leq
		C\big(\|z(\theta_t\omega)\|_p^p+\|w\|_p^p\big),
	\end{gather*}
	while $\int|z(\theta_t\omega)||w|dx$ is controlled by the
	$L^2$ norms. The case $p=2$ follows from the same $L^2$ estimate.
	Consequently,
	\begin{align*}
		&2\left|\big(
		f\big(\cdot,v^0+\epsilon z(\theta_t\omega)\big)
		-f(\cdot,v^0),w
		\big)\right|
		+2\left|\left\langle
		\mathcal G\big(
		v^0+\epsilon z(\theta_t\omega),
		\varphi_P^0(t)\mu_0
		\big)
		-\mathcal G\big(v^0,\varphi_P^0(t)\mu_0\big),w
		\right\rangle_{\mathbb V^*,\mathbb V}\right|
		\notag\\
		&\qquad\leq
		C\epsilon\left(
		\|w\|^2+\|w\|_p^p+\|v^0\|_p^p
		+\|z(\theta_t\omega)\|^2
		+\|z(\theta_t\omega)\|_p^p
		\right).
	\end{align*}
	Moreover,
	\begin{align*}
		\|w\|_p^p
		\leq
		2^{p-1}\big(\|v^\epsilon\|_p^p+\|v^0\|_p^p\big)\quad\text{and}\quad
		2\epsilon\left|\big(\Delta z(\theta_t\omega),w\big)\right|
		\leq
		\epsilon\|w\|^2
		+\epsilon\|\Delta z(\theta_t\omega)\|^2.
	\end{align*}
	Substituting these estimates into \eqref{Le11-5} and using
	\eqref{s4.1-1}, which controls the displayed norms of
	$z(\theta_t\omega)$, we obtain
	\begin{align}\label{Le11-6}
		\frac{d}{dt}\|w(t)\|^2
		\leq
		C\|w(t)\|^2
		+L_G\mathcal W_2^2\big(
		\varphi_P^\epsilon(t)\mu_\epsilon,
		\varphi_P^0(t)\mu_0
		\big)
		+C\epsilon\left(
		\|v^\epsilon(t)\|_p^p
		+\|v^0(t)\|_p^p
		+\mathfrak z(\theta_t\omega)
		\right),
	\end{align}
	where $C>0$ is independent of $\epsilon$ and the initial data.
	Gronwall's inequality gives
	\begin{align}\label{Le11-7}
		&\sup_{0\leq t\leq T}\|w(t)\|^2
		\leq
		C_T\Big[
		\|w(0)\|^2
		+\int_0^T
		\mathcal W_2^2\big(
		\varphi_P^\epsilon(t)\mu_\epsilon,
		\varphi_P^0(t)\mu_0
		\big)dt
		\notag\\
		&\qquad\qquad
		+\epsilon\int_0^T
		\left(
		\|v^\epsilon(t)\|_p^p
		+\|v^0(t)\|_p^p
		+\mathfrak z(\theta_t\omega)
		\right)dt
		\Big].
	\end{align}
	Lemma~\ref{Le10} implies
	\begin{align*}
		\int_0^T
		\mathcal W_2^2\big(
		\varphi_P^\epsilon(t)\mu_\epsilon,
		\varphi_P^0(t)\mu_0
		\big)dt
		\leq
		C_T^{\rm law}
		\Big(
		\mathcal W_2^2(\mu_\epsilon,\mu_0)
		+\epsilon^2\sum_{j=1}^{m}\|h_j\|^2
		\Big),
	\end{align*}
	where
	\begin{align*}
		C_T^{\rm law}
		=\int_0^T\big(e^{(L_G-\bar\gamma)t}
		+\mathfrak J_{\bar\gamma-L_G}(t)\big)dt<\infty.
	\end{align*}
	Since
	\begin{align*}
		\|w(0)\|^2
		\leq
		2\|x_\epsilon-x_0\|^2
		+2\epsilon^2\|z(\omega)\|^2,
	\end{align*}
	combining \eqref{Le11-3} and \eqref{Le11-7}, and using
	$\epsilon^2\leq\epsilon_0\epsilon$, yields
	\begin{align*}
		\sup_{0\leq t\leq T}\|w(t)\|^2
		\leq
		C_{T,\omega,R}
		\left(
		\|x_\epsilon-x_0\|^2
		+\mathcal W_2^2(\mu_\epsilon,\mu_0)
		+\epsilon
		\right).
	\end{align*}
	Finally, \eqref{Le11-2a} gives
	\begin{align*}
		\left\|
		\varphi_H^\epsilon(t,\omega,x_\epsilon,\mu_\epsilon)
		-\varphi_H^0(t,\omega,x_0,\mu_0)
		\right\|^2
		\leq
		2\|w(t)\|^2
		+2\epsilon^2\|z(\theta_t\omega)\|^2.
	\end{align*}
	The continuity of $t\mapsto z(\theta_t\omega)$ in $H$ on $[0,T]$
	therefore proves \eqref{Le11-1}.
	
	Let $\mathcal C\subset\mathbb X$ be compact. Since
	$M_2(\mu)^{1/2}=\mathcal W_2(\mu,\delta_{0_H})$, the function
	$(x,\mu)\mapsto\|x\|^2+M_2(\mu)$ is continuous and hence bounded
	on $\mathcal C$. Taking
	$(x_\epsilon,\mu_\epsilon)=(x_0,\mu_0)=(x,\mu)$ in
	\eqref{Le11-1}, and using \eqref{Le10-2} together with the product
	metric on $\mathbb X$, we obtain
	\begin{align*}
		\sup_{(x,\mu)\in \mathcal C}\sup_{0\leq t\leq T}
		d_{\mathbb X}\big(
		\Phi^\epsilon(t,\omega)(x,\mu),
		\Phi^0(t,\omega)(x,\mu)
		\big)
		\leq C_{T,\omega,\mathcal C}\big(\epsilon^{1/2}+\epsilon\big),
	\end{align*}
	which proves \eqref{Le11-2}.
\end{proof}

For common initial data, the preceding estimate is $O(\epsilon^{1/2})$ in the product metric on bounded time intervals. The linear $O(\epsilon)$ rate in Corollary~\ref{Cor4} for invariant laws and random equilibria is a separate consequence of \textnormal{(A5)}.

The abstract perturbation criterion also requires a common compact family containing the perturbed attractors; Lemma~\ref{Le9} provides it through uniform pullback asymptotic compactness.

\begin{lemma}\label{Le12}
	Assume \textnormal{(A1)-(A4)} and define
	\begin{align}\label{Le12-1}
		\mathcal C(\omega)
		=
		\overline{
			\bigcup_{\epsilon\in[0,\epsilon_0]}
			\mathcal A^\epsilon(\omega)
		}^{H\times P},
		\qquad \omega\in\Omega_0.
	\end{align}
	Then $\mathcal C(\omega)$ is compact in $H\times P$ for every $\omega\in\Omega_0$, and the family $\mathcal C=\{\mathcal C(\omega)\}_{\omega\in\Omega_0}$ belongs to $\mathfrak D$.
\end{lemma}

\begin{proof}
	Fix $\omega\in\Omega_0$ and take an arbitrary sequence
	\begin{align*}
		a_n\in\mathcal A^{\epsilon_n}(\omega),
		\qquad \epsilon_n\in[0,\epsilon_0].
	\end{align*}
	By strict invariance, for $t_n=n$ there exists
	$b_n\in\mathcal A^{\epsilon_n}(\theta_{-t_n}\omega)$ such that
	\begin{align*}
		a_n
		=\Phi^{\epsilon_n}(t_n,\theta_{-t_n}\omega)b_n.
	\end{align*}
	Theorem~\ref{The2} gives
	\begin{align*}
		b_n\in K(\theta_{-t_n}\omega).
	\end{align*}
	Since $K\in\mathfrak D$, Lemma~\ref{Le9}, applied with $D=K$, yields a convergent subsequence of $\{a_n\}$. Hence the union in \eqref{Le12-1} is relatively compact in $H\times P$, and its closure $\mathcal C(\omega)$ is compact.
	
	For every $\epsilon\in[0,\epsilon_0]$, Theorem~\ref{The2} also gives
	$\mathcal A^\epsilon(\omega)\subset K(\omega)$. Since $K(\omega)$ is closed,
	\begin{align*}
		\mathcal C(\omega)\subset K(\omega).
	\end{align*}
	The universe $\mathfrak D$ is inclusion closed and $K\in\mathfrak D$; therefore $\mathcal C\in\mathfrak D$.
\end{proof}

Lemmas~\ref{Le10}-\ref{Le12} now yield zero-noise upper semicontinuity of the attractors.

\begin{theorem}\label{The3}
	Assume \textnormal{(A1)-(A4)}. Then the pullback random attractors obtained in Theorem~\ref{The2} are upper semicontinuous at $\epsilon=0$: for every $\omega\in\Omega_0$,
	\begin{align}\label{The3-1}
		\lim_{\epsilon\to0}
		\operatorname{dist}_{\mathbb X}\big(
		\mathcal A^\epsilon(\omega),
		\mathcal A^0(\omega)
		\big)=0.
	\end{align}
\end{theorem}

\begin{proof}
	Condition \textnormal{(U1)} of Theorem~\ref{theorem3} follows from Theorem~\ref{The2}. Lemma~\ref{Le12} verifies \textnormal{(U2)} with $\epsilon_1=\epsilon_0$, and \eqref{Le11-2} is precisely \textnormal{(U3)}. The conclusion therefore follows from Theorem~\ref{theorem3}.
\end{proof}

\begin{corollary}\label{Cor4}
	Assume \textnormal{(A1)-(A5)}, and let $\mu_\epsilon^*$ and $\xi_\epsilon$ be given by Corollary~\ref{Cor2}. Then there exists a unique $\bar u\in H$ such that the constant trajectory $u(t)\equiv\bar u$ is a self-consistent stationary variational solution of \eqref{Eq3} with $\epsilon=0$. If $\delta_{\bar u}$ denotes the Dirac measure at $\bar u$, then, for every $\omega\in\Omega_0$,
	\begin{align}\label{Cor4-1}
		\mathcal A^0(\omega)
		=\big\{(\bar u,\delta_{\bar u})\big\},
		\qquad
		\xi_0(\omega)=\bar u,
		\qquad
		\mu_0^*=\delta_{\bar u}.
	\end{align}
	Moreover,
	\begin{align}\label{Cor4-2}
		\lim_{\epsilon\to0}
		\|\xi_\epsilon(\omega)-\bar u\|=0,
		\qquad \omega\in\Omega_0,
	\end{align}
	and, for every $\epsilon\in[0,\epsilon_0]$,
	\begin{align}\label{Cor4-3}
		\left(
		\mathbb E\|\xi_\epsilon-\bar u\|^2
		\right)^{1/2}
		=
		\mathcal W_2(\mu_\epsilon^*,\delta_{\bar u})
		\leq
		\epsilon
		\Big(
		\frac1{\kappa_0}
		\sum_{j=1}^{m}\|h_j\|^2
		\Big)^{1/2}.
	\end{align}
\end{corollary}

\begin{proof}
	At $\epsilon=0$, the transformed equation \eqref{Eq4-trans} contains no random term. Hence the invariant-law state cocycle is independent of $\omega$; write
	\begin{align*}
		S(t)x=\Psi^0(t,\omega)x,
		\qquad t\geq0.
	\end{align*}
	The cocycle identity shows that $S$ is a continuous semigroup on $H$, and \eqref{Le4-2} gives
	\begin{align*}
		\|S(t)x-S(t)y\|^2
		\leq e^{-\bar\gamma t}\|x-y\|^2.
	\end{align*}
	Thus $S(1)$ is a strict contraction on the complete space $H$ and has a unique fixed point $\bar u$. For every $s\geq0$, the semigroup identity shows that $S(s)\bar u$ is also fixed by $S(1)$; hence
	\begin{align*}
		S(s)\bar u=\bar u,
		\qquad s\geq0.
	\end{align*}
	Furthermore, for $B\in\mathfrak D_H$,
	\begin{align*}
		\operatorname{dist}_H\big(
		S(t)B(\theta_{-t}\omega),\{\bar u\}
		\big)
		\leq
		e^{-\bar\gamma t/2}
		\big(
		\|B(\theta_{-t}\omega)\|_H+\|\bar u\|
		\big)
		\longrightarrow0.
	\end{align*}
	Therefore $\{\bar u\}$ is the $\mathfrak D_H$-pullback attractor of $S$. By the uniqueness asserted in Corollary~\ref{Cor2},
	\begin{align*}
		\mathcal A_H^0(\omega)=\{\bar u\},
		\qquad \omega\in\Omega_0.
	\end{align*}
	The law identity in \eqref{Cor2-3} now gives $\mu_0^*=\delta_{\bar u}$, and the product representation \eqref{Cor2-2} proves \eqref{Cor4-1}. Since the frozen law $\mu_0^*$ equals the law of the constant state $\bar u$, this stationary solution is self-consistent. If $\widetilde u$ is any self-consistent stationary variational solution of the zero-noise equation, its one-time law is invariant under $\varphi_P^0$ and hence equals $\mu_0^*=\delta_{\bar u}$ by Corollary~\ref{Cor1}. Thus $\widetilde u(t)=\bar u$ a.s.\ for every rational $t\geq0$; path continuity then shows that $\widetilde u$ is indistinguishable from the constant trajectory, which proves uniqueness.
	
	Using \eqref{Cor2-2} and \eqref{Cor4-1}, we have
	\begin{align*}
		\operatorname{dist}_{\mathbb X}\big(
		\mathcal A^\epsilon(\omega),
		\mathcal A^0(\omega)
		\big)
		=
		\|\xi_\epsilon(\omega)-\bar u\|
		+\mathcal W_2(\mu_\epsilon^*,\delta_{\bar u}).
	\end{align*}
	Theorem~\ref{The3} therefore yields \eqref{Cor4-2}. Finally, the second identity in \eqref{Cor2-3} and the uniqueness of a coupling with a Dirac measure give
	\begin{align*}
		\mathbb E\|\xi_\epsilon-\bar u\|^2
		=
		\int_H\|x-\bar u\|^2\,\mu_\epsilon^*(dx)
		=
		\mathcal W_2^2(\mu_\epsilon^*,\delta_{\bar u}).
	\end{align*}
	Estimate \eqref{Cor3-1}, together with $\mu_0^*=\delta_{\bar u}$, proves \eqref{Cor4-3}.
\end{proof}


\begin{thebibliography}{99}
	
	\bibitem{Ahmed95spa}
	N.U.Ahmed and X.Ding,
	A semilinear McKean-Vlasov stochastic evolution equation in Hilbert space.
	Stochastic Process. Appl. 60 (1995), no. 1, 65-85.
	\href{https://doi.org/10.1016/0304-4149(95)00050-X}{doi}.
	
	\bibitem{Ambrosio08lmzb}
	L.Ambrosio, N.Gigli and G.Savar\'e,
	\emph{Gradient Flows in Metric Spaces and in the Space of Probability Measures}.
	Second edition, Lectures in Mathematics ETH Z\"urich, Birkh\"auser, Basel, 2008.
	\href{https://doi.org/10.1007/978-3-7643-8722-8}{doi}.
	
	\bibitem{Arendt18sm}
	W.Arendt and M.Kreuter,
	Mapping theorems for Sobolev spaces of vector-valued functions.
	Studia Math. 240 (2018), no. 3, 275-299.
	\href{https://doi.org/10.4064/sm8757-4-2017}{doi}.
	
	\bibitem{Arnold98}
	L.Arnold,
	\emph{Random Dynamical Systems}.
	Springer Monographs in Mathematics, Springer, Berlin, 1998.
	\href{https://doi.org/10.1007/978-3-662-12878-7}{doi}.
	
	\bibitem{Bates09jde}
	P.W.Bates, K.Lu and B.Wang,
	Random attractors for stochastic reaction-diffusion equations on unbounded domains.
	J. Differential Equations 246 (2009), no. 2, 845-869.
	\href{https://doi.org/10.1016/j.jde.2008.05.017}{doi}.
	
	\bibitem{Brzez93ptrf}
	Z.Brze\'zniak, M.Capi\'nski and F.Flandoli,
	Pathwise global attractors for stationary random dynamical systems.
	Probab. Theory Related Fields 95 (1993), no. 1, 87-102.
	\href{https://doi.org/10.1007/BF01197339}{doi}.
	
	\bibitem{Brzez06tams}
	Z.Brze\'zniak and Y.Li,
	Asymptotic compactness and absorbing sets for 2D stochastic Navier-Stokes equations on some unbounded domains.
	Trans. Amer. Math. Soc. 358 (2006), no. 12, 5587-5629.
	\href{https://doi.org/10.1090/S0002-9947-06-03923-7}{doi}.
	
	\bibitem{Buckdahn17aop}
	R.Buckdahn, J.Li, S.Peng and C.Rainer,
	Mean-field stochastic differential equations and associated PDEs.
	Ann. Probab. 45 (2017), no. 2, 824-878.
	\href{https://doi.org/10.1214/15-AOP1076}{doi}.
	
	\bibitem{Caraballo98cpde}
	T.Caraballo, J.A.Langa and J.C.Robinson,
	Upper semicontinuity of attractors for small random perturbations of dynamical systems.
	Comm. Partial Differential Equations 23 (1998), no. 9-10, 1557-1581.
	\href{https://doi.org/10.1080/03605309808821394}{doi}.
	
	\bibitem{Carrillo20arma}
	J.A.Carrillo, R.S.Gvalani, G.A.Pavliotis and A.Schlichting,
	Long-time behaviour and phase transitions for the McKean-Vlasov equation on the torus.
	Arch. Ration. Mech. Anal. 235 (2020), 635-690.
	\href{https://doi.org/10.1007/s00205-019-01430-4}{doi}.
	
	\bibitem{Chaudru22jmpa}
	P.-E.Chaudru de Raynal and N.Frikha,
	Well-posedness for some non-linear SDEs and related PDE on the Wasserstein space.
	J. Math. Pures Appl. 159 (2022), 1-167.
	\href{https://doi.org/10.1016/j.matpur.2021.12.001}{doi}.
	
	\bibitem{Chen26avg}
	H.Chen, M.Cheng and Z.Liu,
	Averaging principle and pullback attractor convergence for McKean-Vlasov stochastic reaction-diffusion equations.
	arXiv:2608.09319, 2026.
	\href{https://arxiv.org/abs/2608.09319}{link}.
	
	\bibitem{Chen26amo}
	Z.Chen and B.Wang,
	Well-posedness and large deviations of fractional McKean-Vlasov stochastic reaction-diffusion equations on unbounded domains.
	Appl. Math. Optim. 94 (2026), Paper No. 40.
	\href{https://doi.org/10.1007/s00245-026-10426-y}{doi}.
	
	\bibitem{Chen25spde}
	Z.Chen and B.Wang,
	Long-term dynamics of fractional stochastic delay reaction-diffusion equations on unbounded domains.
	Stoch. Partial Differ. Equ. Anal. Comput. 13 (2025), 180-242.
	\href{https://doi.org/10.1007/s40072-024-00334-z}{doi}.
	
	\bibitem{Cheng25}
	M.Cheng, X.Cheng and Z.Liu,
	Random attractors for McKean-Vlasov SDEs.
	arXiv:2511.16190, 2025.
	\href{https://arxiv.org/abs/2511.16190}{link}.
	
	\bibitem{Coghi20aap}
	M.~Coghi, J.-D.~Deuschel, P.~K.~Friz and M.~Maurelli,
	Pathwise McKean-Vlasov theory with additive noise.
	Ann. Appl. Probab. 30 (2020), no. 5, 2355-2392.
	\href{https://doi.org/10.1214/20-AAP1560}{doi}.
	
	\bibitem{Crauel94}
	H.Crauel and F.Flandoli,
	Attractors for random dynamical systems.
	Probab. Theory Related Fields 100 (1994), no. 3, 365-393.
	\href{https://doi.org/10.1007/BF01193705}{doi}.
	
	\bibitem{Crauel97jdde}
	H.Crauel, A.Debussche and F.Flandoli,
	Random attractors.
	J. Dynam. Differential Equations 9 (1997), no. 2, 307-341.
	\href{https://doi.org/10.1007/BF02219225}{doi}.
	
	\bibitem{Eberle19tams}
	A.Eberle, A.Guillin and R.Zimmer,
	Quantitative Harris-type theorems for diffusions and McKean-Vlasov processes.
	Trans. Amer. Math. Soc. 371 (2019), no. 10, 7135-7173.
	\href{https://doi.org/10.1090/tran/7576}{doi}.
	
	\bibitem{Evans10}
	L.C.Evans,
	\emph{Partial Differential Equations}.
	Second edition, Graduate Studies in Mathematics, vol. 19,
	American Mathematical Society, Providence, RI, 2010.
	\href{https://bookstore.ams.org/gsm-19-r}{link}.
	
	\bibitem{Gess25}
	B.Gess, R.S.Gvalani and S.Hu,
	Random dynamical systems for McKean-Vlasov SDEs via rough path theory.
	arXiv:2507.02449, 2025.
	\href{https://arxiv.org/abs/2507.02449}{link}.
	
	\bibitem{Gess11jde}
	B.Gess, W.Liu and M.R{\"o}ckner,
	Random attractors for a class of stochastic partial differential equations driven by general additive noise.
	J. Differential Equations 251 (2011), no. 4-5, 1225-1253.
	\href{https://doi.org/10.1016/j.jde.2011.02.013}{doi}.
	
	\bibitem{Gess20jde}
	B.Gess, W.Liu and A.Schenke,
	Random attractors for locally monotone stochastic partial differential equations.
	J. Differential Equations 269 (2020), no. 4, 3414-3455.
	\href{https://doi.org/10.1016/j.jde.2020.03.002}{doi}.
	
	\bibitem{Gyongy17spde}
	I.Gy\"ongy and D.\v{S}i\v{s}ka,
	It\^o formula for processes taking values in intersection of finitely many Banach spaces.
	Stoch. Partial Differ. Equ. Anal. Comput. 5 (2017), no. 3, 428-455.
	\href{https://doi.org/10.1007/s40072-017-0093-6}{doi}.
	
	\bibitem{Hong24aap}
	W.Hong, S.Hu and W.Liu,
	McKean-Vlasov SDE and SPDE with locally monotone coefficients.
	Ann. Appl. Probab. 34 (2024), no. 2, 2136-2189.
	\href{https://doi.org/10.1214/23-AAP2016}{doi}.
	
	\bibitem{Kotelenez10ptrf}
	P.M.Kotelenez and T.G.Kurtz,
	Macroscopic limits for stochastic partial differential equations of McKean-Vlasov type.
	Probab. Theory Related Fields 146 (2010), no. 1-2, 189-222.
	\href{https://doi.org/10.1007/s00440-008-0188-0}{doi}.
	
	\bibitem{Lions69}
	J.-L.Lions,
	\emph{Quelques m\'ethodes de r\'esolution des probl\`emes aux limites non lin\'eaires}.
	Dunod, Gauthier-Villars, Paris, 1969.
	
	\bibitem{Liu10jfa}
	W.Liu and M.R{\"o}ckner,
	SPDE in Hilbert space with locally monotone coefficients.
	J. Funct. Anal. 259 (2010), no. 11, 2902-2922.
	\href{https://doi.org/10.1016/j.jfa.2010.05.012}{doi}.
	
	\bibitem{McKean66pnas}
	H.P.McKean, Jr.,
	A class of Markov processes associated with nonlinear parabolic equations.
	Proc. Natl. Acad. Sci. USA 56 (1966), no. 6, 1907-1911.
	\href{https://doi.org/10.1073/pnas.56.6.1907}{doi}.
	
	\bibitem{Roeckner08cmp}
	M.R{\"o}ckner, B.Schmuland and X.Zhang,
	Yamada-Watanabe theorem for stochastic evolution equations in infinite dimensions.
	Condens. Matter Phys. 11 (2008), no. 2, 247-259.
	\href{https://doi.org/10.5488/CMP.11.2.247}{doi}.
	
	\bibitem{Shi24jde}
	L.Shi, J.Shen and B.Wang,
	Long-time dynamics of McKean-Vlasov stochastic reaction-diffusion equations on $\mathbb R^n$.
	J. Differential Equations 453 (2026), 113855.
	\href{https://doi.org/10.1016/j.jde.2025.113855}{doi}.
	
	\bibitem{Sznitman91lnm}
	A.-S.Sznitman,
	Topics in propagation of chaos.
	In: \emph{\'Ecole d'\'Et\'e de Probabilit\'es de Saint-Flour XIX-1989},
	Lecture Notes in Mathematics, vol. 1464, Springer, Berlin, 1991, 165-251.
	\href{https://doi.org/10.1007/BFb0085169}{doi}.
	
	\bibitem{Tang16sd}
	B.Q.Tang,
	Regularity of random attractors for stochastic reaction-diffusion equations on unbounded domains.
	Stoch. Dyn. 16 (2016), no. 1, 1650006.
	\href{https://doi.org/10.1142/S0219493716500064}{doi}.
	
	\bibitem{Villani09}
	C.Villani,
	\emph{Optimal Transport: Old and New}.
	Grundlehren der mathematischen Wissenschaften, vol. 338, Springer, Berlin, 2009.
	\href{https://doi.org/10.1007/978-3-540-71050-9}{doi}.
	
	\bibitem{Wang99pd}
	B.Wang,
	Attractors for reaction-diffusion equations in unbounded domains.
	Phys. D 128 (1999), no. 1, 41-52.
	\href{https://doi.org/10.1016/S0167-2789(98)00304-2}{doi}.
	
	\bibitem{Wang19jde}
	B.Wang,
	Dynamics of fractional stochastic reaction-diffusion equations on unbounded domains driven by nonlinear noise.
	J. Differential Equations 268 (2019), no. 1, 1-59.
	\href{https://doi.org/10.1016/j.jde.2019.08.007}{doi}.
	
	\bibitem{Wang09jde}
	B.Wang,
	Upper semicontinuity of random attractors for non-compact random dynamical systems.
	Electron. J. Differential Equations 2009 (2009), no. 139, 1-18.
	\href{https://ejde.math.txstate.edu/Volumes/2009/139/wang.pdf}{link}.
	
	\bibitem{Wang12jde}
	B.Wang,
	Sufficient and necessary criteria for existence of pullback attractors for non-compact random dynamical systems.
	J. Differential Equations 253 (2012), no. 5, 1544-1583.
	\href{https://doi.org/10.1016/j.jde.2012.05.015}{doi}.
	
	\bibitem{Wang18jde}
	X.Wang, K.Lu and B.Wang,
	Wong-Zakai approximations and attractors for stochastic reaction-diffusion equations on unbounded domains.
	J. Differential Equations 264 (2018), no. 1, 378-424.
	\href{https://doi.org/10.1016/j.jde.2017.09.006}{doi}.
	
	\bibitem{Zeng24}
	T.Zeng, R.Li and D.Li,
	Uniform measure attractors of McKean-Vlasov stochastic reaction-diffusion equations on unbounded thin domain.
	Math. Methods Appl. Sci. 49 (2026), no. 10, 11105-11137.
	\href{https://doi.org/10.1002/mma.70643}{doi}.
\end{thebibliography}
\end{document}